\documentclass[reqno]{amsart}
\usepackage{amscd}
\usepackage{amsfonts}
\usepackage{amsmath} 
\usepackage{amssymb}
\usepackage{amsthm}
\usepackage{bm}
\usepackage{bbm}         
\usepackage{color}
\usepackage{comment}
\usepackage{dsfont}
\usepackage{enumerate}
\usepackage{fancyhdr}
\usepackage{graphicx}
\usepackage[colorlinks=true, linkcolor=blue, citecolor=blue, urlcolor=blue, breaklinks=true]{hyperref}
\usepackage{latexsym}
\usepackage{mathtools}
\usepackage{pifont}
\usepackage{quiver}
\usepackage{scalerel}
\usepackage{setspace}
\usepackage{stmaryrd}
\usepackage{tabularx}
\usepackage{tikz}
    \usetikzlibrary{decorations.markings}
    \usetikzlibrary{cd}
    \usetikzlibrary{patterns}
    \tikzset{anchorbase/.style={baseline={([yshift=-0.5ex]current bounding box.center)}}}
    \tikzset{wipe/.style={white,line width=4pt}}
    \tikzset{anchorbase/.style={baseline={([yshift=-0.5ex]current bounding box.center)}}}
    \tikzset{every picture/.style={line width=0.5pt}} 
\usepackage{verbatim}
\usepackage{wasysym}
\usepackage[all]{xy}
\usepackage[vcentermath]{youngtab}

\theoremstyle{plain}
\newtheorem{lemma}{Lemma}[section]
\newtheorem{proposition}[lemma]{Proposition}

\newtheorem{theorem}[lemma]{Theorem}

\theoremstyle{definition}
\newtheorem{definition}[lemma]{Definition}

\theoremstyle{remark}
\newtheorem{remark}[lemma]{Remark}

\definecolor{lightorange}{RGB}{220, 162, 55}
\definecolor{lightblue}{RGB}{111,178,228}
\definecolor{darkblue}{RGB}{48,112,173}
\definecolor{green}{RGB}{70,155,118}
\definecolor{yellow}{RGB}{238,228,98}
\def\nset#1{\{1,\dots,#1\}}
\def\bi{\text{\boldmath$i$}}
\def\bj{\text{\boldmath$j$}}
\def\bk{\text{\boldmath$k$}}
\def\Stab{\operatorname{Stab}}
\def\Schur{\mathcal{S}chur}
\def\HSchur{H\mathcal{S}chur}
\def\Web{\mathcal{W}eb}
\def\HWeb{H\mathcal{W}eb}

\newcommand{\arxiv}[1]{{\tt arXiv:#1}}
\DeclareMathOperator{\Ind}{Ind}
\DeclareMathOperator{\Mat}{Mat}
\DeclareMathOperator{\HMat}{HMat}

\DeclareMathOperator{\End}{End}
\DeclareMathOperator{\Hom}{Hom}
\DeclareMathOperator{\im}{Im}
\DeclareMathOperator{\diag}{diag}
\DeclareMathOperator{\triv}{triv}

\makeatletter
\title{Hyperoctahedral Schur Category and Hyperoctahedral Web Category}
\author{Razzi Masroor}
\begin{document}
\maketitle
\begin{abstract}
We extend the Schur algebra and the polynomial web category of the symmetric group to the hyperoctahedral group. In particular, we define the hyperoctahedral web category diagrammatically by generators and relations, and prove that it is equivalent to the hyperoctahedral Schur category. 
\end{abstract}  
\section{Introduction}

Diagrammatic presentations are useful for simplifying and better understanding the multiplicative structure of algebraic objects. A classic example is the symmetric group, which admits a presentation via string diagrams. Strings go straight from a bottom row of vertices to a top row of vertices in order to encode a permutation. Below is an example in the symmetric group of degree $4$ when $\rho = (12)$,
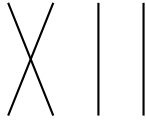
\begin{center}
\begin{tikzpicture}[scale=1,thick,anchorbase]
\begin{scope}[xshift=4cm]
\draw  (2.8,-0.5) -- +(.6,1.5);
\draw  (3.4,-0.5) -- +(-.6,1.5);
\draw  (4.0,-0.5) -- +(0,1.5);
\draw  (4.6,-0.5) -- +(0,1.5);
\end{scope}.
\end{tikzpicture}.
\end{center}
Permutations are composed by putting diagrams on top of one another. In particular, while the order of $S_n$ is finite, there are an infinite number of diagrams (for $n\geq 2$) which are visually distinct. For instance, with
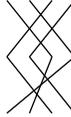
\begin{center}
    \begin{tikzpicture}[anchorbase, scale = 0.5]
\draw  (2.2,-0.5) -- +(1.8,1.5) ;
\draw  (2.8,1) -- +(1.2,1.5) ;
\draw  (2.8,-0.5) -- +(0.6,1.5);
\draw  (2.2,1) -- +(1.2,1.5) ;
\draw  (3.4,-0.5) -- +(-1.2,1.5);
\draw  (4,-0.5) -- +(-1.2,1.5);
\draw  (4,1) -- +(-1.2,1.5);
\draw  (3.4,1) -- +(-1.2,1.5);
    \end{tikzpicture}
\end{center}
the strings between vertices are more complex than a single straight edge, yet the strings in this diagram start and end in the same places as the previous diagram. There are three relations which, when observed, allow these two diagrams to be equal:
\begin{equation}\label{E:3relns}
 \begin{tikzpicture}[anchorbase,scale=0.5,thick]
\draw  (2.2,-0.5) -- +(0.6,1.5) ;
\draw  (2.8,1) -- +(-0.6,1.5) ;
\draw  (2.8,-0.5) -- +(-0.6,1.5);
\draw  (2.2,1) -- +(0.6,1.5) ;
\end{tikzpicture} = \begin{tikzpicture}[anchorbase,scale=1,thick]
\draw  (2.2,-0.5) -- +(0,1.5) ;
\draw  (2.8,-0.5) -- +(0,1.5) ;
\end{tikzpicture},
 \qquad \qquad
    \begin{tikzpicture}[anchorbase,scale=0.33,thick]
    \draw  (2.2,-0.5) -- +(0,1.5);
\draw  (2.8,-0.5) -- +(0.6,1.5);
\draw  (3.4,-0.5) -- +(-0.6,1.5);
 \draw  (2.2,1) -- +(0.6,1.5);
\draw  (2.8,1) -- +(-0.6,1.5);
\draw  (3.4,1) -- +(0,1.5);
 \draw  (2.2,-2) -- +(0.6,1.5);
\draw  (2.8,-2) -- +(-0.6,1.5);
\draw  (3.4,-2) -- +(0,1.5);
\end{tikzpicture}
= \begin{tikzpicture}[anchorbase,scale=0.33,thick]
    \draw  (2.2,-0.5) -- +(0.6,1.5);
\draw  (2.8,-0.5) -- +(-0.6,1.5);
\draw  (3.4,-0.5) -- +(0,1.5);
 \draw  (2.2,1) -- +(0,1.5);
\draw  (2.8,1) -- +(0.6,1.5);
\draw  (3.4,1) -- +(-0.6,1.5);
 \draw  (2.2,-2) -- +(0,1.5);
\draw  (2.8,-2) -- +(0.6,1.5);
\draw  (3.4,-2) -- +(-0.6,1.5);
\end{tikzpicture},
 \qquad \qquad
    \begin{tikzpicture}
    [anchorbase,scale=.5,thick]
\draw  (2.8,-0.5) -- +(0,1.5);
\draw  (3.4,-0.5) -- +(0.6,1.5);
\draw  (2.2,-0.5) -- +(0,1.5);
\draw  (4,-0.5) -- +(-0.6,1.5);
\draw  (2.8,1) -- +(-.6,1.5);
\draw  (3.4,1) -- +(0,1.5);
\draw  (2.2,1) -- +(0.6,1.5);
\draw  (4,1) -- +(0,1.5);
\end{tikzpicture} = \begin{tikzpicture}
    [anchorbase,scale=.5,thick]
\draw  (2.2,-0.5) -- +(0.6,1.5);
\draw  (2.8,-0.5) -- +(-0.6,1.5);
\draw  (3.4,-0.5) -- +(0,1.5);
\draw  (4,-0.5) -- +(0,1.5);
\draw  (2.2,1) -- +(0,1.5);
\draw  (2.8,1) -- +(0,1.5);
\draw  (3.4,1) -- +(0.6,1.5);
\draw  (4,1) -- +(-0.6,1.5);
\end{tikzpicture}.
\end{equation}
Importantly, these relations are sufficient to prove that any two diagrams representing the same permutation are equal, such as our two examples. One should compare the relations in Equation~\ref{E:3relns} with the Coxeter presentation:
\begin{equation}\label{Coxeter}S_{n}\cong \bigg\langle\bigg\langle s_1, s_2, \dots, s_{n-1} \bigg| \substack{s_i^2=1 \\ s_{i}s_{i+ 1}s_{i} = s_{i+1}s_is_{i+1} \\ s_is_j=s_js_i, |i-j|>1}\bigg\rangle \bigg\rangle.\end{equation}

In this example, the diagrammatic description of the symmetric group motivates the Coxeter presentation. Furthermore, one can in fact prove the generators and relations presentation of $S_n$ as in Equation~\ref{Coxeter} using string diagrams. The proof follows from establishing that, using only the relations in Equation~\ref{E:3relns}, an arbitrary diagram can be rewritten so that no two strings cross more than once, and that two reduced diagrams representing the same permutation are equal.

A classical result of Iwahori~\cite{Iwahori} is that the endomorphism algebra 
\[
\End_{GL_n(\mathbb{F}_q)}(\mathbb{C}[Flag_n(\mathbb{F}_q)])
\]
is isomorphic to the generators and relations algebra
\[
\big\langle\big\langle T_1, \dots, T_{n-1} \ | \ \substack{T_i^2= (q-1)T_i + q, \\ T_iT_{i\pm 1}T_i = T_{i\pm 1}T_iT_{i\pm 1}, \\ T_iT_j=T_jT_i, \quad \text{for $|i-j| > 1$}}\big\rangle \big\rangle.
\] 
A similar result was established by~\cite{Brundan} on the Schur algebra, another endomorphism algebra of a permutation module. Graphical techniques were applied to the Schur algebra by depicting its homomorphisms through pictures called web diagrams, see~\cite[Definition 4.7]{Brundan}. These diagrams are generated by merges, crossings, and splits involving strings of varying thickness. By creating a category of diagrams $\Web$, they were able to prove in~\cite[Theorem 4.10]{Brundan} that it is isomorphic to the category of homomorphisms from the Schur algebra, denoted as $\Schur$.

In this paper, we take a similar approach to~\cite{Brundan} by working with a variant of the Schur algbera, which we call the hyperoctahedral Schur algebra, based on the hyperoctahedral group. Formally, the \emph{hyperoctahedral group of degree $r$} is $H_r = \langle s_1, \ldots, s_r \rangle$ with relations
\begin{itemize}
  \item $s_i^2 = 1$;
  \item $s_is_j = s_js_i$ if $|i-j|>1$;
  \item $s_{i}s_js_{i} = s_{j}s_{i}s_{j}$ if $|i-j|=1$ and $i,j < r$;
  \item $s_{r-1}s_{r}s_{r-1}s_{r} = s_{r}s_{r-1}s_{r}s_{r-1}$.
\end{itemize}

As with~\cite{Brundan}, we will write our results on algebras in terms of two categories $\HSchur$ and $\HWeb$ for convenience. For $\HWeb$, it is defined through generating merges, splits, and crossings with relations. As for $\HSchur$, observe that given an index set $I$, a group $G$, and $M_i$ a vector space with a $G$ action for all $i\in I$, we have the following isomorphism for an endomorphism algebra:
\[\End_{G}(\bigoplus_{i\in I}M_{i})\cong \bigoplus_{i, j\in I} \Hom_{G}(M_i, M_j).\]
To turn this into a category, we take $I$ as the set of objects and $\Hom_{G}(M_i, M_j)$ as the morphisms from $i$ to $j$. This leads to our main theorem which proves a diagrammatic presentation of the hyperoctahedral Schur algebra:  
\bigskip
\\*
\noindent
{\bf Main Theorem.} {\em
    There is a an equivalence of categories $\Phi: \HWeb \rightarrow \HSchur$.
}

\subsection{Relation to previous work}

We are obliged to comment on past work which defines a hyperoctahedral Schur algebra. For instance, in~\cite{Green}, Green focuses on permutation modules (denoted as $V_\lambda$ in~\cite[Lemma 4.1.3]{Green}) induced from subgroups whose generators are some subset of those of $H_r$ (\cite[Definition 4.1.1]{Green}). However, Green excludes the generator $s_r$ from these subgroups, thereby creating a subalgebra of our hyperoctahedral Schur algebra.

Additionally, in~\cite[Example 1.4]{Du}, while they define a generalized notion of the Schur algebra for any Coxeter group (such as the hyperoctahedral group) and establish a basis, our work in proving a generators and relations presentation is novel.

Finally, we like to mention~\cite[Definition 3.1.1]{DKMZ} and the similarities between their generating morphisms and an alternative presentation of our hyperoctahedral web category. For more details, see Remark~\ref{DKMZPresentation}.

\subsection{Organization of the paper}
We will first introduce the necessary prerequisites, such as the symmetric and hyperoctahedral groups in Subsection~\ref{sec:symhyper} and the Schur algebra in Subsection~\ref{sec:schur}. We then recall web diagrams and their connection to the Schur algebra in Subsection~\ref{sec:web}. Afterward, we will define the hyperoctahedral Schur algebra and related terminology in Section~\ref{sec:defsforhyper} before working towards our main theorem in Section~\ref{sec:hweb+hschur}, which is that the hyperoctahedral web category is equivalent to the hyperoctahedral Schur category.
\begin{remark}
We will be using the ground field $\mathbb{F}$ throughout the rest of the paper.
\end{remark}
\subsection{Acknowledgements}
I would like to thank my PRIMES mentor Elijah Bodish for his guidance, assistance, and review throughout this research. Additionally, I want to thank Dr. Felix Gotti for his helpful comments during the editing of this paper. Furthermore, I am thankful to Daniel Tubbenhauer and Huanchen Bao for discussing the mathematics of this paper.

Finally, I want to say thank you to Prof. Pavel
Etingof, Dr. Slava Gerovitch, Dr. Tanya Khovanova, and all other organizers of the PRIMES program for allowing me to work on math research.
\section{Preliminary Definitions, Theorems, and Examples for the Schur Algebra}
\subsection{Symmetric and Hyperoctahedral Groups}\label{sec:symhyper}\begin{definition}
The \emph{symmetric group of degree $n$}, denoted by $S_n$, is the set of permutations of $\{1,2,\ldots, n\}$ with function composition as the group operation.
\end{definition}
This group can be generated by transpositions of adjacent elements; we define $\sigma_i$ for $i= 1$ to $n-1$ as the transposition which swaps $i$ and $i+1$ while preserving all other elements. These transpositions have relations
\begin{itemize}
  \item $\sigma_i^2 = 1$;
  \item $\sigma_i\sigma_j = \sigma_j\sigma_i$ for $|i-j|\geq 2$;
  \item $\sigma_{i-1}\sigma_i\sigma_{i-1} = \sigma_{i}\sigma_{i-1}\sigma_{i}$ for $2\leq i\leq n-1$.
\end{itemize}
In fact, the above relations suffice for a generators and relations presentation of $S_n$. When considering the symmetric group $S_n$, it is natural to restrict ourselves to permutations that fix sets of elements. This leads to the below definitions.
\begin{definition}For the $k$-tuple of positive integers
$\lambda = (\lambda_1,\ldots, \lambda_k)$, it is called a \emph{strict composition of $S_n$} if $\lambda_1 + \dots + \lambda_k = n$. We denote $\ell(\lambda) \coloneqq k$.
\end{definition}
\begin{remark}
The adjective \emph{strict} is commonly used to clarify that compositions consist of positive rather than non-negative integers. Throughout the rest of this paper, however, we will omit this adjective when describing strict compositions. 
\end{remark}
\begin{definition}
    For composition $\lambda$,  partition $\{1,2,\ldots,n\}$ into non-empty subsets with sizes based on our composition\footnote{Specifically the first $\lambda_1$ positive integers in the first set, the next $\lambda_2$ positive integers in the second set, and so on.}. We define the group of permutations which fix all subsets to be $S_{\lambda}$.
\end{definition}
\begin{remark}
We have $S_{\lambda} \cong S_{\lambda_1} \times \cdots \times S_{\lambda_{\ell(\lambda)}}$ as each $S_{\lambda_i}$ term in the Cartesian product characterizes how to map a set of $\lambda_i$ elements to itself.
\end{remark} 

We will now introduce the hyperoctahedral group and its analogous definitions.
\begin{definition}
The \emph{hyperoctahedral group of degree $n$} is $H_n = \left<\sigma_1, \ldots, \sigma_n \right >$ with relations
\begin{itemize}
  \item $\sigma_i^2 = 1$;
  \item $\sigma_i\sigma_j = \sigma_j\sigma_i$ for $|i-j|\geq 2$;
  \item $\sigma_{i-1}\sigma_i\sigma_{i-1} = \sigma_{i}\sigma_{i-1}\sigma_{i}$ for $2\leq i\leq n-1$;
  \item $\sigma_{n-1}\sigma_{n}\sigma_{n-1}\sigma_{n} = \sigma_{n}\sigma_{n-1}\sigma_{n}\sigma_{n-1}$.
\end{itemize}

\end{definition}

\begin{remark}
   Note that the first three relations are identical to those for $S_n$, so the subgroup generated by $\sigma_1, \ldots, \sigma_{n-1}$ is precisely $S_n$. Additionally, for $S_{2n}$ with generators $\sigma'_1, \ldots, \sigma'_{2n-1}$, one can verify that the mapping $\sigma_i\mapsto \sigma'_i\sigma'_{2n-i}$ for $i<n$ and $\sigma_n\mapsto\sigma'_n$ proves that $H_n$ is a subgroup of $S_{2n}$.
\end{remark}
\begin{definition}
    A \emph{string diagram} (or a \emph{crossing diagram}) for $\rho \in S_n$ consists of two rows of $n$ vertices where the $i$-th vertex in the bottom row has an edge connection to the $\rho(i)$-th vertex in the top row. See below for an example with $n=5$ and $\rho = (123)(45)$:
    
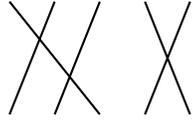
\begin{center}
\begin{tikzpicture}[scale=1,thick,anchorbase]
\begin{scope}[xshift=4cm]
\draw  (2.8,-0.5) -- +(.6,1.5);
\draw  (3.4,-0.5) -- +(0.6,1.5);
\draw  (4,-0.5) -- +(-1.2,1.5);
\draw  (5.2,-0.5) -- +(-.6,1.5);
\draw  (4.6,-0.5) -- +(.6,1.5);
\end{scope}
\end{tikzpicture}.
    \end{center}
\end{definition}
In regards to $H_n$ being a subgroup of $S_{2n}$, we can now describe the crossing diagrams for the generators of $H_n$. Specifically, $H_n$ is the subgroup of $S_{2n}$ consisting of all diagrams with a vertical line of symmetry. Each of the generators of $H_n$ account for this symmetry as seen below with the case of $n=3$:
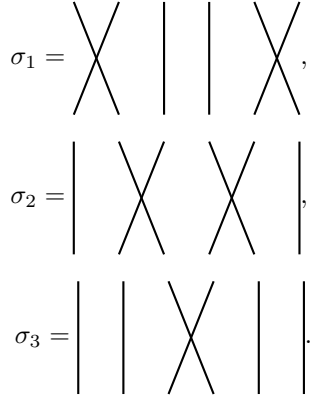
\begin{center}
$\sigma_1 = 
    \begin{tikzpicture}[anchorbase,scale=1,thick]
\draw  (2.2,-0.5) -- +(0.6,1.5);
\draw  (2.8,-0.5) -- +(-0.6,1.5);
\draw  (3.4,-0.5) -- +(0,1.5);
\draw  (4.0,-0.5) -- +(0,1.5);
\draw  (4.6,-0.5) -- +(0.6,1.5);
\draw  (5.2,-0.5) -- +(-0.6,1.5);
\end{tikzpicture} 
,$
\bigskip
\\*
$
\sigma_2 = 
    \begin{tikzpicture}[anchorbase,scale=1,thick]
\draw  (2.8,-0.5) -- +(0.6,1.5);
\draw  (3.4,-0.5) -- +(-0.6,1.5);
\draw  (2.2,-0.5) -- +(0,1.5);
\draw  (5.2,-0.5) -- +(0,1.5);
\draw  (4,-0.5) -- +(0.6,1.5);
\draw  (4.6,-0.5) -- +(-0.6,1.5);
\end{tikzpicture}
,$
\bigskip
\\*
$
\sigma_3 = 
    \begin{tikzpicture}
    [anchorbase,scale=1,thick]
\draw  (2.8,-0.5) -- +(0,1.5);
\draw  (3.4,-0.5) -- +(0.6,1.5);
\draw  (2.2,-0.5) -- +(0,1.5);
\draw  (5.2,-0.5) -- +(0,1.5);
\draw  (4,-0.5) -- +(-0.6,1.5);
\draw  (4.6,-0.5) -- +(0,1.5);
\end{tikzpicture}
.$
\end{center}
\begin{remark}\label{DKMZPresentation}
    There is an alternative presentation of $H_n$ motivated by the fact that $S_n$ is a subgroup. Given an element of $H_n$ and its diagram in $S_{2n}$, we can erase all strings starting from the right-half of the bottom row. To correct any remaining strings that cross the vertical line of symmetry, we reflect their top vertices across this line and mark these strings with a dot. For instance, the generators of $H_3$ are written as
    \begin{center}
$\sigma_1 = 
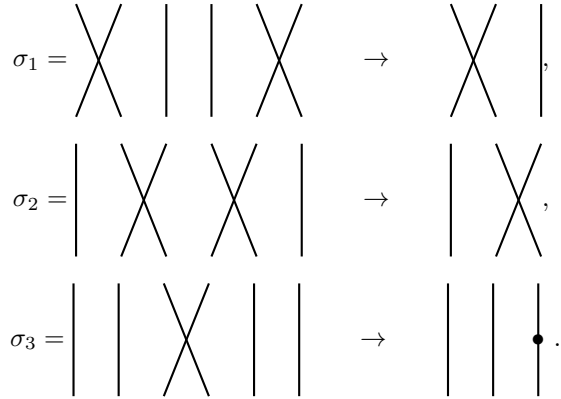
    \begin{tikzpicture}[anchorbase,scale=1,thick]
\draw  (2.2,-0.5) -- +(0.6,1.5);
\draw  (2.8,-0.5) -- +(-0.6,1.5);
\draw  (3.4,-0.5) -- +(0,1.5);
\draw  (4.0,-0.5) -- +(0,1.5);
\draw  (4.6,-0.5) -- +(0.6,1.5);
\draw  (5.2,-0.5) -- +(-0.6,1.5);
\end{tikzpicture} \quad \quad
\rightarrow \quad \quad
\begin{tikzpicture}[anchorbase,scale=1,thick]
\draw  (2.2,-0.5) -- +(0.6,1.5);
\draw  (2.8,-0.5) -- +(-0.6,1.5);
\draw  (3.4,-0.5) -- +(0,1.5);
\end{tikzpicture} ,$
\bigskip
\\*
$\sigma_2 = 
    \begin{tikzpicture}[anchorbase,scale=1,thick]
\draw  (2.2,-0.5) -- +(0,1.5);
\draw  (2.8,-0.5) -- +(0.6,1.5);
\draw  (3.4,-0.5) -- +(-0.6,1.5);
\draw  (4.0,-0.5) -- +(0.6,1.5);
\draw  (4.6,-0.5) -- +(-0.6,1.5);
\draw  (5.2,-0.5) -- +(0,1.5);
\end{tikzpicture} \quad \quad
\rightarrow \quad \quad
\begin{tikzpicture}[anchorbase,scale=1,thick]
\draw  (2.2,-0.5) -- +(0,1.5);
\draw  (2.8,-0.5) -- +(0.6,1.5);
\draw  (3.4,-0.5) -- +(-0.6,1.5);
\end{tikzpicture} ,$
\bigskip
\\*
$\sigma_3 = 
    \begin{tikzpicture}[anchorbase,scale=1,thick]
\draw  (2.2,-0.5) -- +(0,1.5);
\draw  (2.8,-0.5) -- +(0,1.5);
\draw  (3.4,-0.5) -- +(0.6,1.5);
\draw  (4.0,-0.5) -- +(-0.6,1.5);
\draw  (4.6,-0.5) -- +(0,1.5);
\draw  (5.2,-0.5) -- +(0,1.5);
\end{tikzpicture} \quad \quad
\rightarrow \quad \quad
\begin{tikzpicture}[anchorbase,scale=1,thick]
\draw  (2.2,-0.5) -- +(0,1.5);
\draw  (2.8,-0.5) -- +(0,1.5);
\draw  (3.4,-0.5) -- +(0,1.5);
\node at (3.4, 0.25) {\textbullet};
\end{tikzpicture}.$
\end{center}
In addition to the three relations depicted in Equation~\ref{E:3relns} for $S_n$, there are two other relations regarding our dots: they can freely move on a string and two of them on the same string can annihilate each other. While we do not use this dot notation for the rest of this paper, we like to mention that it is visually similar to the generating coupon morphism displayed in~\cite[Definition 3.1.1]{DKMZ}.
\end{remark}

Based on the crossing diagram presentation in $S_{2n}$ for elements of $H_n$, one might define compositions of $H_n$ as tuples of positive integers $\lambda = (\lambda_1,\ldots, \lambda_k)$ with~$\lambda_1+\dots+\lambda_k=2n$ and for all $j$, $\lambda_j = \lambda_{k+1-j}$; this latter condition stemming from the vertical line of symmetry. However, we will use the following altered definition. 
\begin{definition}
For a $k$-tuple of positive integers
$\gamma = (\gamma_1,\ldots, \gamma_k)$ with $\gamma_1 + \dots + \gamma_k = 2n$ and for all $j$, $\gamma_j = \gamma_{k+1-j}$, we say that $\lambda = (\gamma_1,\ldots, \gamma_k)$ is a \emph{hypercomposition of $H_n$} if $k$ is odd, otherwise $\lambda = (\gamma_1,\ldots,\gamma_{\frac{k}{2}}, 0,\gamma_{\frac{k}{2}+1},\ldots,\gamma_k)$ is a hypercomposition of $H_n$. 
\end{definition}
\begin{remark}
As with compositions, for $\lambda = (\lambda_1, \ldots, \lambda_k)$ a hypercomposition we have $\ell(\lambda) \coloneqq k$. The insertion of a $0$-element into a $k$-tuple for $k$ even ensures that all hypercompositions have odd length. 
\end{remark}
For example, the symmetric tuples for $H_2$ are $\gamma = (4), (2,2), (1,2,1),$ and $(1,1,1,1)$, but these correspond to hypercompositions $\lambda = (4), (2,0,2), (1,2,1),$ and $(1,1,0,1,1)$. As with the symmetric group, if we have a hypercomposition $\lambda$ of $H_n$ we can analogously partition $\{1,2,\ldots,2n\}$ into subsets of size $\lambda_1, \ldots, \lambda_{\ell(\lambda)}$. This allows us to define $H_\lambda$ as follows.
\begin{definition}
    For $\lambda$ a hypercomposition of $H_n$, $H_{\lambda}$ is the subgroup of $H_n$ that fixes all subsets.
\end{definition}
\begin{remark}
We have \[H_{\lambda} \cong S_{\lambda_1} \times \cdots \times S_{\lambda_{\frac{\ell(\lambda)-1}{2}}}\times H_{\frac{\lambda_{\frac{\ell(\lambda)+1}{2}}}{2}}\] as each $S_{\lambda_i}$ term in the Cartesian product characterizes how to map a set of $\lambda_i$ elements to itself, with these maps being reflected across the vertical line of symmetry, while $H_{\frac{\lambda_{\frac{\ell(\lambda)+1}{2}}}{2}}$ characterizes how to map a set of $\lambda_{\frac{\ell(\lambda)+1}{2}}$ elements to itself, a map that is already symmetric on the vertical line.
\end{remark} 
\begin{remark}
    Observe that $\lambda_{\frac{\ell(\lambda)+1}{2}}=0$ if and only if $\sigma_n\notin  H_\lambda$. Therefore, Green's restriction in~\cite[Definition 4.1.1]{Green} to consider subgroups without $\sigma_n$  is equivalent to only using hypercompositions with a $0$-element. 
\end{remark}
For instance, with the generators of $H_3$ we have $H_{(1,1,2,1,1)} = \langle \sigma_3 \rangle$ and $H_{(2,2,2)} =  \langle \sigma_1, \sigma_3 \rangle$. We can also interpret $S_n/S_\lambda$ and $H_n/H_{\lambda}$ for various compositions and hypercompositions. By applying the partition of $\lambda$ onto the bottom $n$ or $2n$ vertices, two diagrams are in the same coset of $S_n/S_\lambda$ or $H_n/H_{\lambda}$ if and only if you can permute the vertices in the bottom row within their respective subsets to go from one diagram to another. To choose a representative of a coset, we take the one where edges coming from the same subset do not cross. This is equivalent to choosing the diagram with the least number of crossings.

In the case of $S_4/S_{2,2}$ we have the following six representatives:
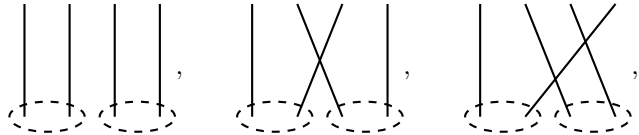
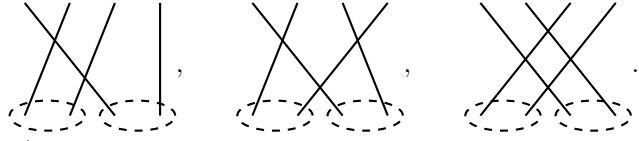
\begin{center}
\begin{tikzpicture}[scale=1,thick,anchorbase]
\begin{scope}[xshift=4cm]
\draw  (2.8,-0.5) -- +(0,1.5);
\draw  (3.4,-0.5) -- +(0,1.5);
\draw  (4,-0.5) -- +(0,1.5);
\draw  (4.6,-0.5) -- +(0,1.5);
\draw[ thick, dashed] (3.1,-0.5) ellipse (.5 and 0.2);
\draw[ thick, dashed] (4.3,-0.5) ellipse (.5 and 0.2);
\end{scope}
\end{tikzpicture},\quad \quad \begin{tikzpicture}[scale=1,thick,anchorbase]
\begin{scope}[xshift=4cm]
\draw  (2.8,-0.5) -- +(0,1.5);
\draw  (3.4,-0.5) -- +(0.6,1.5);
\draw  (4,-0.5) -- +(-.6,1.5);
\draw  (4.6,-0.5) -- +(0,1.5);
\draw[ thick, dashed] (3.1,-0.5) ellipse (.5 and 0.2);
\draw[ thick, dashed] (4.3,-0.5) ellipse (.5 and 0.2);
\end{scope}
\end{tikzpicture},\quad \quad\begin{tikzpicture}[scale=1,thick,anchorbase]
\begin{scope}[xshift=4cm]
\draw  (2.8,-0.5) -- +(0,1.5);
\draw  (3.4,-0.5) -- +(1.2,1.5);
\draw  (4,-0.5) -- +(-.6,1.5);
\draw  (4.6,-0.5) -- +(-.6,1.5);
\draw[ thick, dashed] (3.1,-0.5) ellipse (.5 and 0.2);
\draw[ thick, dashed] (4.3,-0.5) ellipse (.5 and 0.2);
\end{scope}
\end{tikzpicture},
\end{center}
\begin{center} 
\!
\begin{tikzpicture}[scale=1,thick,anchorbase]
\begin{scope}[xshift=4cm]
\draw  (2.8,-0.5) -- +(0.6,1.5);
\draw  (3.4,-0.5) -- +(0.6,1.5);
\draw  (4,-0.5) -- +(-1.2,1.5);
\draw  (4.6,-0.5) -- +(0,1.5);
\draw[ thick, dashed] (3.1,-0.5) ellipse (.5 and 0.2);
\draw[ thick, dashed] (4.3,-0.5) ellipse (.5 and 0.2);
\end{scope}
\end{tikzpicture},\quad \quad \begin{tikzpicture}[scale=1,thick,anchorbase]
\begin{scope}[xshift=4cm]
\draw  (2.8,-0.5) -- +(0.6,1.5);
\draw  (3.4,-0.5) -- +(1.2,1.5);
\draw  (4,-0.5) -- +(-1.2,1.5);
\draw  (4.6,-0.5) -- +(-.6,1.5);
\draw[ thick, dashed] (3.1,-0.5) ellipse (.5 and 0.2);
\draw[ thick, dashed] (4.3,-0.5) ellipse (.5 and 0.2);
\end{scope}
\end{tikzpicture},\quad \quad\begin{tikzpicture}[scale=1,thick,anchorbase]
\begin{scope}[xshift=4cm]
\draw  (2.8,-0.5) -- +(1.2,1.5);
\draw  (3.4,-0.5) -- +(1.2,1.5);
\draw  (4,-0.5) -- +(-1.2,1.5);
\draw  (4.6,-0.5) -- +(-1.2,1.5);
\draw[ thick, dashed] (3.1,-0.5) ellipse (.5 and 0.2);
\draw[ thick, dashed] (4.3,-0.5) ellipse (.5 and 0.2);
\end{scope}
\end{tikzpicture}.
\end{center}
Additionally, $H_2/H_{1,2,1}$ has the following four representatives:
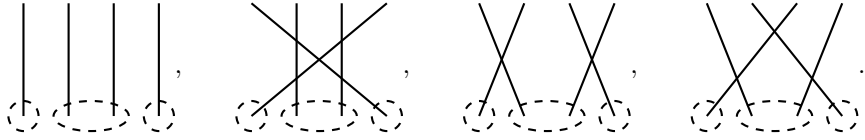
\begin{center}
\begin{tikzpicture}[scale=1,thick,anchorbase]
\begin{scope}[xshift=4cm]
\draw  (2.8,-0.5) -- +(0,1.5);
\draw  (3.4,-0.5) -- +(0,1.5);
\draw  (4,-0.5) -- +(0,1.5);
\draw  (4.6,-0.5) -- +(0,1.5);
\draw[ thick, dashed] (2.8,-0.5) ellipse (.2 and 0.2);
\draw[ thick, dashed] (3.7,-0.5) ellipse (.5 and 0.2);
\draw[ thick, dashed] (4.6,-0.5) ellipse (.2 and 0.2);
\end{scope}
\end{tikzpicture},\quad \quad \begin{tikzpicture}[scale=1,thick,anchorbase]
\begin{scope}[xshift=4cm]
\draw  (2.8,-0.5) -- +(1.8,1.5);
\draw  (3.4,-0.5) -- +(0,1.5);
\draw  (4,-0.5) -- +(0,1.5);
\draw  (4.6,-0.5) -- +(-1.8,1.5);
\draw[ thick, dashed] (2.8,-0.5) ellipse (.2 and 0.2);
\draw[ thick, dashed] (3.7,-0.5) ellipse (.5 and 0.2);
\draw[ thick, dashed] (4.6,-0.5) ellipse (.2 and 0.2);
\end{scope}
\end{tikzpicture},\quad \quad 
\begin{tikzpicture}[scale=1,thick,anchorbase]
\begin{scope}[xshift=4cm]
\draw  (2.8,-0.5) -- +(0.6,1.5);
\draw  (3.4,-0.5) -- +(-0.6,1.5);
\draw  (4,-0.5) -- +(+0.6,1.5);
\draw  (4.6,-0.5) -- +(-0.6,1.5);
\draw[ thick, dashed] (2.8,-0.5) ellipse (.2 and 0.2);
\draw[ thick, dashed] (3.7,-0.5) ellipse (.5 and 0.2);
\draw[ thick, dashed] (4.6,-0.5) ellipse (.2 and 0.2);
\end{scope}
\end{tikzpicture},\quad \quad \begin{tikzpicture}[scale=1,thick,anchorbase]
\begin{scope}[xshift=4cm]
\draw  (2.8,-0.5) -- +(1.2,1.5);
\draw  (3.4,-0.5) -- +(-0.6,1.5);
\draw  (4,-0.5) -- +(+0.6,1.5);
\draw  (4.6,-0.5) -- +(-1.2,1.5);
\draw[ thick, dashed] (2.8,-0.5) ellipse (.2 and 0.2);
\draw[ thick, dashed] (3.7,-0.5) ellipse (.5 and 0.2);
\draw[ thick, dashed] (4.6,-0.5) ellipse (.2 and 0.2);
\end{scope}
\end{tikzpicture}.
\end{center}
 By labelling edges from $1$ to $\ell(\lambda)$ based on the subset their bottom vertices are in, and then reading the labels by the top vertices from left to right, we can represent $H_n/ H_\lambda$ with $2n$-tuples as follows. 
\begin{definition}
For $\lambda = (\lambda_1,\ldots, \lambda_k)$ a hypercomposition of $H_n$, 
\[
I_\lambda \coloneqq \{(i_1, i_2, \ldots, i_{2n}) \ | \  \{\#r| i_r = j\}\ = \lambda_j \ \text{and} \  i_d +i_{2n+1-d} = \ell(\lambda)+1\}.
\]
\end{definition}
\begin{remark}
    We have the anti-symmetry condition $i_d +i_{2n+1-d} = \ell(\lambda)+1$ due to the vertical symmetry of crossing diagrams of $H_n$. Additionally, it is straightforward to show that our previously mentioned labelling process is a bijection between sets $H_n/H_{\lambda}$ and $I_\lambda$.
    \end{remark}
To formalize such a bijection we introduce the following.
\begin{definition}
    For a hypercomposition $\lambda$ of $H_{n}$, we define its labelling function as $L_\lambda: \{1, \ldots, 2n\} \rightarrow \{1, \ldots, \ell(\lambda)\}$ such that the first $\lambda_1$ positive integers are sent to $1$, the next $\lambda_2$ positive integers are sent to $2$, and so on.
\end{definition}
\begin{definition}
   The bijective mapping $\phi_{\lambda}: H_n/H_\lambda \rightarrow I_\lambda$ is defined as $\phi_{\lambda}(gH_\lambda) = (L_\lambda(g^{-1}(1)), \ldots,L_\lambda(g^{-1}(2n)) )$ for $g\in H_n$.
\end{definition}
\begin{remark}
    Note that this is well-defined due to how $H_\lambda$ is the subgroup of permutations of $H_n$ that do not affect the labelling scheme .
\end{remark}
As an example of going between crossing diagrams and $2n$-tuples, see the below pictures for $\lambda = (1,2,1)$:
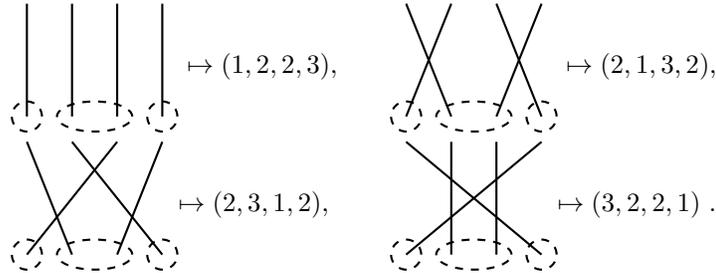
\begin{center}
\begin{tikzpicture}[scale=1,thick,anchorbase ]
\begin{scope}[xshift=4cm]
\draw  (2.8,-0.5) -- +(0,1.5);
\draw  (3.4,-0.5) -- +(0,1.5);
\draw  (4,-0.5) -- +(0,1.5);
\draw  (4.6,-0.5) -- +(0,1.5);
\draw[ thick, dashed] (2.8,-0.5) ellipse (.2 and 0.2);
\draw[ thick, dashed] (3.7,-0.5) ellipse (.5 and 0.2);
\draw[ thick, dashed] (4.6,-0.5) ellipse (.2 and 0.2);
\end{scope}
\end{tikzpicture} $\mapsto (1,2,2,3)$,\quad \quad
\begin{tikzpicture}[scale=1,thick,anchorbase]
\begin{scope}[xshift=4cm]
\draw  (2.8,-0.5) -- +(0.6,1.5);
\draw  (3.4,-0.5) -- +(-0.6,1.5);
\draw  (4,-0.5) -- +(+0.6,1.5);
\draw  (4.6,-0.5) -- +(-0.6,1.5);
\draw[ thick, dashed] (2.8,-0.5) ellipse (.2 and 0.2);
\draw[ thick, dashed] (3.7,-0.5) ellipse (.5 and 0.2);
\draw[ thick, dashed] (4.6,-0.5) ellipse (.2 and 0.2);
\end{scope}
\end{tikzpicture} $ \mapsto  (2,1,3,2)$,
\end{center}
\begin{center} 
\begin{tikzpicture}[scale=1,thick,anchorbase]
\begin{scope}[xshift=4cm]
\draw  (2.8,-0.5) -- +(1.2,1.5);
\draw  (3.4,-0.5) -- +(-0.6,1.5);
\draw  (4,-0.5) -- +(+0.6,1.5);
\draw  (4.6,-0.5) -- +(-1.2,1.5);
\draw[ thick, dashed] (2.8,-0.5) ellipse (.2 and 0.2);
\draw[ thick, dashed] (3.7,-0.5) ellipse (.5 and 0.2);
\draw[ thick, dashed] (4.6,-0.5) ellipse (.2 and 0.2);
\end{scope}
\end{tikzpicture}$\mapsto(2,3,1,2)$, \quad \quad
\begin{tikzpicture}[scale=1,thick,anchorbase] \begin{scope}[xshift=4cm] \draw  (2.8,-0.5) -- +(1.8,1.5); \draw  (3.4,-0.5) -- +(0,1.5); \draw  (4,-0.5) -- +(0,1.5); \draw  (4.6,-0.5) -- +(-1.8,1.5);\draw[ thick, dashed] (2.8,-0.5) ellipse (.2 and 0.2);
\draw[ thick, dashed] (3.7,-0.5) ellipse (.5 and 0.2);
\draw[ thick, dashed] (4.6,-0.5) ellipse (.2 and 0.2);
\end{scope} \end{tikzpicture}$\mapsto(3,2,2,1)$ .
\end{center}\label{H2Reps}
The main use for this tuple representation is because of the following definition and lemma.
\begin{definition}
    We define the action of $H_n$ on $I_\lambda$ by viewing $h\in H_n$ as a permutation in $S_{2n}$ and letting  $h\cdot (\bi_1, \bi_2, \ldots, \bi_{2n} ) = (\bi_{h^{-1}(1)}, \ldots, \bi_{h^{-1}(2n)})$.
\end{definition}
\begin{remark}
    In~\cite[Section 4]{Brundan} they define their analogous action for $S_n$ as a right action while the above definition is a left action. 
\end{remark}
\begin{lemma}
    The action of $H_n$ on $H_n/H_\lambda$ is isomorphic to its action on $I_\lambda$.
\end{lemma}
\begin{proof}
    It suffices to show that $\phi_{\lambda}(h\cdot gH_\lambda) = h\cdot \phi_{\lambda}(gH_\lambda)$ for $h, g\in H_n$. We have
    \[\phi_{\lambda}(h\cdot gH_\lambda) = \phi_{\lambda}((hg)H_\lambda) = (L_\lambda((hg)^{-1}(1)), \ldots,L_\lambda((hg)^{-1}(2n)) ) \]
    \[ =  (L_\lambda(g^{-1}(h^{-1}(1))), \ldots,L_\lambda(g^{-1}(h^{-1}(n)))) \]\[= h\cdot (L_\lambda(g^{-1}(1)), \ldots,L_\lambda(g^{-1}(2n)) )= h\cdot \phi_{\lambda}(gH_\lambda).\]
\end{proof}
\subsection{Schur Algebra}\label{sec:schur}

Before defining the Schur algebra, we will briefly mention some ideas from representation theory.

Recall that a representation of a group is a vector space $V$ equipped with a group homomorphism from the group to $GL(V)$. Consider a subgroup $H$ of $G$. If we have a representation $W$ of $H$, then there is a natural way to construct a representation of $G$. This is formally called an induced representation~\cite[Definition 5.8.1]{EGHLSVY}.

\begin{definition}
If $G$ is a group, $H\subset G$, and $W$ is a representation of $H$, then the \emph{induced representation} 
\[
\Ind_H^{G}(W)=\{f: G\rightarrow W|f(hx) = h\cdot f(x), \forall x\in G,h\in H\}
\]
is the representation of $G$ with action $g(f)(x) = f(xg)$, for all $g \in G$. 
\end{definition}

Any group $G$ will have at least one representation, the trivial representation denoted by $\triv_G$. This is defined by taking the one-dimensional vector space $V = \mathbb{F}$ and using the group homomorphism $\phi: G\rightarrow GL(V)$ satisfying $\phi(g)(v) = v$. Combining this with induced representations, we can create $\Ind_H^{G}(\triv_H)$ which has the following isomorphism as proven in~\cite{EGHLSVY}.
\begin{proposition}
As vector spaces, $\Ind_H^{G}(\triv_H)\cong \mathbb{F}[G/ H]$.
\end{proposition}

\begin{definition}
   The set of homomorphisms which commute with group $G$ is defined by $\Hom_G(U,V) \coloneqq \{f \in \Hom(U,V) | gf(u) = f(gu), \forall g\in G, u\in U\}$ where the two multiplications are group actions of $G$ on $U,V$.
\end{definition}

\begin{definition}
For a group $G$, set $X$, and group action $*: G\times X\rightarrow X$, the  \emph{orbit} of an element $x\in X$ is defined as $G(x) = \{g*x\in X | g\in G\}$.
\end{definition}
In general, understanding homomorphisms between arbitrary representations is difficult, but when both the representations are induced from a subgroup, questions about homomorphisms are often translated into problems regarding orbits. In particular, we have the following.

\begin{proposition}\label{basiscross}
For finite sets $X = \{x_1,\ldots,x_m\}$, $Y= \{y_1,\ldots,y_n\}$, the set of orbits of group $G$ on $X\times Y$ creates a basis of $\Hom_{G}(\mathbb{F}[X], \mathbb{F}[Y])$ where an orbit $S$ corresponds to the homomorphism with matrix representation defined by $a_{i,j} = 1$ if $(x_j, y_i)\in S$, $a_{i,j} = 0$ otherwise.
\end{proposition}

We will now define the Schur algebra for a fixed $n\geq 1$.
\begin{definition}
    For $\lambda$ a composition of $S_n$, write $M_{\lambda}\coloneqq \Ind_{S_{\lambda}}^{S_n}(\triv_{S_{\lambda}})$. 
\end{definition}
\begin{definition}
    The \emph{Schur algebra of degree $n$} is $\End_{S_n}(\bigoplus_{\lambda}M_{\lambda})$ where $\lambda$ ranges over all compositions of $n$.
\end{definition}

 The Schur algebra is an algebra with idempotents: $e_{\mu}:\oplus M_{\lambda}\rightarrow \oplus M_{\lambda}$, where $e_{\mu}|_{M_{\mu}} = \mathrm{id}$ and $e_{\mu}|_{M_{\lambda \ne \mu}} = 0$. Moreover, $\Hom_{S_n}(M_{\lambda}, M_{\mu})\cong e_{\mu} \End_{S_n}(\oplus_{\lambda}M_\lambda) e_{\lambda}$.

\begin{remark}
     It is important to note that $\End_{S_n}(\oplus_{\lambda}M_\lambda)\cong \oplus_{\mu,\lambda}\Hom_{S_n}(M_{\mu}, M_\lambda)$, so we often just focus on finding a basis of $\Hom_{S_n}(M_{\mu}, M_\lambda)$ over all compositions $\mu,\lambda$ to derive a basis for the entire algebra.
\end{remark}
As an example, we can consider $\End_{S_4}(M_{2,2})$. By Proposition~\ref{basiscross}, this endomorphism algebra has a basis stemming from the orbits of $S_4$ on $(S_4 / S_2\times S_2)\times (S_4 / S_2\times S_2)$. These correspond to three matrices:
\begin{center}
    \[E_1 = \begin{bmatrix}
1 & 0 & 0 & 0 & 0 & 0  \\
0 & 1 & 0 & 0 & 0 & 0  \\
0 & 0 & 1 & 0 & 0 & 0  \\
0 & 0 & 0 & 1 & 0 & 0  \\
0 & 0 & 0 & 0 & 1 & 0  \\
0 & 0 & 0 & 0 & 0 & 1  \\
\end{bmatrix},\]
\[E_2 = \begin{bmatrix}
0 & 0 & 0 & 0 & 0 & 1  \\
0 & 0 & 0 & 0 & 1 & 0  \\
0 & 0 & 0 & 1 & 0 & 0  \\
0 & 0 & 1 & 0 & 0 & 0  \\
0 & 1 & 0 & 0 & 0 & 0  \\
1 & 0 & 0 & 0 & 0 & 0  \\
\end{bmatrix},\]
\[E_3 = \begin{bmatrix}
0 & 1 & 1 & 1 & 1 & 0  \\
1 & 0 & 1 & 1 & 0 & 1  \\
1 & 1 & 0 & 0 & 1 & 1  \\
1 & 1 & 0 & 0 & 1 & 1  \\
1 & 0 & 1 & 1 & 0 & 1  \\
0 & 1 & 1 & 1 & 1 & 0  \\
\end{bmatrix}.\]
\end{center}
In particular, note that 
\begin{equation}\label{matrixex}
E_3\cdot E_3 = 4E_1+4E_2+2E_3
\end{equation}
so, as expected, the algebra is closed under multiplication. Following~\cite[Definition 4.2]{Brundan}, we formalize the construction of a category for the Schur algebra.

\begin{definition}
    The \emph{Schur category}, denoted $\Schur$, is the category with compositions and objects and morphisms: 
    \[
    \Hom_{\Schur}(\mu, \lambda):=\begin{cases}\Hom_{S_n}(M_{\mu}, M_\lambda) \quad \text{if} \quad \sum \lambda_i = \sum \mu_i \\
    0 \quad \text{otherwise}.
    \end{cases}.
    \]
\end{definition}

\begin{remark}
    It is further noted in~\cite[Definition 4.2]{Brundan} that $\Schur$ is a strict monoidal category. There is a standard graphical calculus which is convenient for describing strict monoidal categories. This is made precise in the following section. 
\end{remark}

\subsection{Web Diagrams and Chicken Foot Diagrams}\label{sec:web}

In light of the computations of structure constants for the algebra $\End_{S_n}(\mathbb{F}[S_n / S_\lambda])$, such as those for $\End_{S_4}(\mathbb{F}[S_4 / S_2\times S_2])$ in Subsection~\ref{sec:schur}, web diagrams in the polynomial web category were studied in~\cite{Brundan}. We now informally recall some key aspects of this construction, for more precise definitions see~\cite[Definition 4.7]{Brundan}.
\begin{definition}
   For $\mu, \lambda$ compositions of $n$, a $\lambda \times \mu$ \emph{web diagram} consists of strings determined by $\mu$ at the bottom that split, merge, and cross each other, ending at $\lambda$ at the top.
\end{definition}
Below is an example of a web diagram with $\mu = (3,2,4)$ and $\lambda = (5,4)$,
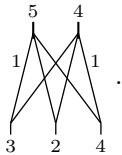
\begin{align*}
  \begin{tikzpicture}[anchorbase,scale=1.5]
\draw[-,thick] (.2,.5) to (.2,.35);
\draw[-,thick] (.6,.5) to (.6,.35);
\draw[-] (0,-.4) to (.2,.4);
\draw[-] (0,-.4) to (.6,.4);
\draw[-] (.4,-.4) to (.6,.4);
\draw[-] (.4,-.4) to (.2,.4);
\draw[-] (.8,-.4) to (.2,.4);
\draw[-] (.8,-.4) to (.6,.4);
\draw[-] (.4,-.5) to (.4,-.4);
\draw[-] (.8,-.5) to (.8,-.4);
\draw[-] (0,-.5) to (0,-.4);
\node at (0.05,0.15) {$\scriptstyle 1$};
\node at (0.75,0.15) {$\scriptstyle 1$};
\node at (0,-0.6) {$\scriptstyle 3$};
\node at (0.4,-0.6) {$\scriptstyle 2$};
\node at (0.8,-0.6) {$\scriptstyle 4$};
\node at (0.2,0.6) {$\scriptstyle 5$};
\node at (0.6,0.6) {$\scriptstyle 4$};
\end{tikzpicture}.
\end{align*}Note that all of the string thicknesses are determinable from the given numbers. Additionally, this web diagram has two properties that make it especially nice to use.
\begin{definition}
   A web diagram is a \emph{chicken foot diagram} if, from the bottom to the top, the generators are splits, crossings, and merges in that order.
\end{definition}
\begin{definition}
    A chicken foot diagram is \emph{reduced} if every pair of strings intersect, merge, and split together at most once.
\end{definition}
When we are considering the strucutre of $\End_{S_n}(\mathbb{F}[S_n /S_\lambda])$, we examine various $\lambda\times \lambda$ chicken foot diagrams, and then consider multiplication of diagrams as putting one on top of another.  We have the following three general relations for reducing conjoined diagrams down to a linear combination of reduced chicken foot diagrams:
\begin{align}\label{introsplitchoice}
\begin{tikzpicture}[anchorbase]
	\draw[-,thick] (0.35,-.3) to (0.08,0.14);
	\draw[-,thick] (0.1,-.3) to (-0.04,-0.06);
	\draw[-,line width=1pt] (0.085,.14) to (-0.035,-0.06);
	\draw[-,thick] (-0.2,-.3) to (0.07,0.14);
	\draw[-,line width=2pt] (0.08,.45) to (0.08,.1);
        \node at (0.45,-.41) {$\scriptstyle c$};
        \node at (0.07,-.4) {$\scriptstyle b$};
        \node at (-0.28,-.41) {$\scriptstyle a$};
\end{tikzpicture}
&=
\begin{tikzpicture}[anchorbase]
	\draw[-,thick] (0.36,-.3) to (0.09,0.14);
	\draw[-,thick] (0.06,-.3) to (0.2,-.05);
	\draw[-,line width=1pt] (0.07,.14) to (0.19,-.06);
	\draw[-,thick] (-0.19,-.3) to (0.08,0.14);
	\draw[-,line width=2pt] (0.08,.45) to (0.08,.1);
        \node at (0.45,-.41) {$\scriptstyle c$};
        \node at (0.07,-.4) {$\scriptstyle b$};
        \node at (-0.28,-.41) {$\scriptstyle a$};
\end{tikzpicture}\:,
\qquad
\begin{tikzpicture}[anchorbase]
	\draw[-,thick] (0.35,.3) to (0.08,-0.14);
	\draw[-,thick] (0.1,.3) to (-0.04,0.06);
	\draw[-,line width=1pt] (0.085,-.14) to (-0.035,0.06);
	\draw[-,thick] (-0.2,.3) to (0.07,-0.14);
	\draw[-,line width=2pt] (0.08,-.45) to (0.08,-.1);
        \node at (0.45,.4) {$\scriptstyle c$};
        \node at (0.07,.42) {$\scriptstyle b$};
        \node at (-0.28,.4) {$\scriptstyle a$};
\end{tikzpicture}
=\begin{tikzpicture}[anchorbase]
	\draw[-,thick] (0.36,.3) to (0.09,-0.14);
	\draw[-,thick] (0.06,.3) to (0.2,.05);
	\draw[-,line width=1pt] (0.07,-.14) to (0.19,.06);
	\draw[-,thick] (-0.19,.3) to (0.08,-0.14);
	\draw[-,line width=2pt] (0.08,-.45) to (0.08,-.1);
        \node at (0.45,.4) {$\scriptstyle c$};
        \node at (0.07,.42) {$\scriptstyle b$};
        \node at (-0.28,.4) {$\scriptstyle a$};
\end{tikzpicture}\:,\end{align}
\begin{align}
\label{introtrivial}
\begin{tikzpicture}[anchorbase,scale=.8]
	\draw[-,line width=2pt] (0.08,-.8) to (0.08,-.5);
	\draw[-,line width=2pt] (0.08,.3) to (0.08,.6);
\draw[-,thick] (0.1,-.51) to [out=45,in=-45] (0.1,.31);
\draw[-,thick] (0.06,-.51) to [out=135,in=-135] (0.06,.31);
        \node at (-.33,-.05) {$\scriptstyle a$};
        \node at (.45,-.05) {$\scriptstyle b$};
\end{tikzpicture}
&= 
\binom{a+b}{a}\:\:
\begin{tikzpicture}[anchorbase,scale=.8]
	\draw[-,line width=2pt] (0.08,-.8) to (0.08,.6);
        \node at (.62,-.05) {$\scriptstyle a+b$};
\end{tikzpicture},\end{align}
\begin{align}
\label{intromergesplit}
\begin{tikzpicture}[anchorbase,scale=1]
	\draw[-,line width=1.2pt] (0,0) to (.275,.3) to (.275,.7) to (0,1);
	\draw[-,line width=1.2pt] (.6,0) to (.315,.3) to (.315,.7) to (.6,1);
        \node at (0,1.13) {$\scriptstyle b$};
        \node at (0.63,1.13) {$\scriptstyle d$};
        \node at (0,-.1) {$\scriptstyle a$};
        \node at (0.63,-.1) {$\scriptstyle c$};
\end{tikzpicture}
&=
\sum_{\substack{0 \leq s \leq \min(a,b)\\0 \leq t \leq \min(c,d)\\t-s=d-a}}
\begin{tikzpicture}[anchorbase,scale=1]
	\draw[-,thick] (0.58,0) to (0.58,.2) to (.02,.8) to (.02,1);
	\draw[-,thick] (0.02,0) to (0.02,.2) to (.58,.8) to (.58,1);
	\draw[-,thin] (0,0) to (0,1);
	\draw[-,line width=1pt] (0.61,0) to (0.61,1);
        \node at (0,1.13) {$\scriptstyle b$};
        \node at (0.6,1.13) {$\scriptstyle d$};
        \node at (0,-.1) {$\scriptstyle a$};
        \node at (0.6,-.1) {$\scriptstyle c$};
        \node at (-0.1,.5) {$\scriptstyle s$};
        \node at (0.77,.5) {$\scriptstyle t$};
\end{tikzpicture}.
\end{align}
This also lets you deduce other relations whose proofs are presented in~\cite[Appendix A]{Brundan}:
\begin{align}
\label{introswitch}
\begin{tikzpicture}[anchorbase,scale=1]
	\draw[-,line width=1.2pt] (0,0) to (.6,1);
	\draw[-,line width=1.2pt] (0,1) to (.6,0);
        \node at (0,-.1) {$\scriptstyle a$};
        \node at (0.6,-0.1) {$\scriptstyle b$};
\end{tikzpicture}
&=\sum_{t=0}^{\min(a,b)}
(-1)^t
\begin{tikzpicture}[anchorbase,scale=1]
	\draw[-,thick] (0,0) to (0,1);
	\draw[-,thick] (0.015,0) to (0.015,.2) to (.57,.4) to (.57,.6)
        to (.015,.8) to (.015,1);
	\draw[-,line width=1.2pt] (0.6,0) to (0.6,1);
        \node at (0.6,-.1) {$\scriptstyle b$};
        \node at (0,-.1) {$\scriptstyle a$};
        \node at (-0.1,.5) {$\scriptstyle t$};
\end{tikzpicture}
, \end{align}
\begin{align}
\label{introswallows}
\begin{tikzpicture}[anchorbase,scale=.7]
	\draw[-,line width=2pt] (0.08,.3) to (0.08,.5);
\draw[-,thick] (-.2,-.8) to [out=45,in=-45] (0.1,.31);
\draw[-,thick] (.36,-.8) to [out=135,in=-135] (0.06,.31);
        \node at (-.3,-.95) {$\scriptstyle a$};
        \node at (.45,-.95) {$\scriptstyle b$};
\end{tikzpicture}
=\begin{tikzpicture}[anchorbase,scale=.7]
	\draw[-,line width=2pt] (0.08,.1) to (0.08,.5);
\draw[-,thick] (.46,-.8) to [out=100,in=-45] (0.1,.11);
\draw[-,thick] (-.3,-.8) to [out=80,in=-135] (0.06,.11);
        \node at (-.3,-.95) {$\scriptstyle a$};
        \node at (.43,-.95) {$\scriptstyle b$};
\end{tikzpicture}
,\end{align}
\begin{align}\label{introsliders}
\begin{tikzpicture}[anchorbase,scale=0.7]
	\draw[-,thick] (0.4,0) to (-0.6,1);
	\draw[-,thick] (0.08,0) to (0.08,1);
	\draw[-,thick] (0.1,0) to (0.1,.6) to (.5,1);
        \node at (0.6,1.13) {$\scriptstyle c$};
        \node at (0.1,1.16) {$\scriptstyle b$};
        \node at (-0.65,1.13) {$\scriptstyle a$};
\end{tikzpicture} = \begin{tikzpicture}[anchorbase,scale=0.7]
	\draw[-,thick] (0.7,0) to (-0.3,1);
	\draw[-,thick] (0.08,0) to (0.08,1);
	\draw[-,thick] (0.1,0) to (0.1,.2) to (.9,1);
        \node at (0.9,1.13) {$\scriptstyle c$};
        \node at (0.1,1.16) {$\scriptstyle b$};
        \node at (-0.4,1.13) {$\scriptstyle a$};
\end{tikzpicture}
,\end{align} 
\begin{align}
\label{introsymmetric}
\begin{tikzpicture}[anchorbase,scale=0.8]
	\draw[-,thick] (0.28,0) to[out=90,in=-90] (-0.28,.6);
	\draw[-,thick] (-0.28,0) to[out=90,in=-90] (0.28,.6);
	\draw[-,thick] (0.28,-.6) to[out=90,in=-90] (-0.28,0);
	\draw[-,thick] (-0.28,-.6) to[out=90,in=-90] (0.28,0);
        \node at (0.3,-.75) {$\scriptstyle b$};
        \node at (-0.3,-.75) {$\scriptstyle a$};
\end{tikzpicture}
= \begin{tikzpicture}[anchorbase,scale=0.8]
	\draw[-,thick] (0.2,-.6) to (0.2,.6);
	\draw[-,thick] (-0.2,-.6) to (-0.2,.6);
        \node at (0.2,-.75) {$\scriptstyle b$};
        \node at (-0.2,-.75) {$\scriptstyle a$};
\end{tikzpicture}
,\end{align}
\begin{align}
\label{introbraid}
\begin{tikzpicture}[anchorbase,scale=0.8]
	\draw[-,thick] (0.45,.6) to (-0.45,-.6);
	\draw[-,thick] (0.45,-.6) to (-0.45,.6);
        \draw[-,thick] (0,-.6) to[out=90,in=-90] (-.45,0);
        \draw[-,thick] (-0.45,0) to[out=90,in=-90] (0,0.6);
        \node at (0,-.77) {$\scriptstyle b$};
        \node at (0.5,-.77) {$\scriptstyle c$};
        \node at (-0.5,-.77) {$\scriptstyle a$};
\end{tikzpicture}
 =
\begin{tikzpicture}[anchorbase,scale=0.8]
	\draw[-,thick] (0.45,.6) to (-0.45,-.6);
	\draw[-,thick] (0.45,-.6) to (-0.45,.6);
        \draw[-,thick] (0,-.6) to[out=90,in=-90] (.45,0);
        \draw[-,thick] (0.45,0) to[out=90,in=-90] (0,0.6);
        \node at (0,-.77) {$\scriptstyle b$};
        \node at (0.5,-.77) {$\scriptstyle c$};
        \node at (-0.5,-.77) {$\scriptstyle a$};
\end{tikzpicture}.
\end{align}
In the case of $\End_{S_4}(\mathbb{F}[S_4 / S_2\times S_2])$, there are three possible $(2,2) \times (2,2)$ reduced chicken foot diagrams which serve as a basis:
\begin{align}D_1&=
\begin{tikzpicture}[anchorbase,scale=.8]
\draw[-,line width=2pt] (0.08,-.8) to (0.08,-.5);
\draw[-,line width=2pt] (0.08,.3) to (0.08,.6);
\draw[-,line width=1pt] (0.08,-.51) to  (0.08,.31);
\draw[-,line width=2pt] (1.08,-.8) to (1.08,-.5);
\draw[-,line width=2pt] (1.08,.3) to (1.08,.6);
\draw[-,line width=1pt] (1.08,-.51) to  (1.08,.31);
\node at (-.33,-.05) {$\scriptstyle 2$};
\node at (1.45,-.05) {$\scriptstyle 2$};
\end{tikzpicture}
,D_2=
\begin{tikzpicture}[anchorbase,scale=.8]
\draw[-,line width=2pt] (0.08,-.8) to (0.08,-.5);
\draw[-,line width=2pt] (0.08,.3) to (0.08,.6);
\draw[-,line width=1pt] (0.08,-.51) to  (1.08,.31); 
\draw[-,line width=2pt] (1.08,-.8) to (1.08,-.5);
\draw[-,line width=2pt] (1.08,.3) to (1.08,.6);
\draw[-,line width=1pt] (1.08,-.51) to  (0.08,.31);
\node at (-.33,-.05) {$\scriptstyle 2$};
\node at (1.45,-.05) {$\scriptstyle 2$};
\end{tikzpicture}
,D_3=
\begin{tikzpicture}[anchorbase,scale=.8]
\draw[-,line width=2pt] (0.08,-.8) to (0.08,-.5);
\draw[-,line width=2pt] (0.08,.3) to (0.08,.6);
\draw[-,line width=1pt] (0.08,-.51) to  (1.08,.31); 
\draw[-,line width=2pt] (1.08,-.8) to (1.08,-.5);
\draw[-,line width=2pt] (1.08,.3) to (1.08,.6);
\draw[-,line width=1pt] (1.08,-.51) to  (0.08,.31);
\draw[-,line width=1pt] (0.08,-.51) to  (0.08,.31);
\draw[-,line width=1pt] (1.08,-.51) to  (1.08,.31);
\node at (-.33,-.05) {$\scriptstyle 1$};
\node at (1.45,-.05) {$\scriptstyle 1$};
\node at (.33,.25) {$\scriptstyle 1$};
\node at (.77,.25) {$\scriptstyle 1$};
\end{tikzpicture}
.&
\end{align}
As an example to see how multiplication works, consider $D_3\cdot D_3$. We have the following diagram as a result,
\[
D_3\cdot D_3 = 
    \begin{tikzpicture}[anchorbase,scale=.8]
\draw[-,line width=2pt] (0.08,-.8) to (0.08,-.5);
\draw[-,line width=2pt] (0.08,.3) to (0.08,.6);
\draw[-,line width=1pt] (0.08,-.51) to  (1.08,.31); 
\draw[-,line width=2pt] (1.08,-.8) to (1.08,-.5);
\draw[-,line width=2pt] (1.08,.3) to (1.08,.6);
\draw[-,line width=1pt] (1.08,-.51) to  (0.08,.31);
\draw[-,line width=1pt] (0.08,-.51) to  (0.08,.31);
\draw[-,line width=1pt] (1.08,-.51) to  (1.08,.31);
\node at (-.33,-.05) {$\scriptstyle 1$};
\node at (1.45,-.05) {$\scriptstyle 1$};
\draw[-,line width=2pt] (0.08,1.4) to (0.08,1.7);
\draw[-,line width=1pt] (0.08,.59) to  (1.08,1.41); 
\draw[-,line width=2pt] (1.08,1.4) to (1.08,1.7);
\draw[-,line width=1pt] (1.08,.59) to  (0.08,1.41);
\draw[-,line width=1pt] (0.08,.59) to  (0.08,1.41);
\draw[-,line width=1pt] (1.08,.59) to  (1.08,1.41);
\node at (-.33,1.05) {$\scriptstyle 1$};
\node at (1.45,1.05) {$\scriptstyle 1$};
\node at (.33,1.35) {$\scriptstyle 1$};
\node at (.77,1.35) {$\scriptstyle 1$};
\end{tikzpicture}.
\]
Utilizing Equation~\ref{intromergesplit}, we have
\[
\begin{tikzpicture}[anchorbase,scale=1]
	\draw[-,line width=1.2pt] (0,0) to (.275,.3) to (.275,.7) to (0,1);
	\draw[-,line width=1.2pt] (.6,0) to (.315,.3) to (.315,.7) to (.6,1);
        \node at (0,1.13) {$\scriptstyle 1$};
        \node at (0.63,1.13) {$\scriptstyle 1$};
        \node at (0,-.1) {$\scriptstyle 1$};
        \node at (0.63,-.1) {$\scriptstyle 1$};
\end{tikzpicture} =
\begin{tikzpicture}[anchorbase,scale=.8]
\draw[-,line width=2pt] (0.08,-.8) to (0.08,-.5);
\draw[-,line width=2pt] (0.08,.3) to (0.08,.6);
\draw[-,line width=1pt] (0.08,-.51) to  (1.08,.31); 
\draw[-,line width=2pt] (1.08,-.8) to (1.08,-.5);
\draw[-,line width=2pt] (1.08,.3) to (1.08,.6);
\draw[-,line width=1pt] (1.08,-.51) to  (0.08,.31);
\node at (.33,.25) {$\scriptstyle 1$};
\node at (.77,.25) {$\scriptstyle 1$};
\end{tikzpicture}
+
\begin{tikzpicture}[anchorbase,scale=.8]
\draw[-,line width=2pt] (0.08,-.8) to (0.08,-.5);
\draw[-,line width=2pt] (0.08,.3) to (0.08,.6);
\draw[-,line width=1pt] (0.08,-.51) to  (0.08,.31);
\draw[-,line width=2pt] (1.08,-.8) to (1.08,-.5);
\draw[-,line width=2pt] (1.08,.3) to (1.08,.6);
\draw[-,line width=1pt] (1.08,-.51) to  (1.08,.31);
\node at (-.33,-.05) {$\scriptstyle 1$};
\node at (1.45,-.05) {$\scriptstyle 1$};
\end{tikzpicture}.
\]
Taking this substitution twice we get
\[
    \begin{tikzpicture}[anchorbase,scale=.8]
\draw[-,line width=2pt] (0.08,-.8) to (0.08,-.5);
\draw[-,line width=2pt] (0.08,.3) to (0.08,.6);
\draw[-,line width=1pt] (0.08,-.51) to  (1.08,.31); 
\draw[-,line width=2pt] (1.08,-.8) to (1.08,-.5);
\draw[-,line width=2pt] (1.08,.3) to (1.08,.6);
\draw[-,line width=1pt] (1.08,-.51) to  (0.08,.31);
\draw[-,line width=1pt] (0.08,-.51) to  (0.08,.31);
\draw[-,line width=1pt] (1.08,-.51) to  (1.08,.31);
\node at (-.33,-.05) {$\scriptstyle 1$};
\node at (1.45,-.05) {$\scriptstyle 1$};
\draw[-,line width=2pt] (0.08,1.4) to (0.08,1.7);
\draw[-,line width=1pt] (0.08,.59) to  (1.08,1.41); 
\draw[-,line width=2pt] (1.08,1.4) to (1.08,1.7);
\draw[-,line width=1pt] (1.08,.59) to  (0.08,1.41);
\draw[-,line width=1pt] (0.08,.59) to  (0.08,1.41);
\draw[-,line width=1pt] (1.08,.59) to  (1.08,1.41);
\node at (-.33,1.05) {$\scriptstyle 1$};
\node at (1.45,1.05) {$\scriptstyle 1$};
\node at (.33,1.35) {$\scriptstyle 1$};
\node at (.77,1.35) {$\scriptstyle 1$};
\end{tikzpicture} = 
\begin{tikzpicture}[anchorbase,scale=.8]
\draw[-,line width=2pt] (0.08,-.8) to (0.08,-.5);
\draw[-,line width=2pt] (0.08,.3) to (0.08,.6);
\draw[-,line width=1pt] (0.08,-.51) to [out=135, in=225] (0.08,.31);
\draw[-,line width=1pt] (0.08,-.51) to [out=45, in=-45] (0.08,.31);
\draw[-,line width=2pt] (1.08,-.8) to (1.08,-.5);
\draw[-,line width=2pt] (1.08,.3) to (1.08,.6);
\draw[-,line width=1pt] (1.08,-.51) to [out=135, in=225] (1.08,.31);
\draw[-,line width=1pt] (1.08,-.51) to [out=45, in=-45] (1.08,.31);
\node at (-.24,.05) {$\scriptstyle 1$};
\node at (.39,.05) {$\scriptstyle 1$};
\node at (1.35,.05) {$\scriptstyle 1$};
\node at (.73,.05) {$\scriptstyle 1$};
\end{tikzpicture}
+
\begin{tikzpicture}[anchorbase,scale=.8]
\draw[-,line width=2pt] (0.08,-.8) to (0.08,-.5);
\draw[-,line width=2pt] (0.08,.3) to (0.08,.6);
\draw[-,line width=1pt] (0.08,-.51) to [out=90, in=180] (1.08,.31);
\draw[-,line width=1pt] (0.08,-.51) to [out=0, in=270] (1.08,.31);
\draw[-,line width=2pt] (1.08,-.8) to (1.08,-.5);
\draw[-,line width=2pt] (1.08,.3) to (1.08,.6);
\draw[-,line width=1pt] (1.08,-.51) to [out=180, in=270] (0.08,.31);
\draw[-,line width=1pt] (1.08,-.51) to [out=90, in=0] (0.08,.31);
\node at (-.24,.15) {$\scriptstyle 1$};
\node at (.39,.4) {$\scriptstyle 1$};
\node at (1.35,.15) {$\scriptstyle 1$};
\node at (.73,.4) {$\scriptstyle 1$};
\end{tikzpicture}
+2
\begin{tikzpicture}[anchorbase,scale=.8]
\draw[-,line width=2pt] (0.08,-.8) to (0.08,-.5);
\draw[-,line width=2pt] (0.08,.3) to (0.08,.6);
\draw[-,line width=1pt] (0.08,-.51) to  (1.08,.31); 
\draw[-,line width=2pt] (1.08,-.8) to (1.08,-.5);
\draw[-,line width=2pt] (1.08,.3) to (1.08,.6);
\draw[-,line width=1pt] (1.08,-.51) to  (0.08,.31);
\draw[-,line width=1pt] (0.08,-.51) to  (0.08,.31);
\draw[-,line width=1pt] (1.08,-.51) to  (1.08,.31);
\node at (-.33,-.05) {$\scriptstyle 1$};
\node at (1.45,-.05) {$\scriptstyle 1$};
\node at (.33,.25) {$\scriptstyle 1$};
\node at (.77,.25) {$\scriptstyle 1$};
\end{tikzpicture}
. \]
Using Equation~\ref{introtrivial} and Equation~\ref{introsliders} this becomes
\[   \begin{tikzpicture}[anchorbase,scale=.8]
\draw[-,line width=2pt] (0.08,-.8) to (0.08,-.5);
\draw[-,line width=2pt] (0.08,.3) to (0.08,.6);
\draw[-,line width=1pt] (0.08,-.51) to  (1.08,.31); 
\draw[-,line width=2pt] (1.08,-.8) to (1.08,-.5);
\draw[-,line width=2pt] (1.08,.3) to (1.08,.6);
\draw[-,line width=1pt] (1.08,-.51) to  (0.08,.31);
\draw[-,line width=1pt] (0.08,-.51) to  (0.08,.31);
\draw[-,line width=1pt] (1.08,-.51) to  (1.08,.31);
\node at (-.33,-.05) {$\scriptstyle 1$};
\node at (1.45,-.05) {$\scriptstyle 1$};
\draw[-,line width=2pt] (0.08,1.4) to (0.08,1.7);
\draw[-,line width=1pt] (0.08,.59) to  (1.08,1.41); 
\draw[-,line width=2pt] (1.08,1.4) to (1.08,1.7);
\draw[-,line width=1pt] (1.08,.59) to  (0.08,1.41);
\draw[-,line width=1pt] (0.08,.59) to  (0.08,1.41);
\draw[-,line width=1pt] (1.08,.59) to  (1.08,1.41);
\node at (-.33,1.05) {$\scriptstyle 1$};
\node at (1.45,1.05) {$\scriptstyle 1$};
\node at (.33,1.35) {$\scriptstyle 1$};
\node at (.77,1.35) {$\scriptstyle 1$};
\end{tikzpicture} = 4\begin{tikzpicture}[anchorbase,scale=.8]
\draw[-,line width=2pt] (0.08,-.8) to (0.08,-.5);
\draw[-,line width=2pt] (0.08,.3) to (0.08,.6);
\draw[-,line width=1pt] (0.08,-.51) to  (0.08,.31);
\draw[-,line width=2pt] (1.08,-.8) to (1.08,-.5);
\draw[-,line width=2pt] (1.08,.3) to (1.08,.6);
\draw[-,line width=1pt] (1.08,-.51) to  (1.08,.31);
\node at (-.33,-.05) {$\scriptstyle 2$};
\node at (1.45,-.05) {$\scriptstyle 2$};
\end{tikzpicture}
+
\begin{tikzpicture}[anchorbase,scale=.8]
\draw[-,line width=2pt] (0.08,-.8) to (0.08,-.5);
\draw[-,line width=2pt] (0.08,.3) to (0.08,.6);
\draw[-,line width=1pt] (0.08,-.1) to (1.08,.31);
\draw[-,line width=1pt] (0.08,-.51) to [out=135, in=225] (.08,-.1);
\draw[-,line width=1pt] (0.08,-.51) to [out=45, in=315] (.08,-.1);
\draw[-,line width=2pt] (1.08,-.8) to (1.08,-.5);
\draw[-,line width=2pt] (1.08,.3) to (1.08,.6);
\draw[-,line width=1pt] (1.08,-.1) to (.08,.31);
\draw[-,line width=1pt] (1.08,-.51) to [out=135, in=225] (1.08,-.1);
\draw[-,line width=1pt] (1.08,-.51) to [out=45, in=315] (1.08,-.1);
\node at (-.24,-.2) {$\scriptstyle 1$};
\node at (.39,.4) {$\scriptstyle 2$};
\node at (1.35,-.2) {$\scriptstyle 1$};
\node at (.73,.4) {$\scriptstyle 2$};
\end{tikzpicture}
+2
\begin{tikzpicture}[anchorbase,scale=.8]
\draw[-,line width=2pt] (0.08,-.8) to (0.08,-.5);
\draw[-,line width=2pt] (0.08,.3) to (0.08,.6);
\draw[-,line width=1pt] (0.08,-.51) to  (1.08,.31); 
\draw[-,line width=2pt] (1.08,-.8) to (1.08,-.5);
\draw[-,line width=2pt] (1.08,.3) to (1.08,.6);
\draw[-,line width=1pt] (1.08,-.51) to  (0.08,.31);
\draw[-,line width=1pt] (0.08,-.51) to  (0.08,.31);
\draw[-,line width=1pt] (1.08,-.51) to  (1.08,.31);
\node at (-.33,-.05) {$\scriptstyle 1$};
\node at (1.45,-.05) {$\scriptstyle 1$};
\node at (.33,.25) {$\scriptstyle 2$};
\node at (.77,.25) {$\scriptstyle 2$};
\end{tikzpicture}
, \]
\[   \begin{tikzpicture}[anchorbase,scale=.8]
\draw[-,line width=2pt] (0.08,-.8) to (0.08,-.5);
\draw[-,line width=2pt] (0.08,.3) to (0.08,.6);
\draw[-,line width=1pt] (0.08,-.51) to  (1.08,.31); 
\draw[-,line width=2pt] (1.08,-.8) to (1.08,-.5);
\draw[-,line width=2pt] (1.08,.3) to (1.08,.6);
\draw[-,line width=1pt] (1.08,-.51) to  (0.08,.31);
\draw[-,line width=1pt] (0.08,-.51) to  (0.08,.31);
\draw[-,line width=1pt] (1.08,-.51) to  (1.08,.31);
\node at (-.33,-.05) {$\scriptstyle 1$};
\node at (1.45,-.05) {$\scriptstyle 1$};
\draw[-,line width=2pt] (0.08,1.4) to (0.08,1.7);
\draw[-,line width=1pt] (0.08,.59) to  (1.08,1.41); 
\draw[-,line width=2pt] (1.08,1.4) to (1.08,1.7);
\draw[-,line width=1pt] (1.08,.59) to  (0.08,1.41);
\draw[-,line width=1pt] (0.08,.59) to  (0.08,1.41);
\draw[-,line width=1pt] (1.08,.59) to  (1.08,1.41);
\node at (-.33,1.05) {$\scriptstyle 1$};
\node at (1.45,1.05) {$\scriptstyle 1$};
\node at (.33,1.35) {$\scriptstyle 1$};
\node at (.77,1.35) {$\scriptstyle 1$};
\end{tikzpicture} = 4\begin{tikzpicture}[anchorbase,scale=.8]
\draw[-,line width=2pt] (0.08,-.8) to (0.08,-.5);
\draw[-,line width=2pt] (0.08,.3) to (0.08,.6);
\draw[-,line width=1pt] (0.08,-.51) to  (0.08,.31);
\draw[-,line width=2pt] (1.08,-.8) to (1.08,-.5);
\draw[-,line width=2pt] (1.08,.3) to (1.08,.6);
\draw[-,line width=1pt] (1.08,-.51) to  (1.08,.31);
\node at (-.33,-.05) {$\scriptstyle 2$};
\node at (1.45,-.05) {$\scriptstyle 2$};
\end{tikzpicture}
+4
\begin{tikzpicture}[anchorbase,scale=.8]
\draw[-,line width=2pt] (0.08,-.8) to (0.08,-.5);
\draw[-,line width=2pt] (0.08,.3) to (0.08,.6);
\draw[-,line width=1pt] (0.08,-.51) to  (1.08,.31); 
\draw[-,line width=2pt] (1.08,-.8) to (1.08,-.5);
\draw[-,line width=2pt] (1.08,.3) to (1.08,.6);
\draw[-,line width=1pt] (1.08,-.51) to  (0.08,.31);
\node at (-.33,-.05) {$\scriptstyle 2$};
\node at (1.45,-.05) {$\scriptstyle 2$};
\end{tikzpicture}
+2
\begin{tikzpicture}[anchorbase,scale=.8]
\draw[-,line width=2pt] (0.08,-.8) to (0.08,-.5);
\draw[-,line width=2pt] (0.08,.3) to (0.08,.6);
\draw[-,line width=1pt] (0.08,-.51) to  (1.08,.31); 
\draw[-,line width=2pt] (1.08,-.8) to (1.08,-.5);
\draw[-,line width=2pt] (1.08,.3) to (1.08,.6);
\draw[-,line width=1pt] (1.08,-.51) to  (0.08,.31);
\draw[-,line width=1pt] (0.08,-.51) to  (0.08,.31);
\draw[-,line width=1pt] (1.08,-.51) to  (1.08,.31);
\node at (-.33,-.05) {$\scriptstyle 1$};
\node at (1.45,-.05) {$\scriptstyle 1$};
\node at (.33,.25) {$\scriptstyle 1$};
\node at (.77,.25) {$\scriptstyle 1$};
\end{tikzpicture}
,\]
\[D_3\cdot D_3 = 4D_1+4D_2+2D_3.\]
Observe that this is identical to Equation~\ref{matrixex}. In general it is a result of~\cite{Brundan} that $\lambda\times \mu$ reduced chicken foot diagrams have a direct correspondence with basis matrices of $\Hom_{S_n}(\mathbb{F}\left[S_n/S_{\mu}\right], \mathbb{F}\left[S_n/S_{\lambda}\right])$ and their multiplicative structure within the Schur algebra. 
Additionally, observe that in our computation of web diagram $D_3\cdot D_3$, which was neither chicken foot nor reduced, we first used Equation~\ref{intromergesplit} to write it as a sum of chicken foot diagrams, and then used Equation~\ref{introtrivial} and Equation~\ref{introsliders} to write these chicken foot diagrams in terms of reduced chicken foot diagrams.

\section{Definitions for the Hyperoctahedral Group}\label{sec:defsforhyper}
\subsection{Basic Definitions for the Hyperoctahedral Schur Algebra and Category}
First of all, the hyperoctahedral Schur algebra is defined in a similar fashion to the Schur algbera.
\begin{definition}
    For $\lambda$ a hypercomposition of $H_n$, we define a $H_n$-module $M_\lambda \coloneqq \mathbb{F}[H_n/H_\lambda]$.
\end{definition}
\begin{definition}
    The \emph{hyperoctahedral Schur algebra} of degree $n$ is $\End_{H_n}(\bigoplus_{\lambda}M_{\lambda})$ where $\lambda$ ranges over all hypercompositions of $H_n$.
\end{definition}
As with the Schur algebra, if we seek to find a basis of \[
\Hom_{H_n}(\mathbb{F}[H_n/H_\mu], \mathbb{F}[H_n/H_\lambda]) \cong \Hom_{H_n}(\mathbb{F}[I_\mu], \mathbb{F}[I_\lambda]),
\]
by Proposition~\ref{basiscross} it suffices to find the set of orbits of $H_n$ on $I_\lambda \times I_\mu$. Due to the description of the action of $H_n$ on $I_\lambda$ and $I_\mu$, we can categorize the set of orbits as follows.
\begin{definition}
For hypercompositions $\lambda, \mu$ of $H_n$, $\HMat_{\lambda, \mu}$ is defined as the set of $\ell(\lambda) \times \ell(\mu)$ matrices with non-negative integer coefficients such that the sum of the terms of the $i$-th row is $\lambda_i$, the sum of the terms of the $j$-th column is $\mu_j$, and the matrix exhibits a $180^\circ$ rotational symmetry (i.e $a_{i, j} = a_{\ell(\lambda)+1-i, \ell(\mu) + 1-j}$).
\end{definition}
\begin{remark}
    For the symmetric group, $\Mat_{\lambda, \mu}$ was similarly defined in~\cite{Brundan} for $\mu, \lambda$ compositions of $S_n$ except that the rotational symmetry condition was unneeded.
\end{remark}
\begin{definition}
For a matrix $A\in \HMat_{\lambda, \mu}$, we define 
\[
\Pi_A = \{(\bi, \bj)\in I_{\lambda}\times I_{\mu} \ | \{\# k | (\bi_k, \bj_k) = (i, j)\} = a_{i,j}\} .
\]
\end{definition}
\begin{theorem} \label{horbits}
    The set of orbits of $H_n$ on $I_\lambda \times I_\mu$ is precisely $\{\Pi_A| A\in \HMat_{\lambda, \mu}\}$.
\end{theorem}
\begin{proof}
    The proof will be broken down into two parts. We will first show that for all $h\in H_n$, $h\cdot \Pi_A = \Pi_A$ and then prove that any two $(\bi, \bj), (\bi', \bj') \in \Pi_A$ lie in the same orbit.
    First off, note that we have
    \[h\cdot ((\bi_1, \ldots, \bi_{2n}), (\bj_1, \ldots, \bj_{2n})) =((\bi_{h^{-1}(1)}, \ldots, \bi_{h^{-1}(2n)}), (\bj_{h^{-1}(1)}, \ldots, \bj_{h^{-1}(2n)})) .\]
    In particular, the two multisets \[\{(\bi_{1}, \bj_{1}), \ldots, (\bi_{2n}, \bj_{2n})\},\quad \quad \{(\bi_{h^{-1}(1)}, \bj_{h^{-1}(1)}), \ldots,(\bi_{h^{-1}(2n)}, \bj_{h^{-1}(2n)})\}\] are the same as $h^{-1}$ merely permute the indexes from $1$ to $2n$. Therefore $h\cdot (\bi, \bj)\in \Pi_A$ so $h\cdot \Pi_A = \Pi_A$.
    Next,  if we have $(\bi, \bj), (\bi', \bj') \in \Pi_A$ then we know that exists a permutation $h\in S_{2n}$ sending $(\bi_1, \bj_1), \ldots, (\bi_{2n}, \bj_{2n})$ to $(\bi'_1, \bj'_1), \ldots, (\bi'_{2n}, \bj'_{2n})$ in some order as for all $i,j$, each list has $a_{i,j}$ $(i, j)$ pairs. In particular, as $(\bi_k, \bj_k) =(\ell(\lambda)+1-\bi_{2n+1-k}, \ell(\mu)+1-\bj_{2n+1-k})$ and analogous for $(\bi', \bj')$, we can construct $h$ such that if $h$ sends $(i_k, j_k)$ to $(\bi'_t, \bj'_t)$ then $h$ sends $(i_{2n+1-k}, j_{2n+1-k})$ to $(\bi'_{2n+1-t}, \bj'_{2n+1-t})$. This means that $h\in H_n$ and gives $h\cdot (\bi, \bj) = (\bi', \bj')$ as needed.
\end{proof}
Because of the above theorem, every matrix $A\in \HMat_{\lambda, \mu}$ corresponds to a homomorphism. In particular we have the following.
\begin{definition}
    We define $\xi_A \in \Hom_{H_n}(\mathbb{F}[I_\mu], \mathbb{F}[I_\lambda])$ as the homomorphism sending $2n$-tuple $\bj$ to $\sum\limits_{(\bi, \bj)\in \Pi_A}\bi$, and $E_A$ as the corresponding $|I_\lambda| \times |I_\mu|$ matrix for $\xi_A$.
\end{definition}
\begin{remark}
    Instead of indexing the coefficients of $E_A$ by $\{1, \ldots, |I_\lambda|\}\times \{1, \ldots, |I_\mu|\}$, we will simply use the $2n$-tuples of $I_\lambda, I_\mu$. By Equation~\ref{basiscross} and Theorem~\ref{horbits}, the set of $\xi_A$'s form a basis of $\Hom_{H_n}(\mathbb{F}[I_\mu], \mathbb{F}[I_\lambda])$. 
\end{remark}
To determine how the composition of homomorphisms can be represented, we have the following theorem which gives us a precise rule.

\begin{theorem} \label{hschurrule}
For hypercompositions $\lambda,\mu,\nu$ of $H_n$,
$A \in \HMat_{\lambda,\mu}, B \in \HMat_{\mu,\nu}$ we have
\begin{equation}\label{hschurs}
\xi_A \circ \xi_B = 
\sum_{C \in \HMat_{\lambda,\nu}} Z\left(A,B,C\right)
\xi_C
\end{equation}
where
\begin{equation}\label{Z}
Z\left(A,B,C\right) \coloneqq \#\!\left\{\bj \:\big|\:(\bi,\bj) \in \Pi_A\text{ and } (\bj,\bk) \in \Pi_B
\right\}
\end{equation}
for $(\bi,\bk)$ a fixed element of $\Pi_C$.
\end{theorem}
\begin{proof}
First of all, as $\xi_A \circ \xi_B \in \Hom_{H_n}(\mathbb{F}[I_\nu], \mathbb{F}[I_\lambda])$, it is expressible as a linear combination of $\xi_C$'s for $C \in \HMat_{\lambda,\nu}$. As for the needed coefficients of such linear combination, observe that by standard matrix multiplication, the coefficients $c_{\bi, \bk}$ of $E_A\cdot E_B$ can be written as \[c_{\bi,\bk} = \sum_{\bj}a_{\bi,\bj}\cdot b_{\bj,\bk}.\]
For every $(\bi, \bk)$, there exists exactly one matrix $C\in \HMat_{\lambda,\nu}$ such that $(\bi, \bk) \in \Pi_C$. In Iverson bracket notation we have $a_{\bi,\bj} = [(\bi, \bj)\in \Pi_A]$ and $b_{\bj,\bk} = [(\bj, \bk)\in \Pi_B]$\footnote{For a statement P we have $[\text{P}] =\begin{cases} 
      1 & \text{if P} \\
      0 & \text{otherwise.} 
   \end{cases}.$} so
\[c_{\bi,\bk} = \sum_{\bj}[(\bi, \bj)\in \Pi_A]\cdot [(\bj, \bk)\in \Pi_B] = \sum_{\bj}[(\bi, \bj)\in \Pi_A, (\bj, \bk)\in \Pi_B] = Z\left(A,B,C\right).\]
 In particular $E_C$ has a coefficient of $1$ in the $\bi$-th row, $\bk$-th column, while all other $E_D$ for $D\in \HMat_{\lambda,\nu}, D\neq C$ have a coefficient of $0$. Because of this analysis we have
\[E_A \cdot E_B  = 
\sum_{C \in \HMat_{\lambda,\nu}} Z\left(A,B,C\right)
E_C,\]
\[\xi_A \circ \xi_B = 
\sum_{C \in \HMat_{\lambda,\nu}} Z\left(A,B,C\right)
\xi_C.\]
\end{proof}
We also have the following lemma akin to~\cite[Lemma 4.1]{Brundan} but for the hyperoctahedral group.
\begin{lemma}\label{htricky}
Suppose that $(\bi,\bj) \in \Pi_A$ and $(\bj,\bk) \in \Pi_B$ satisfy \[\Stab_{H_n}(\bi) \cap \Stab_{H_n}(\bk) = \Stab_{H_n}(\bj).\]
Then $Z\left(A,B,C\right) = 1$ if $(\bi,\bk) \in \Pi_C$, and $Z\left(A,B,C\right) = 0$ otherwise.
\end{lemma}
\begin{proof}
Consider a $(\bi',\bk') \in \Pi_C$. To compute $Z\left(A,B,C\right)$, we need to find the number of 
$\bj'$ such that $(\bi',\bj') \in \Pi_A$ and
$(\bj',\bk') \in \Pi_B$.
In other words, we need to find the number of $\bj'$
with $(\bi',\bj') \sim (\bi,\bj)$ and $(\bj',\bk') \sim
(\bj,\bk)$. If a $\bj'$ exists, we can find $g \in H_n$ such that $g\cdot \bj'=
\bj$, so $(g\cdot \bi', \bj) \sim (\bi,\bj)$ and $(\bj,
g \cdot \bk') \sim (\bj,\bk)$. This means that there exists $h \in \Stab_{H_n}(\bj)$ such that
$g\cdot \bi' = h\cdot \bi$ and $g\cdot \bk' = h \cdot \bk$. However, as
$\Stab_{H_n}(\bj) \subseteq \Stab_{H_n}(\bi) \cap \Stab_{H_n}(\bk)$ we get $g\cdot \bi' = \bi$ and $g\cdot \bk' = \bk$, so
$(\bi,\bk) \in \Pi_C$.

Finally if we have $(\bi,\bk) \in \Pi_C$ then we can assume $(\bi',\bk')= (\bi,\bk)$, 
and $Z\left(A,B,C\right)$ is the number of $\bj'$ such that
$(\bi,\bj') \sim (\bi,\bj)$ and $(\bj',\bk) \sim (\bj,\bk)$. Since $\bi$ and $\bk$ are fixed, we see that $\bj'$ can be written in the form $g\cdot \bj$ for $g \in
\Stab_{H_n}(\bi) \cap \Stab_{H_n}(\bk)$. However,  $\Stab_{H_n}(\bi) \cap \Stab_{H_n}(\bk) \subseteq
\Stab_{H_n}(\bj)$ implies $\bj' = \bj$ and $Z\left(A,B,C\right) = 1$.
\end{proof}

As an example, consider endomorphisms of $H_2/H_{1,2,1}$ that commute with $H_2$. For our crossing diagrams of Section~\ref{H2Reps}, labelling these in order from $1$ to $4$ we can find the three orbits by inspection.
\begin{center}
    Orbit 1: $(1,1), (2,2), (3,3), (4,4)$.
    \\*
    Orbit 2: $(1,4), (2,3), (3,2), (4,1)$.
    \\*
    Orbit 3: Everything else.
\end{center}
These give rise to the following three matrices which are a basis of $\End_{H_2}(\mathbb{F}[H_2/H_{1,2,1}])$:
\begin{center}
\[E_1 = \begin{bmatrix}
1 & 0 & 0 & 0  \\
0 & 1 & 0 & 0  \\
0 & 0 & 1 & 0  \\
0 & 0 & 0 & 1  
\end{bmatrix},\]
\[E_2 = \begin{bmatrix}
0 & 0 & 0 & 1  \\
0 & 0 & 1 & 0  \\
0 & 1 & 0 & 0  \\
1 & 0 & 0 & 0  
\end{bmatrix},\]
\[E_3 = \begin{bmatrix}
0 & 1 & 1 & 0  \\
1 & 0 & 0 & 1  \\
1 & 0 & 0 & 1  \\
0 & 1 & 1 & 0 
\end{bmatrix}.\]
\end{center}
These correspond to $\xi_{\left(\begin{smallmatrix}1&0&0\\0&2&0\\0&0&1\end{smallmatrix}\right)}$, $\xi_{\left(\begin{smallmatrix}0&0&1\\0&2&0\\1&0&0\end{smallmatrix}\right)}$, and  $\xi_{\left(\begin{smallmatrix}0&1&0\\1&0&1\\0&1&0\end{smallmatrix}\right)}$ respectively.

From the hyperoctahedral Schur algebra we can define a more general category akin to the Schur category.
\begin{definition}
    The \emph{hyperoctahedral Schur category}, denoted as $\HSchur$, is the category with objects: hypercompositions and morphisms: 
    \[
    \Hom_{\HSchur}(\mu, \lambda):=\begin{cases}\Hom_{H_n}(M_{\mu}, M_\lambda) \quad \text{if} \quad \sum \lambda_i = \sum \mu_i\\
    0 \quad \text{else}.
    \end{cases}.
    \]
\end{definition}
\subsection{Definition of Hyperoctahedral Web Category}\label{sec:hgendef}
As with how the polynomial web category utilized web diagrams, we can likewises define the hyperoctahedral web category.
\begin{definition}
The \emph{hyperoctahedral web category}, denoted as $\HWeb$, has objects as hypercompositions and morphisms from hypercompositions $\mu$ to $\lambda$ as web diagrams from $\mu$ to $\lambda$. We have the following web diagrams as generators:
\[
\begin{tikzpicture}
[baseline = -.5mm]
	\draw[-,line width=1pt] (0.28,-.3) to (0.08,0.04);
	\draw[-,line width=1pt] (-0.12,-.3) to (0.08,0.04);
	\draw[-,line width=2pt] (0.08,.4) to (0.08,0);
        \node at (-0.22,-.4) {$\scriptstyle a$};
        \node at (0.35,-.4) {$\scriptstyle b$};
\end{tikzpicture} : (a,b,b,a) \rightarrow (a+b,a+b),
\begin{tikzpicture}[baseline = -.5mm]
 \draw[-,line width=1pt,dashed,lightblue] (0.08,-.75) to (0.08,0.7);
	\draw[-,line width=1pt] (0.28,-.3) to (0.08,0.04);
	\draw[-,line width=1pt] (-0.12,-.3) to (0.08,0.04);
    \draw[-,line width=1pt] (0.08,-.3) to (0.08,0.04);
    \draw[-,line width=2pt] (0.08,.4) to (0.08,0);
        \node at (-0.22,-.4) {$\scriptstyle a$};
        \node at (0.07,-.4) {$\scriptstyle 2b$};
        \node at (0.35,-.4) {$\scriptstyle a$};
\end{tikzpicture} 
:(a,2b,a) \rightarrow (2a+2b),
\]
\[
\begin{tikzpicture}[baseline = -.5mm]
	\draw[-,line width=2pt] (0.08,-.3) to (0.08,0.04);
	\draw[-,line width=1pt] (0.28,.4) to (0.08,0);
	\draw[-,line width=1pt] (-0.12,.4) to (0.08,0);
        \node at (-0.22,.5) {$\scriptstyle a$};
        \node at (0.36,.5) {$\scriptstyle b$};
\end{tikzpicture}
:(a+b, a+b)\rightarrow (a,b, b, a), 
\begin{tikzpicture}[baseline = -.5mm]
   \draw[-,line width=1pt,dashed,lightblue] (0.08,-.7) to (0.08,0.75);
	\draw[-,line width=2pt] (0.08,-.3) to (0.08,0.04);
	\draw[-,line width=1pt] (0.28,.4) to (0.08,0);
	\draw[-,line width=1pt] (-0.12,.4) to (0.08,0);
 \draw[-,line width=1pt] (0.08,.4) to   (0.08,0);
        \node at (-0.22,.5) {$\scriptstyle a$};
        \node at (0.36,.5) {$\scriptstyle a$};
         \node at (0.07,.5) {$\scriptstyle 2b$};
\end{tikzpicture}
:(2a+2b)\rightarrow (a,2b,a),
\]
\[
\begin{tikzpicture}[baseline=-.5mm]
	\draw[-,thick] (-0.3,-.3) to (.3,.4);
	\draw[-,thick] (0.3,-.3) to (-.3,.4);
        \node at (0.3,-.4) {$\scriptstyle b$};
        \node at (-0.3,-.4) {$\scriptstyle a$};
\end{tikzpicture}
:(a,b, b, a) \rightarrow (b,a, a, b),
\begin{tikzpicture}[baseline=-.5mm]
  \draw[-,line width=1pt,dashed,lightblue] (0,-.5) to (0,0.6);
	\draw[-,thick] (-0.3,-.3) to (.3,.4);
	\draw[-,thick] (0.3,-.3) to (-.3,.4);
    \draw[-,thick] (0,-.3) to (0,.4);
        \node at (0.3,-.4) {$\scriptstyle a$};
        \node at (0,-.4) {$\scriptstyle 2b$};
        \node at (-0.3,-.4) {$\scriptstyle a$};
\end{tikzpicture}
:(a,2b, a) \rightarrow (a,2b, a).
\]
We also have the following relations, which from here on out are only presented up to horizontal and vertical symmetry:
\begin{align}
\label{hsplitchoice}
\begin{tikzpicture}[anchorbase]
	\draw[-,thick] (0.35,-.3) to (0.08,0.14);
	\draw[-,thick] (0.1,-.3) to (-0.04,-0.06);
	\draw[-,line width=1pt] (0.085,.14) to (-0.035,-0.06);
	\draw[-,thick] (-0.2,-.3) to (0.07,0.14);
	\draw[-,line width=2pt] (0.08,.45) to (0.08,.1);
        \node at (0.45,-.41) {$\scriptstyle c$};
        \node at (0.07,-.4) {$\scriptstyle b$};
        \node at (-0.28,-.41) {$\scriptstyle a$};
\end{tikzpicture}
=
\begin{tikzpicture}[anchorbase]
	\draw[-,thick] (0.36,-.3) to (0.09,0.14);
	\draw[-,thick] (0.06,-.3) to (0.2,-.05);
	\draw[-,line width=1pt] (0.07,.14) to (0.19,-.06);
	\draw[-,thick] (-0.19,-.3) to (0.08,0.14);
	\draw[-,line width=2pt] (0.08,.45) to (0.08,.1);
        \node at (0.45,-.41) {$\scriptstyle c$};
        \node at (0.07,-.4) {$\scriptstyle b$};
        \node at (-0.28,-.41) {$\scriptstyle a$};
\end{tikzpicture},\qquad \begin{tikzpicture}[anchorbase]
    \draw[-,line width=1pt,dashed,lightblue] (0,-.7) to (0,.7);
	\draw[-,thick] (-0.16,-.3) to (-0.22,-0.21);
	\draw[-,line width=1pt] (0,.14) to (-0.22,-0.21);
	\draw[-,thick] (-0.28,-.3) to (0,0.14);
 \draw[-,thick] (0.16,-.3) to (0.22,-0.21);
	\draw[-,line width=1pt] (0,.14) to (0.22,-0.21);
	\draw[-,thick] (0.28,-.3) to (0,0.14);
	\draw[-,line width=2pt] (0,.45) to (0,.1);
    \draw[-,line width=1pt] (0,.1) to  (0,-.3);
        \node at (0,-.41) {$\scriptstyle c$};
        \node at (-0.15,-.4) {$\scriptstyle b$};
        \node at (-0.33,-.41) {$\scriptstyle a$};
          \node at (0.15,-.4) {$\scriptstyle b$};
        \node at (0.33,-.41) {$\scriptstyle a$};
\end{tikzpicture}
 = \begin{tikzpicture}[anchorbase]
    \draw[-,line width=1pt,dashed,lightblue] (0,-.7) to (0,.7);
	\draw[-,thick] (-0.16,-.3) to (0,-0.06);
	\draw[-,line width=1pt] (0,.14) to (-0.22,-0.21);
	\draw[-,thick] (-0.28,-.3) to (0,0.14);
    \draw[-,thick] (0.16,-.3) to (0,-0.06);
	\draw[-,line width=1pt] (0,.14) to (0.22,-0.21);
	\draw[-,thick] (0.28,-.3) to (0,0.14);
	\draw[-,line width=2pt] (0,.45) to (0,.1);
    \draw[-,line width=1pt] (0,.1) to  (0,-.3);
        \node at (0,-.41) {$\scriptstyle c$};
        \node at (-0.15,-.4) {$\scriptstyle b$};
        \node at (-0.33,-.41) {$\scriptstyle a$};
          \node at (0.15,-.4) {$\scriptstyle b$};
        \node at (0.33,-.41) {$\scriptstyle a$};
\end{tikzpicture}
,\end{align}
\begin{align}
\label{htrivial}
\begin{tikzpicture}[anchorbase,scale=.8]
	\draw[-,line width=2pt] (0.08,-.8) to (0.08,-.5);
	\draw[-,line width=2pt] (0.08,.3) to (0.08,.6);
\draw[-,thick] (0.1,-.51) to [out=45,in=-45] (0.1,.31);
\draw[-,thick] (0.06,-.51) to [out=135,in=-135] (0.06,.31);
        \node at (-.33,-.05) {$\scriptstyle a$};
        \node at (.45,-.05) {$\scriptstyle b$};
\end{tikzpicture}
&= 
\binom{a+b}{a}\:\:
\begin{tikzpicture}[anchorbase,scale=.8]
	\draw[-,line width=2pt] (0.08,-.8) to (0.08,.6);
        \node at (.62,-.05) {$\scriptstyle a+b$};
\end{tikzpicture},\begin{tikzpicture}[anchorbase,scale=.8]
 \draw[-,line width=1pt,dashed,lightblue] (0.08,-1) to (0.08,1);
	\draw[-,line width=2pt] (0.08,-.8) to (0.08,-.5);
	\draw[-,line width=2pt] (0.08,.3) to (0.08,.6);
 \draw[-,line width=1pt] (0.08,-.5) to (0.08,.3);
\draw[-,thick] (0.1,-.51) to [out=0,in=0] (0.1,.31);
\draw[-,thick] (0.06,-.51) to [out=180,in=-180] (0.06,.31);
        \node at (-.33,-.05) {$\scriptstyle a$};
        \node at (-0.05,-.05) {$\scaleto{2b}{3pt}$};
        \node at (.45,-.05) {$\scriptstyle a$};
\end{tikzpicture}
=2^a \binom{a+b}{a}\:\:
\begin{tikzpicture}[anchorbase,scale=.8]
   \draw[-,line width=1pt,dashed,lightblue] (0.08,-1) to (0.08,1);
	\draw[-,line width=2pt] (0.08,-.8) to (0.08,.6);
        \node at (.65,-.05) {$\scriptstyle 2a+2b$};
\end{tikzpicture}
,\end{align}
\label{mergesplit}
\begin{align}
\begin{tikzpicture}[anchorbase,scale=1]
	\draw[-,line width=1.2pt] (0,0) to (.275,.3) to (.275,.7) to (0,1);
	\draw[-,line width=1.2pt] (.6,0) to (.315,.3) to (.315,.7) to (.6,1);
        \node at (0,1.13) {$\scriptstyle b$};
        \node at (0.63,1.13) {$\scriptstyle d$};
        \node at (0,-.1) {$\scriptstyle a$};
        \node at (0.63,-.1) {$\scriptstyle c$};
\end{tikzpicture}
&=
\sum_{\substack{0 \leq s \leq \min(a,b)\\0 \leq t \leq \min(c,d)\\t-s=d-a}}
\begin{tikzpicture}[anchorbase,scale=1]
	\draw[-,thick] (0.58,0) to (0.58,.2) to (.02,.8) to (.02,1);
	\draw[-,thick] (0.02,0) to (0.02,.2) to (.58,.8) to (.58,1);
	\draw[-,thin] (0,0) to (0,1);
	\draw[-,thin] (0.6,0) to (0.6,1);
        \node at (0,1.13) {$\scriptstyle b$};
        \node at (0.6,1.13) {$\scriptstyle d$};
        \node at (0,-.1) {$\scriptstyle a$};
        \node at (0.6,-.1) {$\scriptstyle c$};
        \node at (-0.1,.5) {$\scriptstyle s$};
        \node at (0.77,.5) {$\scriptstyle t$};
\end{tikzpicture}, \label{hmergesplit}
\begin{tikzpicture}[anchorbase,scale=1]
  \draw[-,line width=1pt,dashed,lightblue] (0.3,-.5) to (0.3,1.5);
	\draw[-,line width=1.2pt] (0,0) to (.275,.3) to (.275,.7) to (0,1);
     \draw[-,line width=1.2pt] (.3,0) to (.295,.3) to (.295,.7) to (.3,1);
	\draw[-,line width=1.2pt] (.6,0) to (.315,.3) to (.315,.7) to (.6,1);
        \node at (0,1.13) {$\scriptstyle a$};
         \node at (0.31,1.13) {$\scriptstyle 2b$};
        \node at (0.63,1.13) {$\scriptstyle a$};
        \node at (0,-.1) {$\scriptstyle a$};
          \node at (0.31,-.1) {$\scriptstyle 2b$};
        \node at (0.63,-.1) {$\scriptstyle a$};
\end{tikzpicture}
= \sum_{\substack{ 0\leq s\leq a }} \sum_{\substack{ 0\leq t\leq \min(b, a-s)}}
\begin{tikzpicture}[anchorbase,scale=1]
 \draw[-,line width=1pt,dashed,lightblue] (0.3,-.5) to (0.3,1.5);
	\draw[-,thick] (0.6,0) to (0.6,.2) to (0,.8) to (0,1);
	\draw[-,thick] (0,0) to (0,.2) to (.6,.8) to (.6,1);
	\draw[-,thin] (0,0) to (0,1);
    \draw[-,thin] (0,0.2) to (0.3,.8);
     \draw[-,thin] (0.6,0.2) to (0.3,.8);
	\draw[-,thin] (0.6,0) to (0.6,1);
 \draw[-,thick] (0.3,0) to  (0.3,1);
 \draw[-,thin] (0,0.8) to (0.3,0.2);
  \draw[-,thin] (0.6,.8) to (0.3,0.2);
        \node at (0,1.13) {$\scriptstyle a$};
           \node at (0.3,1.13) {$\scriptstyle 2b$};
             \node at (0.3,-.1) {$\scriptstyle 2b$};
        \node at (0.6,1.13) {$\scriptstyle a$};
        \node at (0,-.1) {$\scriptstyle a$};
        \node at (0.6,-.1) {$\scriptstyle a$};
        \node at (-0.1,.5) {$\scriptstyle s$};
        \node at (0.16,.8) {$\scriptstyle t$};
         \node at (0.46,.8) {$\scriptstyle t$};
        \node at (0.7,.5) {$\scriptstyle s$};
\end{tikzpicture},
\end{align}
\end{definition}
\begin{align}\label{commute}
\begin{tikzpicture}[anchorbase]
\draw[-,thick] (0.2,0) to (.2,0.2);
 \draw[draw=black] (0,0.2) rectangle ++(0.4,0.4);
 \draw[-,thick] (0.2,.6) to (.2,1.2);
 \draw[-,thick] (0.8,0) to (.8,0.6);
 \draw[draw=black] (0.6,0.6) rectangle ++(0.4,0.4);
 \draw[-,thick] (0.8,1) to (.8,1.2);
\end{tikzpicture} = \begin{tikzpicture}[anchorbase]
\draw[-,thick] (0.2,0) to (.2,0.6);
 \draw[draw=black] (0,0.6) rectangle ++(0.4,0.4);
 \draw[-,thick] (0.2,1) to (.2,1.2);
 \draw[-,thick] (0.8,0) to (.8,0.2);
 \draw[draw=black] (0.6,0.2) rectangle ++(0.4,0.4);
 \draw[-,thick] (0.8,.6) to (.8,1.2);
\end{tikzpicture}
,
\end{align}
\begin{align}\label{hcommute}
\begin{tikzpicture}[anchorbase]
\draw[-,line width=1pt,dashed,lightblue] (0.8,-.3) to (0.8,1.5);
\draw[-,thick] (0.2,0) to (.2,0.2);
 \draw[draw=black] (0,0.2) rectangle ++(0.4,0.4);
 \draw[-,thick] (0.2,.6) to (.2,1.2);
 \draw[-,thick] (0.8,0) to (.8,0.6);
 \draw[draw=black] (0.6,0.6) rectangle ++(0.4,0.4);
 \draw[-,thick] (0.8,1) to (.8,1.2);
 \draw[-,thick] (1.4,0) to (1.4,0.2);
 \draw[draw=black] (1.2,0.2) rectangle ++(0.4,0.4);
 \draw[-,thick] (1.4,.6) to (1.4,1.2);
\end{tikzpicture} = \begin{tikzpicture}[anchorbase]
\draw[-,line width=1pt,dashed,lightblue] (0.8,-.3) to (0.8,1.5);
\draw[-,thick] (0.2,0) to (.2,0.6);
 \draw[draw=black] (0,0.6) rectangle ++(0.4,0.4);
 \draw[-,thick] (0.2,1) to (.2,1.2);
 \draw[-,thick] (0.8,0) to (.8,0.2);
 \draw[draw=black] (0.6,0.2) rectangle ++(0.4,0.4);
 \draw[-,thick] (0.8,.6) to (.8,1.2);
 \draw[-,thick] (1.4,0) to (1.4,0.6);
 \draw[draw=black] (1.2,0.6) rectangle ++(0.4,0.4);
 \draw[-,thick] (1.4,1) to (1.4,1.2);
\end{tikzpicture}
.
\end{align}
where the boxes represent arbitrary generating morphisms. 
\begin{remark}
    The dashed lines in relations indicate that the thick strings lie on the vertical line of symmetry. Relations without a dashed line are assumed to have their symmetric copy on the other side of the line, hence the input and output hypercompositions for generators without a dashed line have twice as many elements as there are input and output strings depicted. 
    
    For example, using Equation~\ref{htrivial} to reduce  $(6,0,6)\times(6,0,6)$ and $(3,6,3)\times (3,6,3)$ diagrams, we have the following equalities:
    \begin{center} \begin{tikzpicture}[anchorbase,scale=.8]
\draw[-,line width=1pt,dashed,lightblue] (0.58,-1.1) to (0.58,.9);
\draw[-,line width=2pt] (0.08,-.8) to (0.08,-.5);
\draw[-,line width=2pt] (0.08,.3) to (0.08,.6);
\draw[-,line width=1pt] (0.08,-.51) to [out=135, in=225] (0.08,.31);
\draw[-,line width=1pt] (0.08,-.51) to [out=45, in=-45] (0.08,.31);
\draw[-,line width=2pt] (1.08,-.8) to (1.08,-.5);
\draw[-,line width=2pt] (1.08,.3) to (1.08,.6);
\draw[-,line width=1pt] (1.08,-.51) to [out=135, in=225] (1.08,.31);
\draw[-,line width=1pt] (1.08,-.51) to [out=45, in=-45] (1.08,.31);
\node at (-.24,.05) {$\scriptstyle 3$};
\node at (.39,.05) {$\scriptstyle 3$};
\node at (1.35,.05) {$\scriptstyle 3$};
\node at (.73,.05) {$\scriptstyle 3$};
\end{tikzpicture}
 $= {\binom{6}{3}}\cdot$ \begin{tikzpicture}[anchorbase,scale=.8]
 \draw[-,line width=1pt,dashed,lightblue] (0.58,-1.1) to (0.58,.9);
\draw[-,line width=2pt] (0.08,-.8) to (0.08,-.5);
\draw[-,line width=2pt] (0.08,.3) to (0.08,.6);
\draw[-,line width=1pt] (0.08,-.51) to  (0.08,.31);
\draw[-,line width=2pt] (1.08,-.8) to (1.08,-.5);
\draw[-,line width=2pt] (1.08,.3) to (1.08,.6);
\draw[-,line width=1pt] (1.08,-.51) to  (1.08,.31);
\node at (-.33,-.05) {$\scriptstyle 6$};
\node at (1.45,-.05) {$\scriptstyle 6$};
\end{tikzpicture}, \quad \quad \begin{tikzpicture}[anchorbase,scale=.8]
\draw[-,line width=1pt,dashed,lightblue] (0.58,-1.1) to (0.58,.9);
\draw[-,line width=2pt] (0.08,-.8) to (0.08,-.5);
\draw[-,line width=2pt] (0.08,.3) to (0.08,.6);
\draw[-,line width=1pt] (0.08,-.51) to  (0.08,.31);
\draw[-,line width=2pt] (0.58,-.8) to (0.58,-.5);
\draw[-,line width=2pt] (0.58,.3) to (0.58,.6);
\draw[-,line width=1pt] (0.58,-.51) to [out=135, in=225] (0.58,.31);
\draw[-,line width=1pt] (0.58,-.51) to [out=45, in=-45] (0.58,.31);
\draw[-,line width=2pt] (1.08,-.8) to (1.08,-.5);
\draw[-,line width=2pt] (1.08,.3) to (1.08,.6);
\draw[-,line width=1pt] (1.08,-.51) to (1.08,.31);
\node at (-.24,.05) {$\scriptstyle 3$};
\node at (.24,.05) {$\scriptstyle 3$};
\node at (1.35,.05) {$\scriptstyle 3$};
\node at (.88,.05) {$\scriptstyle 3$};
\end{tikzpicture}
 $= 2^3\cdot$ \begin{tikzpicture}[anchorbase,scale=.8]
 \draw[-,line width=1pt,dashed,lightblue] (0.58,-1.1) to (0.58,.9);
\draw[-,line width=2pt] (0.08,-.8) to (0.08,-.5);
\draw[-,line width=2pt] (0.08,.3) to (0.08,.6);
\draw[-,line width=1pt] (0.08,-.51) to  (0.08,.31);
\draw[-,line width=2pt] (0.58,-.8) to (0.58,-.5);
\draw[-,line width=2pt] (0.58,.3) to (0.58,.6);
\draw[-,line width=1pt] (0.58,-.51) to (0.58,.31);
\draw[-,line width=2pt] (1.08,-.8) to (1.08,-.5);
\draw[-,line width=2pt] (1.08,.3) to (1.08,.6);
\draw[-,line width=1pt] (1.08,-.51) to (1.08,.31);
\node at (-.24,.05) {$\scriptstyle 3$};
\node at (.45,.05) {$\scriptstyle 6$};
\node at (1.35,.05) {$\scriptstyle 3$};
\end{tikzpicture}
.
\end{center}
Importantly, while $H_n$ is a subgroup of $S_{2n}$, $\HWeb$ is not a subcategory of the polynomial web category. 
\end{remark}
\begin{remark}
From~\ref{commute} and~\ref{hcommute}, it follows that these relations hold when the boxes represent any possible web diagrams. These relations are called~\emph{rectilinear isotopies} and, in $\Web$, they implicitly followed from $\Web$ being a strict monoidal category. However $\HWeb$ is not a strict monoidal category as, in general, putting two diagrams side by side (as opposed to on top of one another) breaks the needed vertical line of symmetry. Nevertheless we will frequently use~\ref{commute} and~\ref{hcommute} without reference due to their simplicity.
\end{remark}

As with the work done in~\cite[Equation 4.24]{Brundan}, Equation~\ref{hsplitchoice} allows us to define general 
$n$-fold merges and splits off and on the axis. For example, a three-fold merge off the axis (which is used in Equation~\ref{hmergesplit})  and a five-fold merge on the axis have the below definitions:
\[ \begin{tikzpicture}[baseline = 0]
	\draw[-,thick] (0.36,-.3) to (.08,0.14);
	\draw[-,thick] (0.08,-.3) to (.08,0.14);
	\draw[-,thick] (-0.2,-.3) to (.08,0.14);
	\draw[-,line width=2pt] (0.08,.45) to (0.08,.1);
        \node at (0.36,-.41) {$\scriptstyle c$};
        \node at (0.08,-.4) {$\scriptstyle b$};
        \node at (-0.2,-.41) {$\scriptstyle a$};
\end{tikzpicture}
\vcentcolon =
\begin{tikzpicture}[baseline = 0]
	\draw[-,thick] (0.36,-.3) to (.08,0.14);
	\draw[-,thick] (0.08,-.3) to (-.06,-.08);
	\draw[-,thick] (-0.2,-.3) to (.08,0.14);
	\draw[-,line width=2pt] (0.08,.45) to (0.08,.1);
        \node at (0.36,-.41) {$\scriptstyle c$};
        \node at (0.08,-.4) {$\scriptstyle b$};
        \node at (-0.2,-.41) {$\scriptstyle a$};
\end{tikzpicture}
= \begin{tikzpicture}[baseline = 0]
	\draw[-,thick] (0.36,-.3) to (.08,0.14);
	\draw[-,thick] (0.08,-.3) to (.22,-.08);
	\draw[-,thick] (-0.2,-.3) to (.08,0.14);
	\draw[-,line width=2pt] (0.08,.45) to (0.08,.1);
        \node at (0.36,-.41) {$\scriptstyle c$};
        \node at (0.08,-.4) {$\scriptstyle b$};
        \node at (-0.2,-.41) {$\scriptstyle a$};
\end{tikzpicture}, \]
\[ \begin{tikzpicture}[baseline = 0]
\draw[-,line width=1pt,dashed,lightblue] (0.07,-.55) to (0.07,.7);
	\draw[-,thick] (0.56,-.3) to (.08,0.24);
	\draw[-,thick] (-0.4,-.3) to (.08,0.24);
  	\draw[-,thick] (-0.16,-.3) to (.08,0.24);
    \draw[-,thick] (0.32,-.3) to (.08,0.24);
   	\draw[-,thick] (0.08,-.3) to (.08,0.24);
	\draw[-,line width=2pt] (0.08,.45) to (0.08,.2);
        \node at (0.56,-.41) {$\scriptstyle a$};
        \node at (-0.40,-.41) {$\scriptstyle a$};
        \node at (-.16,-.4) {$\scriptstyle b$};
        \node at (.32,-.4) {$\scriptstyle b$};
        \node at (0.08,-.4) {$\scriptstyle 2c$};
\end{tikzpicture}
\vcentcolon =
\begin{tikzpicture}[baseline = 0]
	\draw[-,line width=1pt,dashed,lightblue] (0.07,-.55) to (0.07,.7);
	\draw[-,thick] (0.56,-.3) to (.08,0.24);
	\draw[-,thick] (-0.4,-.3) to (.08,0.24);
  	\draw[-,thick] (-0.16,-.3) to (-.28,-.165);
    \draw[-,thick] (0.32,-.3) to (.44,-.165);
   	\draw[-,thick] (0.08,-.3) to (.08,0.24);
	\draw[-,line width=2pt] (0.08,.45) to (0.08,.2);
        \node at (0.56,-.41) {$\scriptstyle a$};
        \node at (-0.40,-.41) {$\scriptstyle a$};
        \node at (-.16,-.4) {$\scriptstyle b$};
        \node at (.32,-.4) {$\scriptstyle b$};
        \node at (0.08,-.4) {$\scriptstyle 2c$};
\end{tikzpicture}
= \begin{tikzpicture}[baseline = 0]
	\draw[-,line width=1pt,dashed,lightblue] (0.07,-.55) to (0.07,.7);
	\draw[-,thick] (0.56,-.3) to (.08,0.24);
	\draw[-,thick] (-0.4,-.3) to (.08,0.24);
  	\draw[-,thick] (-0.16,-.3) to (.08,-.03);
    \draw[-,thick] (0.32,-.3) to (.08,-.03);
   	\draw[-,thick] (0.08,-.3) to (.08,0.24);
	\draw[-,line width=2pt] (0.08,.45) to (0.08,.2);
        \node at (0.56,-.41) {$\scriptstyle a$};
        \node at (-0.40,-.41) {$\scriptstyle a$};
        \node at (-.16,-.4) {$\scriptstyle b$};
        \node at (.32,-.4) {$\scriptstyle b$};
        \node at (0.08,-.4) {$\scriptstyle 2c$};
\end{tikzpicture}.\]
From our defining relations, we can prove a variety of other relations displayed below. Note that the defining relations off the vertical axis are identical to those for the polynomial web category, so the below equations hold as relations in $\HWeb$. For further details, see~\cite[Appendix A]{Brundan} for how the following were proven in the polynomial web category: 
\begin{align}
\label{switch}
\begin{tikzpicture}[anchorbase,scale=1]
	\draw[-,line width=1.2pt] (0,0) to (.6,1);
	\draw[-,line width=1.2pt] (0,1) to (.6,0);
        \node at (0,-.1) {$\scriptstyle a$};
        \node at (0.6,-0.1) {$\scriptstyle b$};
\end{tikzpicture}
&=\sum_{t=0}^{\min(a,b)}
(-1)^t
\begin{tikzpicture}[anchorbase,scale=1]
	\draw[-,thick] (0,0) to (0,1);
	\draw[-,thick] (0.015,0) to (0.015,.2) to (.57,.4) to (.57,.6)
        to (.015,.8) to (.015,1);
	\draw[-,line width=1.2pt] (0.6,0) to (0.6,1);
        \node at (0.6,-.1) {$\scriptstyle b$};
        \node at (0,-.1) {$\scriptstyle a$};
        \node at (-0.1,.5) {$\scriptstyle t$};
\end{tikzpicture}
, \end{align}
\begin{align}
\label{swallows}
\begin{tikzpicture}[anchorbase,scale=.7]
	\draw[-,line width=2pt] (0.08,.3) to (0.08,.5);
\draw[-,thick] (-.2,-.8) to [out=45,in=-45] (0.1,.31);
\draw[-,thick] (.36,-.8) to [out=135,in=-135] (0.06,.31);
        \node at (-.3,-.95) {$\scriptstyle a$};
        \node at (.45,-.95) {$\scriptstyle b$};
\end{tikzpicture}
=\begin{tikzpicture}[anchorbase,scale=.7]
	\draw[-,line width=2pt] (0.08,.1) to (0.08,.5);
\draw[-,thick] (.46,-.8) to [out=100,in=-45] (0.1,.11);
\draw[-,thick] (-.3,-.8) to [out=80,in=-135] (0.06,.11);
        \node at (-.3,-.95) {$\scriptstyle a$};
        \node at (.43,-.95) {$\scriptstyle b$};
\end{tikzpicture}
,\end{align}
\begin{align}\label{sliders}
\begin{tikzpicture}[anchorbase,scale=0.7]
	\draw[-,thick] (0.4,0) to (-0.6,1);
	\draw[-,thick] (0.08,0) to (0.08,1);
	\draw[-,thick] (0.1,0) to (0.1,.6) to (.5,1);
        \node at (0.6,1.13) {$\scriptstyle c$};
        \node at (0.1,1.16) {$\scriptstyle b$};
        \node at (-0.65,1.13) {$\scriptstyle a$};
\end{tikzpicture} = \begin{tikzpicture}[anchorbase,scale=0.7]
	\draw[-,thick] (0.7,0) to (-0.3,1);
	\draw[-,thick] (0.08,0) to (0.08,1);
	\draw[-,thick] (0.1,0) to (0.1,.2) to (.9,1);
        \node at (0.9,1.13) {$\scriptstyle c$};
        \node at (0.1,1.16) {$\scriptstyle b$};
        \node at (-0.4,1.13) {$\scriptstyle a$};
\end{tikzpicture}
,\end{align} 
\begin{align}
\label{symmetric}
\begin{tikzpicture}[anchorbase,scale=0.8]
	\draw[-,thick] (0.28,0) to[out=90,in=-90] (-0.28,.6);
	\draw[-,thick] (-0.28,0) to[out=90,in=-90] (0.28,.6);
	\draw[-,thick] (0.28,-.6) to[out=90,in=-90] (-0.28,0);
	\draw[-,thick] (-0.28,-.6) to[out=90,in=-90] (0.28,0);
        \node at (0.3,-.75) {$\scriptstyle b$};
        \node at (-0.3,-.75) {$\scriptstyle a$};
\end{tikzpicture}
= \begin{tikzpicture}[anchorbase,scale=0.8]
	\draw[-,thick] (0.2,-.6) to (0.2,.6);
	\draw[-,thick] (-0.2,-.6) to (-0.2,.6);
        \node at (0.2,-.75) {$\scriptstyle b$};
        \node at (-0.2,-.75) {$\scriptstyle a$};
\end{tikzpicture}
,\end{align}
\begin{align}
\label{braid}
\begin{tikzpicture}[anchorbase,scale=0.8]
	\draw[-,thick] (0.45,.6) to (-0.45,-.6);
	\draw[-,thick] (0.45,-.6) to (-0.45,.6);
        \draw[-,thick] (0,-.6) to[out=90,in=-90] (-.45,0);
        \draw[-,thick] (-0.45,0) to[out=90,in=-90] (0,0.6);
        \node at (0,-.77) {$\scriptstyle b$};
        \node at (0.5,-.77) {$\scriptstyle c$};
        \node at (-0.5,-.77) {$\scriptstyle a$};
\end{tikzpicture}
 =
\begin{tikzpicture}[anchorbase,scale=0.8]
	\draw[-,thick] (0.45,.6) to (-0.45,-.6);
	\draw[-,thick] (0.45,-.6) to (-0.45,.6);
        \draw[-,thick] (0,-.6) to[out=90,in=-90] (.45,0);
        \draw[-,thick] (0.45,0) to[out=90,in=-90] (0,0.6);
        \node at (0,-.77) {$\scriptstyle b$};
        \node at (0.5,-.77) {$\scriptstyle c$};
        \node at (-0.5,-.77) {$\scriptstyle a$};
\end{tikzpicture}.
\end{align} 
However, we also have various novel relations on the vertical axis that are presented below. We will prove the following in the Appendix:
\begin{align}
    \label{hthickcrossing}
\begin{tikzpicture}[anchorbase,scale=1]
 \draw[-,line width=1pt,dashed,lightblue] (0.3,-.25) to (0.3,1.25);
	\draw[-,line width=1pt] (0,0) to (.6,1);
	\draw[-,line width=1pt] (0,1) to (.6,0);
    \draw[-,line width=1pt] (0.3,1) to (.3,0);
        \node at (0,-.1) {$\scriptstyle a$};
         \node at (0.3,-.1) {$\scriptstyle 2b$};
        \node at (0.6,-0.1) {$\scriptstyle a$};
\end{tikzpicture} =  \sum_{u = 0}^{\min(a, b)}\sum_{t = 0}^{a-u}(-1)^{t+u}
\begin{tikzpicture}[anchorbase, x=0.75pt,y=0.75pt,yscale=-1,xscale=1, scale = 0.3]
\draw[-,line width=1pt,dashed,lightblue] (130,90) to (130,260);
\draw    (110,130) -- (110,190) ; 
\draw    (150,130) -- (150,190) ;
\draw    (130,110) -- (130,230) ;
\draw    (90,110) -- (90,190) ;
\draw    (170,110) -- (170,190) ;
\draw    (90,110) -- (130,150) ;
\draw    (170,110) -- (130,150) ;
\draw    (130,170) -- (110,190) ;
\draw    (150,190) -- (130,170) ;
\draw    (170,190) -- (130,230) ;
\draw    (90,190) -- (130,230) ;
\draw    (150,190) -- (170,230) ;
\draw    (110,190) -- (90,230) ;
\node at (80,160) {$\scriptstyle u$};
\node at (100,160) {$\scriptstyle t$};
 \node at (90,250) {$\scriptstyle a$};
\node at (170,250) {$\scriptstyle a$};
\node at (130,250) {$\scriptstyle 2b$};
\end{tikzpicture} ,
\end{align}
\begin{align}
\label{hswallows}
\begin{tikzpicture}[anchorbase,scale=.7]
\draw[-,line width=1pt,dashed,lightblue] (0.08,-1.25) to (0.08,.75);
	\draw[-,line width=2pt] (0.08,.3) to (0.08,.5);
\draw[-,thick] (-.2,-.8) to [out=45,in=-45] (0.1,.31);
\draw[-,thick] (.36,-.8) to [out=135,in=-135] (0.06,.31);
\draw[-,thick] (0.08,-.8) to (0.08,.31);
        \node at (-.3,-.95) {$\scriptstyle a$};
        \node at (.45,-.95) {$\scriptstyle a$};
         \node at (.08,-.95) {$\scriptstyle 2b$};
\end{tikzpicture} =\begin{tikzpicture}[anchorbase,scale=.7]
\draw[-,line width=1pt,dashed,lightblue] (0.08,-1.25) to (0.08,.75);
\draw[-,line width=2pt] (0.08,.1) to (0.08,.5);
\draw[-,thick] (.46,-.8) to [out=100,in=-45] (0.1,.11);
\draw[-,thick] (-.3,-.8) to [out=80,in=-135] (0.06,.11);
\draw[-,thick] (0.08,-.8) to (0.08,.31);
        \node at (-.3,-.95) {$\scriptstyle a$};
        \node at (.45,-.95) {$\scriptstyle a$};
         \node at (.08,-.95) {$\scriptstyle 2b$};
\end{tikzpicture}
,\end{align}
\begin{align}
\label{hmidsliders}
\begin{tikzpicture}[anchorbase,scale=0.7]
    \draw[-,line width=1pt,dashed,lightblue] (-.32,-1.5) to (-.32,.5);
    \draw[-,thick] (-0.93,0) to (0.29,-1);
     \draw[-,thick] (0.29,0) to (-0.93,-1);
	\draw[-,thick] (-0.32,-1) to (-0.32,-.85) to (-.62,0);
    \draw[-,thick] (-0.32,-1) to (-.32,0);
    \draw[-,thick] (-0.32,-1) to (-0.32,-.85) to (-.02,0);
    \node at (-0.96,.13) {$\scriptstyle a$};
    \node at (-0.64,.13) {$\scriptstyle b$};
    \node at (-.32,.13) {$\scriptstyle 2c$};
    \node at (0,.13) {$\scriptstyle b$};
    \node at (0.32,.13) {$\scriptstyle a$};
\end{tikzpicture} = 
\begin{tikzpicture}[anchorbase,scale=0.7]
       \draw[-,line width=1pt,dashed,lightblue] (-.32,-1.5) to (-.32,.5);
    \draw[-,thick] (-0.93,0) to (0.29,-1);
     \draw[-,thick] (0.29,0) to (-0.93,-1);
	\draw[-,thick] (-0.32,-1) to (-0.32,-.15) to (-.62,0);
  \draw[-,thick] (-0.32,-1) to (-.32,0);
    \draw[-,thick] (-0.32,0) to (-0.32,-.15) to (-.02,0);
      \node at (-0.96,.13) {$\scriptstyle a$};
    \node at (-0.64,.13) {$\scriptstyle b$};
    \node at (-.32,.13) {$\scriptstyle 2c$};
    \node at (0,.13) {$\scriptstyle b$};
    \node at (0.32,.13) {$\scriptstyle a$};
\end{tikzpicture},\end{align}
\begin{align}
\label{hsidesliders}
\begin{tikzpicture}[anchorbase,scale=0.8]
\draw[-,line width=1pt,dashed,lightblue] (0,-1) to (0,1);
	\draw[-,thick] (-0.45,-.6) to (-0.45,-.45) to (-.6, -.3) to (.3, .6);
    \draw[-,thick] (-0.45,-.6) to (-0.45,-.45) to (-.3, -.3) to (.6, .6);
	\draw[-,thick] (0.45,-.6) to (0.45,-.45) to (.6, -.3) to (-.3, .6);
    \draw[-,thick] (0.45,-.6) to (0.45,-.45) to (.3, -.3) to (-.6, .6);
        \draw[-,thick] (0,-0.6) to (0,0.6);
        \node at (-.6,.7) {$\scriptstyle a$};
        \node at (.6,.7) {$\scriptstyle a$};
        \node at (-.3,.7) {$\scriptstyle b$};
        \node at (.3,.7) {$\scriptstyle b$};
        \node at (0,.7) {$\scriptstyle 2c$};
\end{tikzpicture} 
 =
\begin{tikzpicture}[anchorbase,scale=0.8]
\draw[-,line width=1pt,dashed,lightblue] (0,-1) to (0,1);
	\draw[-,thick] (-0.45,-.6) to (0.45,.45) to (.6, .6);
    \draw[-,thick] (-0.45,-.6) to (0.45,.45) to (.3, .6);
	\draw[-,thick] (0.45,-.6) to (-0.45,.45) to (-.6, .6);
    \draw[-,thick] (0.45,-.6) to (-0.45,.45) to (-.3, .6);
        \draw[-,thick] (0,-0.6) to (0,0.6);
        \node at (-.6,.7) {$\scriptstyle a$};
        \node at (.6,.7) {$\scriptstyle a$};
        \node at (-.3,.7) {$\scriptstyle b$};
        \node at (.3,.7) {$\scriptstyle b$};
        \node at (0,.7) {$\scriptstyle 2c$};
\end{tikzpicture} 
,\end{align}
\begin{align}
\label{hbraid}
\begin{tikzpicture}[anchorbase,scale=0.8]
\draw[-,line width=1pt,dashed,lightblue] (0,-1) to (0,1);
	\draw[-,thick] (0.6,.6) to (-0.6,-.6);
	\draw[-,thick] (0.6,-.6) to (-0.6,.6);
\draw[-,thick] (.3,-.6) to[out=90,in=-90] (-.3,0) to (-.3, .6);
\draw[-,thick] (-.3,-.6) to[out=90,in=-90] (.3,0) to (.3, .6);
\draw[-,thick] (0,-0.6) to (0,0.6);
        \node at (-.6,-.77) {$\scriptstyle a$};
        \node at (.6,-.77) {$\scriptstyle a$};
        \node at (-.3,-.77) {$\scriptstyle b$};
        \node at (.3,-.77) {$\scriptstyle b$};
        \node at (0,-.77) {$\scriptstyle 2c$};
\end{tikzpicture}
 =
\begin{tikzpicture}[anchorbase,scale=0.8]
\draw[-,line width=1pt,dashed,lightblue] (0,-1) to (0,1);
	\draw[-,thick] (0.6,.6) to (-0.6,-.6);
	\draw[-,thick] (0.6,-.6) to (-0.6,.6);
\draw[-,thick] (.3,.6) to[out=-90,in=90] (-.3,0) to (-.3, -.6);
\draw[-,thick] (-.3,.6) to[out=-90,in=90] (.3,0) to (.3, -.6);
\draw[-,thick] (0,-0.6) to (0,0.6);
        \node at (-.6,-.77) {$\scriptstyle a$};
        \node at (.6,-.77) {$\scriptstyle a$};
        \node at (-.3,-.77) {$\scriptstyle b$};
        \node at (.3,-.77) {$\scriptstyle b$};
        \node at (0,-.77) {$\scriptstyle 2c$};
\end{tikzpicture}
,\end{align}
\begin{align}
\label{hsymmetric}
\begin{tikzpicture}[anchorbase,scale=0.8]
\draw[-,line width=1pt,dashed,lightblue] (0,-1) to (0,1);
	\draw[-,thick] (0.6,.6) to (-0.6,-.6);
	\draw[-,thick] (0.6,-.6) to (-0.6,.6);
        \draw[-,thick] (.3,-.6) to[out=90,in=-90] (-.45,0);
        \draw[-,thick] (-0.45,0) to[out=90,in=-90] (.3,0.6);
        \draw[-,thick] (-.3,-.6) to[out=90,in=-90] (.45,0);
        \draw[-,thick] (0.45,0) to[out=90,in=-90] (-.3,0.6);
        \draw[-,thick] (0,-0.6) to (0,0.6);
        \node at (-.6,-.77) {$\scriptstyle a$};
        \node at (.6,-.77) {$\scriptstyle a$};
        \node at (-.3,-.77) {$\scriptstyle b$};
        \node at (.3,-.77) {$\scriptstyle b$};
        \node at (0,-.77) {$\scriptstyle 2c$};
\end{tikzpicture}
 =
\begin{tikzpicture}[anchorbase,scale=0.8]
\draw[-,line width=1pt,dashed,lightblue] (0,-1) to (0,1);
	\draw[-,thick] (0.6,.6) to (-0.6,-.6);
	\draw[-,thick] (0.6,-.6) to (-0.6,.6);
        \draw[-,thick] (.3,-.6) to (.3,0.6);
        \draw[-,thick] (-.3,-.6) to (-.3,0.6);
        \draw[-,thick] (0,-0.6) to (0,0.6);
        \node at (-.6,-.77) {$\scriptstyle a$};
        \node at (.6,-.77) {$\scriptstyle a$};
        \node at (-.3,-.77) {$\scriptstyle b$};
        \node at (.3,-.77) {$\scriptstyle b$};
        \node at (0,-.77) {$\scriptstyle 2c$};
\end{tikzpicture}
. \end{align}
\section{Equivalence between $\HWeb$ and $\HSchur$}\label{sec:hweb+hschur}
Our main result we will be proving is that, just as the polynomial web category and the Schur category are equivalent, so too are $\HWeb$ and $\HSchur$ category. This is presented below.
\bigskip
\\*
\noindent
{\bf Main Theorem.} {\em
     There is a an equivalence of categories $\Phi: \HWeb \rightarrow \HSchur$ which is the identity on objects and acts on generating morphisms as}
\[
\begin{tikzpicture}[baseline = -.5mm, anchorbase]
	\draw[-,line width=1pt] (0.28,-.3) to (0.08,0.04);
	\draw[-,line width=1pt] (-0.12,-.3) to (0.08,0.04);
	\draw[-,line width=2pt] (0.08,.4) to (0.08,0);
        \node at (-0.22,-.4) {$\scriptstyle a$};
        \node at (0.35,-.4) {$\scriptstyle b$};
\end{tikzpicture} 
\mapsto \xi_{\left(\begin{smallmatrix}a&       b&0&0\\0&0&b&a\end{smallmatrix}\right)},
\begin{tikzpicture}[baseline = -.5mm, anchorbase]
 \draw[-,line width=1pt,dashed,lightblue] (0.08,-.75) to (0.08,0.7);
	\draw[-,line width=1pt] (0.28,-.3) to (0.08,0.04);
	\draw[-,line width=1pt] (-0.12,-.3) to (0.08,0.04);
    \draw[-,line width=1pt] (0.08,-.3) to (0.08,0.04);
    \draw[-,line width=2pt] (0.08,.4) to (0.08,0);
        \node at (-0.22,-.4) {$\scriptstyle a$};
        \node at (0.07,-.4) {$\scriptstyle 2b$};
        \node at (0.35,-.4) {$\scriptstyle a$};
\end{tikzpicture} 
\mapsto \xi_{\left(\begin{smallmatrix}a&2b&a\end{smallmatrix}\right)},
\begin{tikzpicture}[baseline = -.5mm, anchorbase]
	\draw[-,line width=2pt] (0.08,-.3) to (0.08,0.04);
	\draw[-,line width=1pt] (0.28,.4) to (0.08,0);
	\draw[-,line width=1pt] (-0.12,.4) to (0.08,0);
        \node at (-0.22,.5) {$\scriptstyle a$};
        \node at (0.36,.5) {$\scriptstyle b$};
\end{tikzpicture}
\mapsto\xi_{\left(\begin{smallmatrix}a&0\\b&0\\0&b\\0&a\end{smallmatrix}\right)}, 
\]
\[
\begin{tikzpicture}[baseline = -.5mm, anchorbase]
\draw[-,line width=1pt,dashed,lightblue] (0.08,-.7) to (0.08,0.75);
	\draw[-,line width=2pt] (0.08,-.3) to (0.08,0.04);
	\draw[-,line width=1pt] (0.28,.4) to (0.08,0);
	\draw[-,line width=1pt] (-0.12,.4) to (0.08,0);
 \draw[-,line width=1pt] (0.08,.4) to   (0.08,0);
        \node at (-0.22,.5) {$\scriptstyle a$};
        \node at (0.36,.5) {$\scriptstyle a$};
        \node at (0.07,.5) {$\scriptstyle 2b$};
\end{tikzpicture}
\mapsto \xi_{\left(\begin{smallmatrix}a\\ 2b\\   a\end{smallmatrix}\right)}, 
\begin{tikzpicture}[baseline=-.5mm]
	\draw[-,thick] (-0.3,-.3) to (.3,.4);
	\draw[-,thick] (0.3,-.3) to (-.3,.4);
        \node at (0.3,-.4) {$\scriptstyle b$};
        \node at (-0.3,-.4) {$\scriptstyle a$};
\end{tikzpicture}
\mapsto \xi_{\left(\begin{smallmatrix}0 & b&0&0\\a & 0&0&0\\0 & 0&0&a\\0 &0&b&0\end{smallmatrix}\right)}, 
\begin{tikzpicture}[baseline=-.5mm, anchorbase]
  \draw[-,line width=1pt,dashed,lightblue] (0,-.5) to (0,0.6);
	\draw[-,thick] (-0.3,-.3) to (.3,.4);
	\draw[-,thick] (0.3,-.3) to (-.3,.4);
    \draw[-,thick] (0,-.3) to (0,.4);
        \node at (0.3,-.4) {$\scriptstyle a$};
        \node at (0,-.4) {$\scriptstyle 2b$};
        \node at (-0.3,-.4) {$\scriptstyle a$};
\end{tikzpicture}
\mapsto\xi_{\left(\begin{smallmatrix}0 & 0&a\\ 0&2b&0 \\  a &0 &0\end{smallmatrix}\right)}
.\]
Our proof will be broken down into the following sections:
\begin{enumerate}
 \item Show that $\Phi$ is a well-defined functor by verifying that the defining relations of $\HWeb$ are true in $\HSchur$.
   \item Prove that $\Phi$ is full by constructing for each $A\in \HMat_{\lambda, \mu}$ a diagram $[A]$ that gets sent to $\xi_A$ under $\Phi$.
  \item Prove that the relations of $\HWeb$ ensure the set of all diagrams is spanned by an easy to categorize subset, namely reduced chicken foot diagrams.
\end{enumerate}
\subsection{Verifying Defining Relations}\label{verify}
In this section, we will verify that $\Phi$ is well-defined; the relations that hold in $\HWeb$, namely~\ref{hsplitchoice}-\ref{hcommute}, are preserved in $\HSchur$. This detail was left in~\cite{Brundan} to the reader, though for completeness sake we will prove all of them up to horizontal symmetry.
\bigskip
\\*
{\em Proof of~\ref{hsplitchoice}.} We need to prove the following two equations for composing merges: \[ \xi_{\left(\begin{smallmatrix}
 a+b & c & 0 & 0\\
  0 & 0 & c & a+b\\
\end{smallmatrix}\right)}\cdot \xi_{\left(\begin{smallmatrix}
  a &  b& 0 & 0 & 0 & 0 \\
  0 & 0 & c & 0 & 0 & 0 \\
  0 & 0 & 0 & c & 0 & 0 \\
  0 & 0 & 0 & 0 & b & a \\
\end{smallmatrix}\right)}=\xi_{\left(\begin{smallmatrix}
 a & b+c & 0 & 0\\
  0 & 0 & b+c & a\\
\end{smallmatrix}\right)}\cdot \xi_{\left(\begin{smallmatrix}
  a & 0 & 0 & 0 & 0 & 0 \\
  0 & b & c & 0 & 0 & 0 \\
  0 & 0 & 0 & c & b & 0 \\
  0 & 0 & 0 & 0 & 0 & a \\
 \end{smallmatrix}\right)},\]
\[\xi_{\left(\begin{smallmatrix}
 a+b & 2c & a+b
\end{smallmatrix}\right)}\cdot \xi_{\left(\begin{smallmatrix}
  a &  b& 0 & 0 & 0 \\
  0 & 0 & 2c & 0 & 0 \\
  0 & 0 & 0 & b & a \\
\end{smallmatrix}\right)} = \xi_{\left(\begin{smallmatrix}
 a & 2b+2c & a
\end{smallmatrix}\right)}\cdot \xi_{\left(\begin{smallmatrix}
  a & 0 & 0 & 0 & 0 \\
  0 & b & 2c & b & 0 \\
  0 & 0 & 0 & 0 & a \\
\end{smallmatrix}\right)}.\]
The two equations for composing splits are analogous and are presented below: 
\[\xi_{\left(\begin{smallmatrix}
  a & 0 & 0 & 0\\
  0 & b & 0 & 0\\
  0 & c & 0 & 0\\
  0 & 0 & c & 0\\
  0 & 0 & b & 0\\
  0 & 0 & 0 & a\\
\end{smallmatrix}\right)}\cdot \xi_{\left(\begin{smallmatrix}
 a & 0 \\
 b+c & 0 \\
 0 & b+c \\
 0 & a \\
\end{smallmatrix}\right)} = \xi_{\left(\begin{smallmatrix}
  a & 0 & 0 & 0\\
  b & 0 & 0 & 0\\
  0 & c & 0 & 0\\
  0 & 0 & c & 0\\
  0 & 0 & 0 & b\\
  0 & 0 & 0 & a\\
\end{smallmatrix}\right)}\cdot \xi_{\left(\begin{smallmatrix}
 a+b & 0 \\
 c & 0 \\
 0 & c \\
 0 & a+b \\
\end{smallmatrix}\right)},\]
\[\xi_{\left(\begin{smallmatrix}
  a&0&0\\
  b&0&0\\
  0&2c&0\\
  0&0&b\\
  0&0&a\\
\end{smallmatrix}\right)}\cdot \xi_{\left(\begin{smallmatrix}
 a+b\\2c\\a+b\\
\end{smallmatrix}\right)} =  \xi_{\left(\begin{smallmatrix}
  a&0&0\\
  0&b&0\\
  0&2c&0\\
  0&b&0\\
  0&0&a\\
\end{smallmatrix}\right)}\cdot \xi_{\left(\begin{smallmatrix}
 a\\
 2b+2c\\
 a\\
\end{smallmatrix}\right)}.\]
We will leave these last two equations without formal proof due to the similarity. For our first equation on merges, by Theorem~\ref{hschurrule} it suffices to show that for all $C \in \HMat_{(a+b+c, a+b+c),(a, b, c, c, b, a)}$ we have \[Z\left(\left(\begin{smallmatrix}
 a & b+c & 0 & 0\\
  0 & 0 & b+c & a\\
\end{smallmatrix}\right),\left(\begin{smallmatrix}
  a &  0& 0 & 0 & 0 & 0 \\
  0 & b & c & 0 & 0 & 0 \\
  0 & 0 & 0 & c & b & 0 \\
  0 & 0 & 0 & 0 & 0 & a \\
 \end{smallmatrix}\right), C \right) = Z\left(\left(\begin{smallmatrix}
 a+b & c & 0 & 0\\
  0 & 0 & c & a+b\\
\end{smallmatrix}\right),\left(\begin{smallmatrix}
  a &  b& 0 & 0 & 0 & 0 \\
  0 & 0 & c & 0 & 0 & 0 \\
  0 & 0 & 0 & c & 0 & 0 \\
  0 & 0 & 0 & 0 & b & a \\
\end{smallmatrix}\right), C \right).\]
For either $Z$ value, if it is not zero then with a fixed $(\bi, \bk)\in \Pi_C$ and some $\bj$ we can deduce that the $1$'s, $2$'s, and $3$'s of $\bk$ are paired up with the $1$'s of $\bi$ while the $4$'s, $5$'s, and $6$'s of $\bk$ are paired up with the $2$'s. This forces
\[C = \left( \begin{smallmatrix}
    a & b & c & 0 & 0 & 0\\
    0 & 0 & 0 & c & b & a\\
\end{smallmatrix}\right).\]
Now, with $(\bi, \bk)\in \Pi_C$ fixed, the positions of the $1$'s to $6$'s in $\bk$ uniquely determine a $\bj$ for each $Z$ value. Therefore
  \[\xi_{\left(\begin{smallmatrix}
 a & b+c & 0 & 0\\
  0 & 0 & b+c & a
\end{smallmatrix}\right)}\cdot \xi_{\left(\begin{smallmatrix}
  a &  0& 0 & 0 & 0 & 0 \\
  0 & b & c & 0 & 0 & 0 \\
  0 & 0 & 0 & c & b & 0 \\
  0 & 0 & 0 & 0 & 0 & a 
\end{smallmatrix}\right)} = \xi_{\left(\begin{smallmatrix}
 a+b & c & 0 & 0\\
  0 & 0 & c & a+b
\end{smallmatrix}\right)}\cdot \xi_{\left(\begin{smallmatrix}
  a &  b& 0 & 0 & 0 & 0 \\
  0 & 0 & c & 0 & 0 & 0 \\
  0 & 0 & 0 & c & 0 & 0 \\
  0 & 0 & 0 & 0 & b & a 
\end{smallmatrix}\right)} = \xi_{\left(\begin{smallmatrix}
    a & b & c & 0 & 0 & 0\\
    0 & 0 & 0 & c & b & a
\end{smallmatrix}\right)}.\]
For the second part of Equation~\ref{hsplitchoice}, it likewise suffices to prove that for all $C \in \HMat_{(2a+2b+2c),(a, b, 2c, b, a)}$ we have
\[Z\left(\left(\begin{smallmatrix}
 a+b & 2c & a+b
\end{smallmatrix}\right),\left(\begin{smallmatrix}
  a &  b& 0 & 0 & 0 \\
  0 & 0 & 2c & 0 & 0 \\
  0 & 0 & 0 & b & a \\
\end{smallmatrix}\right), C \right) = Z\left(\left(\begin{smallmatrix}
 a & 2b+2c & a
\end{smallmatrix}\right),\left(\begin{smallmatrix}
  a & 0 & 0 & 0 & 0 \\
  0 & b & 2c & b & 0 \\
  0 & 0 & 0 & 0 & a \\
\end{smallmatrix}\right), C \right).\]Note that only one $C$ exists, namely
\[C = \left( \begin{smallmatrix}
    a & b & 2c & b & a\\
\end{smallmatrix}\right).\]
Now, once a  $(\bi, \bk)\in \Pi_C$ is fixed, the positions of the $1$'s to $5$'s in $\bk$ uniquely determine a $\bj$ for each $Z$ value meaning
\[\xi_{\left(\begin{smallmatrix}
 a+b & 2c & a+b
\end{smallmatrix}\right)}\cdot \xi_{\left(\begin{smallmatrix}
  a &  b& 0 & 0 & 0 \\
  0 & 0 & 2c & 0 & 0 \\
  0 & 0 & 0 & b & a \\
\end{smallmatrix}\right)} = \xi_{\left(\begin{smallmatrix}
 a & 2b+2c & a
\end{smallmatrix}\right)}\cdot \xi_{\left(\begin{smallmatrix}
  a & 0 & 0 & 0 & 0 \\
  0 & b & 2c & b & 0 \\
  0 & 0 & 0 & 0 & a \\
\end{smallmatrix}\right)} = \xi_{\left(\begin{smallmatrix}
 a & b&2c&b & a
\end{smallmatrix}\right)}.\]
{\em Proof of~\ref{htrivial}.} Below are our two equations to prove:
\[\xi_{\left(\begin{smallmatrix}
 a & b & 0 & 0\\
  0 & 0 & b & a
\end{smallmatrix}\right)}
\cdot \xi_{\left(\begin{smallmatrix}
 a & 0 \\
  b & 0 \\
   0 & b \\
    0 & a 
\end{smallmatrix}\right)} = \binom{a+b}{a}\xi_{\left(\begin{smallmatrix}
 a+b & 0 \\
  0 & a+b 
\end{smallmatrix}\right)},\]
\[\xi_{\left(\begin{smallmatrix}
    a & 2b & a 
\end{smallmatrix}\right)}\cdot \xi_{\left(\begin{smallmatrix}
    a \\
    2b\\
    a
\end{smallmatrix}\right)} = 2^a \binom{a+b}{a}\cdot \xi_{\left(\begin{smallmatrix}
    2a+2b
\end{smallmatrix}\right)}.\]
For our first equation, if $Z\left(\left(\begin{smallmatrix} a & b & 0 & 0\\
  0 & 0 & b & a
\end{smallmatrix}\right), \left(\begin{smallmatrix}
 a & 0 \\
  b & 0 \\
   0 & b \\
    0 & a 
\end{smallmatrix}\right), C\right) \neq 0$ then with a fixed $(\bi, \bk) \in \Pi_C$ and $\bj$ counted, we see that the $1$'s in $\bk$ pair with the $1$'s in $\bi$ and the $2$'s in $\bk$ pair with the $2$'s in $\bi$. Therefore 
\[C = \left(\begin{smallmatrix}
 a+b & 0 \\
  0 & a+b 
\end{smallmatrix}\right).\]
Taking $\bi = \bk =  (1,\ldots, 1, 2, \ldots, 2)$ where the tuple consists of $a+b$ $1$'s and $a+b$ $2$'s, there are $\binom{a+b}{a}$ ways to pair up the $a$ $1$'s and the $b$ $2$'s of $\bj$ with the $1$'s of $\bi, \bk$. This uniquely determines the $b$ $3$'s and the $a$ $4$'s of $\bj$ as well so we have
\[Z\left(\left(\begin{smallmatrix}
 a & b & 0 & 0\\
  0 & 0 & b & a
\end{smallmatrix}\right), \left(\begin{smallmatrix}
 a & 0 \\
  b & 0 \\
   0 & b \\
    0 & a 
\end{smallmatrix}\right),\left(\begin{smallmatrix}
 a+b & 0 \\
  0 & a+b 
\end{smallmatrix}\right)\right)= \binom{a+b}{a}\]
and applying Theorem~\ref{hschurrule} proves the first part of Equation~\ref{htrivial}. We also need to show that
\[\xi_{\left(\begin{smallmatrix}
    a & 2b & a 
\end{smallmatrix}\right)}\cdot \xi_{\left(\begin{smallmatrix}
    a \\
    2b\\
    a
\end{smallmatrix}\right)} = 2^a \binom{a+b}{a}\cdot \xi_{\left(\begin{smallmatrix}
    2a+2b
\end{smallmatrix}\right)}.\]
This follows from another combinatorial argument and the fact that $\left(\begin{smallmatrix}
    2a+2b
\end{smallmatrix}\right)$ is the only matrix in $\HMat_{(2a+2b), (2a+2b)}$. Specifically, with $\bi = \bk = (1,1,\ldots, 1)$ which consists of $2a+2b$ $1$'s, by the anti-symmetry of $\bj$ there are $\binom{a+b}{a}$ ways to position its $2$'s and $2^a$ ways to choose between $1$'s and $3$'s.
\bigskip
\\*
{\em Proof of~\ref{hmergesplit}.} We will first prove that
\[ \xi_{\left(\begin{smallmatrix}
b&0  \\
d&0  \\
0&d  \\
0&b 
\end{smallmatrix}\right)}\cdot \xi_{\left(\begin{smallmatrix}
a&c&0&0  \\
0&0&c&a  
\end{smallmatrix}\right)} = \sum_{\substack{0 \leq s \leq \min(a,b)\\0 \leq t \leq \min(c,d)\\t-s=d-a}} \xi_{\left(\begin{smallmatrix}
s&c-t&0&0  \\
a-s&t&0&0  \\
0&0&t&a-s  \\
0&0&c-t&s 
\end{smallmatrix}\right)}\]
for $a+c = b+d$. Applying Theorem~\ref{hschurrule} to the left hand side, if a $Z$ value is not equal to $0$ then for a $(\bi, \bk) \in \Pi_C$ with a valid $\bj$ we have the following observation: for each of $\bi$, $\bk$, all of the $1's$ and $2$'s are paired with all of the $1$'s of $\bj$, and similarly all of the $3's$ and $4$'s are paired with all of the $2$'s of $\bj$. This restricts the possibilities of $C$ down to the ones enumerated on the right hand side. This observation also tells us that the $Z$ values for these matrices are $1$, as the $1$'s and $2$'s of $\bj$ are uniquely determined by the positions of numbers in $\bi, \bk$.
\bigskip
\\*
Next, we need to show that
\[ \xi_{\left(\begin{smallmatrix}
a\\
2b\\
a
\end{smallmatrix}\right)}\cdot \xi_{\left(\begin{smallmatrix}
a&2b&a
\end{smallmatrix}\right)} = \sum_{\substack{0\leq s\leq a, 0\leq t\leq \min(b, a-s)}} \xi_{\left(\begin{smallmatrix}
s&t&a-s-t  \\
t&2b-2t&t  \\
a-s-t&t&s 
\end{smallmatrix}\right)}.\]
Note that the right hand side is enumerating through all matrices of $\HMat_{(b, 2a, b), (b, 2a, b)}$, so it suffices to show that all appropriate $Z$ values equal $1$. This is clear as $\bj = (1,1,\ldots,1)$ with $2a+2b$ $1$'s is the only tuple possible, and it works for any fixed $(\bi, \bk)$.
\bigskip
\\*
{\em Proof of~\ref{commute},~\ref{hcommute}.} Both relations roughly follow from how disjoint numbers in tuples act independently.
\subsection{Fullness of Phi}
Next, we need to prove that $\Phi$ is full. There are three types of homomorphisms $\xi_A$ in $\HSchur$ which are easy to understand and see that they are in the image of $\Phi$.
\begin{enumerate}
\item
If $A$ is a $1 \times n$ row matrix with $180^\circ$ rotational symmetry, we call $\xi_A$ an {\em $n$-fold
merge}. By Theorem~\ref{hschurrule} we can write any $n$-fold merge
$\xi_{\left(\begin{smallmatrix}\lambda_1&\cdots&
      \lambda_{n}\end{smallmatrix}\right)}$
as a composition of generators of the form 
$\xi_{\left(\begin{smallmatrix}a&b & b & a\end{smallmatrix}\right)}$ and $\xi_{\left(\begin{smallmatrix}a&2b&a\end{smallmatrix}\right)}$, which are merges off and on the axis of symmetry respectively.
\item
If $A$ is an $n \times 1$ column matrix with $180^\circ$ rotational symmetry, we call $\xi_A$ an {\em
  $n$-fold split}. As with an $n$-fold merge, any $n$-fold split can be expressed as a composition of generators of the form 
$\xi_{\left(\begin{smallmatrix}a\\b\\b\\a\end{smallmatrix}\right)}$ and $\xi_{\left(\begin{smallmatrix}a\\2b\\a\end{smallmatrix}\right)}$, which are splits off and on the axis of symmetry respectively.
\item
If $A$ is an $n \times n$ matrix with exactly one
non-zero entry in every row and column and exhibiting $180^\circ$ rotational symmetry, we call $\xi_A$ a {\em generalized
  hyper-permutation}. Any generalized hyper-permutation can be expressed as a composition of crossings generators $\xi_{\left(\begin{smallmatrix}0 & b&0&0\\a & 0&0&0\\0 & 0&0&a\\0 &0&b&0\end{smallmatrix}\right)}$ and $\xi_{\left(\begin{smallmatrix}0&0&a\\0&2b&0\\a&0&0\end{smallmatrix}\right)}$.
\end{enumerate}
We will now prove the following useful lemma.
\begin{lemma}\label{splitA}
Suppose that $A \in \HMat_{\lambda,\mu}$ and $B \in \HMat_{\mu,\nu}$
for $\lambda,\mu,\nu$ compositions of $H_n$ and assume that
\begin{itemize}
\item
$A$ has a unique non-zero entry in every column,
so that there is an associated function $\alpha:\nset{\ell(\mu)}\rightarrow \nset{\ell(\lambda)}$
sending $i$ to the unique $j$ such that $a_{j,i} \neq 0$;
\item
$B$ has a unique non-zero entry in every row,
 so that there is an associated function
$\beta:\nset{\ell(\mu)} \rightarrow \nset{\ell(\nu)}$
sending $i$ to the unique $j$ such that $b_{i,j} \neq 0$;
\item
the function
$\gamma:\nset{\ell(\mu)}\rightarrow\nset{\ell(\lambda)}\times\nset{\ell(\nu)},
i\mapsto (\alpha(i),\beta(i))$ is injective.
\end{itemize}
Then $\xi_A \circ \xi_B = \xi_C$
where $C \in \HMat_{\lambda,\nu}$ is the matrix with 
$c_{\alpha(i),\beta(i)}=\mu_i$ for $i \in \nset{\ell(\mu)}$, all other entries being zero.
\end{lemma}
\begin{proof}
Let
$\bj := (L_\mu(1) \dots, L_\mu(2n))$, with $\bi = \alpha(\bj)$ and $\bk = \beta(\bj)$ as the tuples obtained by applying $\alpha$ and $\beta$ to the entries of $\bj$. Observe that $(\bi,\bj) \in \Pi_A, (\bj,\bk) \in \Pi_B$, and
$(\bi, \bk) \in \Pi_C$.
Furthermore, the injectivity of $\gamma$ implies that
$\Stab_{H_n}(\bi) \cap \Stab_{H_n}(\bk) = \Stab_{H_n}(\bj)$. Applying Lemma~\ref{htricky} completes the proof.
\end{proof}
The reason for proving Lemma ~\ref{splitA} is because of the following construction of $A \in \HMat_{\lambda,\mu}$ into three parts.
\begin{enumerate}
\item
Let $A^-$ be the block diagonal
matrix $\diag(A_1,\dots,A_{\ell(\lambda)})$ where $A_i$ is the 
$1 \times n_i$
matrix obtained from the $i$th row of $A$ by removing all entries equal to $0$. Due to $A\in \HMat_{\lambda, \mu}$ exhibiting $180^\circ$ rotational symmetry, this symmetry also holds for $A^-$. In particular $\xi_{A^-}$ can be seen as the composition of 
\[\xi_{\diag(A_1, A_{\ell(\lambda)})},\ldots,\;\xi_{\diag(A_{\frac{\ell(\lambda)-1}{2}}, A_{\frac{\ell(\lambda)+3}{2}})}, \; \xi_{A_{\frac{\ell(\lambda)+1}{2}}}\]
once the appropriate identity morphisms are added. Each term is a generalized merge off or on the axis. Also define $\lambda^-$ to be the composition recording the column sums of
$A^-$, so that $A^- \in \HMat_{\lambda,\lambda^-}$.
The $i$th entry $\lambda^-_i$ of $\lambda^-$
is the $i$th
{non-zero} entry
of the sequence
$a_{1,1},a_{1,2},\dots, a_{1,\ell(\mu)}, a_{2,1},\dots$
that is the {\em row reading} of the matrix $A$.
\item
Let $A^+$ be the block diagonal matrix
$\diag(A^1,\dots,A^{\ell(\mu)})$ where $A^i$ is the 
$n^i \times 1$ matrix obtained from the $i$th column of $A$ by removing all entries equal to
$0$.
As with $A^-$, $A^+$ has a $180^\circ$ rotational symmetry so that $\xi_{A^+}$ can be seen as the composition of 
\[\xi_{\diag(A^1, A^{\ell(\mu)})},\ldots,\;\xi_{\diag(A^{\frac{\ell(\mu)-1}{2}}, A^{\frac{\ell(\mu)+3}{2}})}, \; \xi_{A^{\frac{\ell(\mu)+1}{2}}}\]
once the appropriate identity morphisms are added. Each term is a generalized split off or on the axis. Also define $\mu^+$ to be the composition recording the row sums of $A^+$, so 
that $A^+ \in \HMat_{\mu^+,\mu}$.
The $i$th entry $\mu^+_i$ of $\mu^+$ is the $i$th {non-zero}
entry of the sequence $a_{1,1}, a_{2,1},\dots,a_{\ell(\lambda),1},
a_{1,2},\dots$
that is the {\em column reading} of $A$. Observe that $\ell(\mu^+) = \ell(\lambda^-)$ as they both count the number of non-zero terms of $A$.
\item 
Observe that the
$\lambda^-$ is a rearrangement of $\mu^+$. We can then define $f_1: \nset{\ell(\lambda^-)} \rightarrow \nset{\ell(\lambda)}$ and $f_2:\nset{\ell(\lambda^-)}\rightarrow \nset{\ell(\mu)}$ so that $\lambda_i^-$, the $i$th non-zero entry
of the row reading of $A$, is in row $f_1(i)$ and column $f_2(i)$.
Likewise let
$h_1:\nset{\ell(\mu^+)} \rightarrow \nset{\ell(\lambda)}$ and $h_2:\nset{\ell(\mu^+)} \rightarrow \nset{\ell(\mu)}$ be 
defined so that $\mu_i^+$, the $i$th non-zero entry of the
column reading of $A$, is in row $h_1(i)$ and column $h_2(i)$.
There is then a unique permutation $g \in S_{\ell(\lambda^-)}$ such that
$(f_1(g(i)), f_2(g(i))) = (h_1(i), h_2(i))$ for each $i\in\nset{\ell(\lambda^-)}$.
We have in particular that
$g(\mu^+) = \lambda^-$. Let $A^\circ \in \HMat_{\lambda^-,\mu^+}$ 
be the $\ell(\lambda^-) \times \ell(\lambda^-)$ monomial matrix with $(g(i),i)$-entry equal to $\mu^+_i$ for
$i=1,\ldots,\ell(\lambda^-)$, 
all other entries being zero.
\end{enumerate}
\begin{remark}
    Unlike in ~\cite{Brundan} and the case of the Schur algebra, it is also important to see that the rotational symmetry of $A\in \HMat_{\lambda, \mu}$ creates the needed rotational symmetries for $A^-, A^+, A^\circ$.
\end{remark}
For example, suppose that $A=
\begin{pmatrix} 1&2&0\\1&2&1\\0&2&1\end{pmatrix}$, 
so $\lambda = (3,4,3)$ and $\mu = (2,6,2)$.
Then
\[
A^- = \begin{pmatrix}1 & 2 & 0 & 0 & 0&0&0\\0&0&1&2&1&0&0\\0 & 0 & 0 & 0 &
  0&2&1\end{pmatrix},
A^\circ = \begin{pmatrix}
1&0&0&0&0&0&0\\
0&0&2&0&0&0&0\\
0&1&0&0&0&0&0\\
0&0&0&2&0&0&0\\
0&0&0&0&0&1&0\\
0&0&0&0&2&0&0\\
0&0&0&0&0&0&1
\end{pmatrix},
 A^+ = \begin{pmatrix}1 & 0 & 0 \\ 1 & 0 & 0 \\ 0 & 2 & 0\\ 0 &
  2 & 0\\ 0 & 2 & 0\\0&0&1\\0&0&1\end{pmatrix}.
\]
Also
$\lambda^- = (1,2,1,2,1,2,1)$
and $\mu^+ = (1,1,2,2,2,1,1)$
so that $g = (23)(56)$. 
\begin{lemma}\label{splitxi}
For $A \in \HMat_{\lambda,\mu}$, we have that
$\xi_A = \xi_{A^-} \circ \xi_{A^\circ} \circ \xi_{A^+}$.
\end{lemma}

\begin{proof}
Define $\lambda^-, \mu^+$ and $f_1,f_2,g,h_1,h_2$ as above.
First, we 
apply Lemma~\ref{splitA} 
with $\alpha = g$ and $\beta = h_2$
to deduce that $\xi_{A^\circ} \circ \xi_{A^+} = \xi_B$
for $B \in \HMat_{\lambda^-,\mu}$ defined so that
$b_{g(i), h_2(i)} = \mu^+_i$ for $i=1,\ldots,\ell(\mu^+)$, all other entries
being zero. We can apply the lemma once more with $\alpha = f_1$ and $\beta = g \circ h_2$ 
to show that $\xi_{A^-} \circ \xi_B = \xi_A$.
\end{proof}
Lemma~\ref{splitxi} shows that any $\xi_A$ can be expressed as the composition of merges, a generalized hyper-permutation, and splits. All three of these components (i.e $\xi_{A^-},\xi_{A^\circ}, \xi_{A^+}$) are in the image of $\Phi$ so $\xi_A = \xi_{A^-} \circ \xi_{A^\circ} \circ \xi_{A^+}\in \im(\Phi)$ as needed. This proves that $\Phi$ is a full functor from $\HWeb$ to $\HSchur$.
\subsection{Reduction of Diagrams}
To find a convenient spanning set of $\Hom_{\HWeb}(\mu, \lambda)$, we will proceed in similar steps to~\cite[Lemma 4.9]{Brundan} with the below definitions.
\begin{definition}
Given two hypercompositions $\lambda, \mu$ of $H_n$, a $\lambda\times \mu$ \emph{chicken foot} \emph{diagram} representing a morphism from $\mu$ to $\lambda$ consists of thick strings determined by $\mu$ at the bottom that split into thinner strings which
cross each other in some way before merging back into thick strings determined by $\lambda$ at the top. In particular there also has to be a vertical line of symmetry.
\end{definition}
Importantly, a general web diagram lacks the ordering of splits, crossings, and merges that chicken foot diagrams have. Besides chicken foot diagrams, we can impose a further restriction.
\begin{definition}
    A chicken foot diagram is~\emph{reduced} if there is at most one intersection, merging, or split between every pair of strings in the diagram.
\end{definition}
\begin{remark}
For example, considering Equation~\ref{hswallows} we can observe that both sides are chicken foot diagrams, but only the right hand side is reduced. This is because for the left hand side, the left and right strands intersect twice: at the merge and the crossing. 
\end{remark}
Given a reduced chicken foot diagram with a thickness of $a_{i,j}$ between the $i$-th top vertex and the $j$-th bottom vertex ($a_{i,j}$ possibly zero), we say that the \emph{type} of this diagram is the matrix $A\in \HMat_{\lambda, \mu}$ with coefficients $a_{i,j}$. For example the diagrams
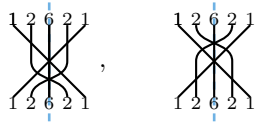
\[\begin{tikzpicture}[anchorbase,scale=0.8]
\draw[-,line width=1pt,dashed,lightblue] (0,-1) to (0,1);
\draw[-,thick] (0.6,.6) to (-0.6,-.6);
\draw[-,thick] (0.6,-.6) to (-0.6,.6);
\draw[-,thick] (.3,-.6) to[out=90,in=-90] (-.3,0) to (-.3, .6);
\draw[-,thick] (-.3,-.6) to[out=90,in=-90] (.3,0) to (.3, .6);
\draw[-,thick] (0,-0.6) to (0,0.6);
\draw[-,thick] (0,-0.6) to (0,0.6);
   \node at (-.6,.7) {$\scriptstyle 1$};
        \node at (.6,.7) {$\scriptstyle 1$};
        \node at (-.3,.7) {$\scriptstyle 2$};
        \node at (.3,.7) {$\scriptstyle 2$};
        \node at (0,.7) {$\scriptstyle 6$};
        \node at (-.6,.7) {$\scriptstyle 1$};
          \node at (.6,-.7) {$\scriptstyle 1$};
        \node at (-.3,-.7) {$\scriptstyle 2$};
        \node at (.3,-.7) {$\scriptstyle 2$};
        \node at (0,-.7) {$\scriptstyle 6$};
        \node at (-.6,-.7) {$\scriptstyle 1$};
\end{tikzpicture} ,
 \quad \quad
\begin{tikzpicture}[anchorbase,scale=0.8]
\draw[-,line width=1pt,dashed,lightblue] (0,-1) to (0,1);
\draw[-,thick] (0.6,.6) to (-0.6,-.6);
	\draw[-,thick] (0.6,-.6) to (-0.6,.6);
\draw[-,thick] (.3,.6) to[out=-90,in=90] (-.3,0) to (-.3, -.6);
\draw[-,thick] (-.3,.6) to[out=-90,in=90] (.3,0) to (.3, -.6);
\draw[-,thick] (0,-0.6) to (0,0.6);
   \node at (-.6,.7) {$\scriptstyle 1$};
        \node at (.6,.7) {$\scriptstyle 1$};
        \node at (-.3,.7) {$\scriptstyle 2$};
        \node at (.3,.7) {$\scriptstyle 2$};
        \node at (0,.7) {$\scriptstyle 6$};
        \node at (-.6,.7) {$\scriptstyle 1$};
          \node at (.6,-.7) {$\scriptstyle 1$};
        \node at (-.3,-.7) {$\scriptstyle 2$};
        \node at (.3,-.7) {$\scriptstyle 2$};
        \node at (0,-.7) {$\scriptstyle 6$};
        \node at (-.6,-.7) {$\scriptstyle 1$};
\end{tikzpicture} \]
both have type \[\left(\begin{smallmatrix}
0 & 0& 0& 0& 1  \\
0 & 0& 0& 2& 0  \\
0 & 0& 6& 0& 0  \\
0 & 2& 0& 0& 0  \\
1 & 0& 0& 0& 0  
\end{smallmatrix}\right)\in \HMat_{(1,2,6,2,1),(1,2,6,2,1)}.\]
Every matrix $A\in \HMat_{\lambda, \mu}$ has a diagram of type $A$, just draw between the $i$-th top vertex and the $j$-th bottom vertex a string of thickness $a_{i,j}$. However, it is not clear that this diagram is unique; in the above example Equation~\ref{hbraid} assures us that the two seemingly distinct diagrams are in fact equal. 
\begin{lemma}\label{onetype}
    For all $A\in \HMat_{\lambda, \mu}$, there is a unique $\lambda \times \mu$ reduced chicken foot diagram with type $A$ up to the relations of $\HWeb$.
\end{lemma}
\begin{proof}
    First of all, due to the generalized merges and splits on and off the axis, two reduced chicken foot diagrams of the same type have equivalent merge and split portions up to Equation~\ref{hsplitchoice}. All that is left to show is that two reduced diagrams encoding the same generalized hyper-permutation are equal in $\HWeb$. We proceed by casework on the string coming from the first vertex of the bottom row.
    \bigskip
    \\*
If this string (indicated in green) does not cross the vertical axis, then each reduced diagram can be written in the form
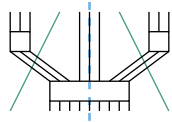
\[\begin{tikzpicture}[anchorbase, x=0.75pt,y=0.75pt,yscale=-1,xscale=1, scale = 0.5]
\draw[-,line width=1pt,dashed,lightblue]    (180,80) -- (180,200) ;
\draw[green]    (150,90) -- (100,190) ;
\draw    (100,90) -- (100,110) ;
\draw    (120,90) -- (120,110) ;
\draw    (110,90) -- (110,110) ; 
\draw[green]    (210,90) -- (260,190) ; 
\draw    (240,90) -- (240,110) ;
\draw    (260,90) -- (260,110) ;
\draw    (250,90) -- (250,110) ;
\draw    (170,90) -- (170,160) ;
\draw    (180,90) -- (180,160) ;
\draw    (190,90) -- (190,160) ;
\draw    (100,130) -- (140,160) ;
\draw    (110,130) -- (150,160) ;
\draw    (120,130) -- (160,160) ;
\draw    (260,130) -- (220,160) ;
\draw    (250,130) -- (210,160) ; 
\draw    (240,130) -- (200,160) ;
\draw   (100,110) -- (120,110) -- (120,130) -- (100,130) -- cycle ;
\draw   (240,110) -- (260,110) -- (260,130) -- (240,130) -- cycle ;
\draw   (140,160) -- (220,160) -- (220,180) -- (140,180) -- cycle ;
\draw    (140,180) -- (140,190) ;
\draw    (150,180) -- (150,190) ;
\draw    (160,190) -- (160,180) ;
\draw    (170,190) -- (170,180) ;
\draw    (180,190) -- (180,180) ;
\draw    (190,190) -- (190,180) ;
\draw    (200,190) -- (200,180) ;
\draw    (210,190) -- (210,180) ;
\draw    (220,190) -- (220,180) ;
\end{tikzpicture}\]
where each box can be imagined to have a number of crossings on an arbitrary number of strings, rather than just this example with $3$ and $9$ strings for our boxes. By Equation~\ref{sliders} we can pull crossings across our string to get 
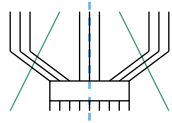
\[\begin{tikzpicture}[anchorbase, x=0.75pt,y=0.75pt,yscale=-1,xscale=1, scale = 0.5]
\draw[-,line width=1pt,dashed,lightblue]    (180,80) -- (180,200) ;
\draw[green]    (150,90) -- (100,190) ;
\draw    (100,90) -- (100,130) ;
\draw    (120,90) -- (120,130) ;
\draw    (110,90) -- (110,130) ; 
\draw[green]    (210,90) -- (260,190) ; 
\draw    (240,90) -- (240,130) ;
\draw    (260,90) -- (260,130) ;
\draw    (250,90) -- (250,130) ;
\draw    (170,90) -- (170,160) ;
\draw    (180,90) -- (180,160) ;
\draw    (190,90) -- (190,160) ;
\draw    (100,130) -- (140,160) ;
\draw    (110,130) -- (150,160) ;
\draw    (120,130) -- (160,160) ;
\draw    (260,130) -- (220,160) ;
\draw    (250,130) -- (210,160) ; 
\draw    (240,130) -- (200,160) ;
\draw   (140,160) -- (220,160) -- (220,180) -- (140,180) -- cycle ;
\draw    (140,180) -- (140,190) ;
\draw    (150,180) -- (150,190) ;
\draw    (160,190) -- (160,180) ;
\draw    (170,190) -- (170,180) ;
\draw    (180,190) -- (180,180) ;
\draw    (190,190) -- (190,180) ;
\draw    (200,190) -- (200,180) ;
\draw    (210,190) -- (210,180) ;
\draw    (220,190) -- (220,180) ;
\end{tikzpicture}\]
at which point we can note that the above box is a reduced diagram on fewer strings. It is reduced because Equation~\ref{braid} preserves with multiplicity the set of pairs of strings that cross.
\bigskip
\\*
On the other hand if our initial string crossed the vertical axis then each reduced diagram has the presentation
\[\begin{tikzpicture}[anchorbase, x=0.75pt,y=0.75pt,yscale=-1,xscale=1, scale = 0.5]
\draw[-,line width=1pt,dashed,lightblue]    (180,40) -- (180,200) ;
\draw[green]   (240,50) -- (100,190) ;
\draw    (100,50) -- (100,110) ;
\draw    (110,50) -- (110,110) ;
\draw[green]    (120,50) -- (260,190) ;
\draw    (260,50) -- (260,110) ;
\draw    (250,50) -- (250,110) ; 
\draw    (160,80) -- (120,110) ;
\draw    (180,80) -- (180,160) ;
\draw    (200,80) -- (240,110) ;
\draw    (100,130) -- (140,160) ;
\draw    (110,130) -- (150,160) ;
\draw    (120,130) -- (160,160) ;
\draw    (260,130) -- (220,160) ;
\draw    (250,130) -- (210,160) ;
\draw    (230,130) -- (190,160) ;
\draw   (100,110) -- (130,110) -- (130,130) -- (100,130) -- cycle ;
\draw   (230,110) -- (260,110) -- (260,130) -- (230,130) -- cycle ; 
\draw   (140,160) -- (220,160) -- (220,180) -- (140,180) -- cycle ; 
\draw    (140,180) -- (140,190) ;
\draw    (150,180) -- (150,190) ;
\draw    (160,190) -- (160,180) ;
\draw    (170,190) -- (170,180) ;
\draw    (180,190) -- (180,180) ;
\draw    (190,190) -- (190,180) ;
\draw    (200,190) -- (200,180) ;
\draw    (210,190) -- (210,180) ;
\draw    (220,190) -- (220,180) ;
\draw   (160,60) -- (200,60) -- (200,80) -- (160,80) -- cycle ;
\draw    (160,50) -- (160,60) ;
\draw    (170,50) -- (170,60) ;
\draw    (180,50) -- (180,60) ;
\draw    (190,50) -- (190,60) ;
\draw    (200,50) -- (200,60) ;
\draw    (170,80) -- (130,110) ; 
\draw    (190,80) -- (230,110) ;
\draw    (130,130) -- (170,160) ; 
\draw    (240,130) -- (200,160) ;
\end{tikzpicture}.\]
We can then use Equations~\ref{braid} and~\ref{hbraid} to simplify this diagram down to
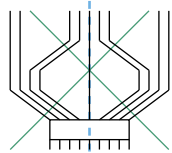
\[\begin{tikzpicture}[anchorbase, x=0.75pt,y=0.75pt,yscale=-1,xscale=1, scale = 0.5]
\draw[-,line width=1pt,dashed,lightblue]    (180,40) -- (180,200) ;
\draw[green]   (240,50) -- (100,190) ;
\draw    (100,50) -- (100,130) ;
\draw    (110,50) -- (110,130) ;
\draw[green]    (120,50) -- (260,190) ;
\draw    (260,50) -- (260,130) ;
\draw    (250,50) -- (250,130) ; 
\draw    (160,80) -- (120,110) ;
\draw    (180,50) -- (180,160) ;
\draw    (200,80) -- (240,110) ;
\draw    (100,130) -- (140,160) ;
\draw    (110,130) -- (150,160) ;
\draw    (120,130) -- (160,160) ;
\draw    (260,130) -- (220,160) ;
\draw    (250,130) -- (210,160) ;
\draw    (230,130) -- (190,160) ;
\draw    (120,110) -- (120,130) ;
\draw    (130,110) -- (130,130) ;
\draw    (230,110) -- (230,130) ;
\draw    (240,110) -- (240,130) ;
\draw   (140,160) -- (220,160) -- (220,180) -- (140,180) -- cycle ; 
\draw    (140,180) -- (140,190) ;
\draw    (150,180) -- (150,190) ;
\draw    (160,190) -- (160,180) ;
\draw    (170,190) -- (170,180) ;
\draw    (180,190) -- (180,180) ;
\draw    (190,190) -- (190,180) ;
\draw    (200,190) -- (200,180) ;
\draw    (210,190) -- (210,180) ;
\draw    (220,190) -- (220,180) ;
\draw    (160,50) -- (160,80) ;
\draw    (170,50) -- (170,80) ;
\draw    (190,50) -- (190,80) ;
\draw    (200,50) -- (200,80) ;
\draw    (170,80) -- (130,110) ; 
\draw    (190,80) -- (230,110) ;
\draw    (130,130) -- (170,160) ; 
\draw    (240,130) -- (200,160) ;
\end{tikzpicture}\]
and as before, this leaves us with a box that is a reduced diagram on fewer strings. By recursively applying this algorithm until our box becomes empty or it only has a string on the vertical line of symmetry, any two reduced diagrams representing the same generalized hyper-permutation can be untangled into the same diagram. 
\end{proof}
\begin{definition}
For $A\in \HWeb_{\lambda, \mu}$, we define the diagram $[A]$ as the reduced chicken foot diagram with type $A$. This is well-defined by Lemma~\ref{onetype}.
\end{definition}
Recall that $\HWeb$ is generated by six morphisms: two types of crossings, merges, and splits each. However, by Equations~\ref{switch} and~\ref{hthickcrossing} we can rewrite all crossings in terms of merges and splits, so $\HWeb$ can be generated from just these four morphisms. This motivates us to prove the following lemma.
\begin{lemma}\label{reducedpreserved}
    Given a reduced chicken foot diagram $g$ and a merge or a split $f$, the vertical composition $f\circ g$ can be rewritten as a linear combination of reduced diagrams.
\end{lemma}
\begin{proof}
Using Equation~\ref{hsplitchoice} to have well-defined general merges and Equation~\ref{hmergesplit} to change the order of a split and a merge, we have four types of rewrites for each of the cases of $f$. Examples are shown below:
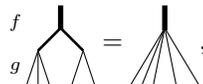
\[
\begin{tikzpicture}[anchorbase,scale=0.015,yscale=-1]
\draw[-,line width=2pt] (180,72) -- (180,50);
\draw[-,line width=1pt]    (180,70) -- (160,90) ;
\draw[-,line width=1pt]    (180,70) -- (200,90) ;
\draw[-,thin]    (160,90) -- (150,120) ;
\draw[-,thin]    (160,90) -- (160,120) ;
\draw[-,thin]    (160,90) -- (170,120) ;
\draw[-,thin]    (200,90) -- (210,120) ;
\draw[-,thin]    (200,90) -- (190,120) ;

\node at (140,65) {$\scriptstyle f$};
\node at (140,105) {$\scriptstyle g$};
\end{tikzpicture}=
\begin{tikzpicture}[anchorbase,scale=0.015,yscale=-1]
\draw[-,line width=2pt] (180,72) -- (180,50);
\draw[-,thin]    (180,70) -- (150,120) ;
\draw[-,thin]    (180,70) -- (160,120) ;
\draw[-,thin]    (180,70) -- (170,120) ;
\draw[-,thin]    (180,70) -- (210,120) ;
\draw[-,thin]    (180,70) -- (190,120) ;
\end{tikzpicture},
\]
\[\begin{tikzpicture}[anchorbase,scale=0.015,yscale=-1]
 \draw[-,line width=1pt,dashed,lightblue] (180,30) to (180,155);
\draw[-,line width=2pt] (180,72) -- (180,50);
\draw[-,line width=1pt]    (180,70) -- (140,100) ;
\draw[-,line width=1pt]    (180,70) -- (220,100) ;
\draw[-,line width=1pt]    (180,70) -- (180,100) ;
\draw[-,thin]    (140,100) -- (130,130) ;
\draw[-,thin]    (140,100) -- (140,130) ;
\draw[-,thin]    (140,100) -- (150,130) ;
\draw[-,thin]    (220,100) -- (230,130) ;
\draw[-,thin]    (220,100) -- (220,130) ;
\draw[-,thin]    (220,100) -- (210,130) ;
\draw[-,thin]    (180,100) -- (190,130) ;
\draw[-,thin]    (180,100) -- (180,130) ;
\draw[-,thin]    (180,100) -- (170,130) ;

\node at (120,75) {$\scriptstyle f$};
\node at (120,115) {$\scriptstyle g$};
\end{tikzpicture}=\begin{tikzpicture}[anchorbase,scale=0.015,yscale=-1]
 \draw[-,line width=1pt,dashed,lightblue] (180,30) to (180,155);
\draw[-,line width=2pt] (180,72) -- (180,50);
\draw[-,thin]    (180,70) -- (130,130) ;
\draw[-,thin]    (180,70) -- (140,130) ;
\draw[-,thin]    (180,70) -- (150,130) ;
\draw[-,thin]    (180,70) -- (230,130) ;
\draw[-,thin]    (180,70) -- (220,130) ;
\draw[-,thin]    (180,70) -- (210,130) ;
\draw[-,thin]    (180,70) -- (190,130) ;
\draw[-,thin]    (180,70) -- (180,130) ;
\draw[-,thin]    (180,70) -- (170,130) ;
\end{tikzpicture},\]
\[
\begin{tikzpicture}[anchorbase,scale=0.015,yscale=-1]
\draw[-,line width=2pt] (180,70) -- (180,92);
\draw[-,line width=1pt]    (180,70) -- (160,50) ;
\draw[-,line width=1pt]    (180,70) -- (200,50) ;
\draw[-,thin]    (180,90) -- (210,120) ;
\draw[-,thin]    (180,90) -- (180,120) ;
\draw[-,thin]    (180,90) -- (150,120) ;
\node at (140,65) {$\scriptstyle f$};
\node at (140,105) {$\scriptstyle g$};
\end{tikzpicture}=\sum \:\begin{tikzpicture}[anchorbase,scale=0.015,yscale=-1]
\draw[-,thin]    (160,50) -- (210,120) ;
\draw[-,thin]    (160,50) -- (180,120) ;
\draw[-,thin]    (160,50) -- (150,120) ;
\draw[-,thin]    (200,50) -- (210,120) ;
\draw[-,thin]    (200,50) -- (180,120) ;
\draw[-,thin]    (200,50) -- (150,120) ;
\end{tikzpicture},
\]
\[
\begin{tikzpicture}[anchorbase,scale=0.015,yscale=-1]
 \draw[-,line width=1pt,dashed,lightblue] (180,30) to (180,140);
	\draw[-,line width=2pt] (180,70) -- (180,92);
\draw[-,line width=1pt]    (180,70) -- (160,50) ;
\draw[-,line width=1pt]    (180,70) -- (180,50) ;
\draw[-,line width=1pt]    (180,70) -- (200,50) ;
\draw[-,thin]    (180,90) -- (210,120) ;
\draw[-,thin]    (180,90) -- (180,120) ;
\draw[-,thin]    (180,90) -- (150,120) ;
\node at (140,65) {$\scriptstyle f$};
\node at (140,105) {$\scriptstyle g$};
\end{tikzpicture}=\sum \:\begin{tikzpicture}
[anchorbase,scale=0.015,yscale=-1]
\draw[-,line width=1pt,dashed,lightblue] (180,30) to (180,140);
\draw[-,thin]    (160,50) -- (210,120) ;
\draw[-,thin]    (160,50) -- (180,120) ;
\draw[-,thin]    (160,50) -- (150,120) ;
\draw[-,thin]    (180,50) -- (210,120) ;
\draw[-,thin]    (180,50) -- (180,120) ;
\draw[-,thin]    (180,50) -- (150,120) ;
\draw[-,thin]    (200,50) -- (210,120) ;
\draw[-,thin]    (200,50) -- (180,120) ;
\draw[-,thin]    (200,50) -- (150,120) ;
\end{tikzpicture}.
\]
However, this diagram may not be reduced. In particular, if $f$ were a merge then two strings which initially intersected each other exactly once due to a crossing or a split may now intersect twice due to $f$. Furthermore, if $f$ were a split then the newly introduced splits on the right hand side of the latter two equations may come after crossings.

To resolve these issues, note that Equations~\ref{htrivial},~\ref{swallows}, and~\ref{hswallows} correct any strings that intersect each other twice due to a merge and a split or crossing. However, these relations apply in isolation. We also need Equations~\ref{sliders},~\ref{braid},~\ref{hmidsliders},~\ref{hsidesliders}, and ~\ref{hbraid} to get our two intersection points next to each other without interference from other strings. 
As for fixing the order of splits and crossings in the case of $f$ being a split, Equations 
~\ref{sliders},~\ref{hmidsliders}, and~\ref{hsidesliders} let us pull our splits past the crossings as needed.
\end{proof}
With this lemma proven, we can now establish a spanning set of $\Hom_{\HWeb}(\mu, \lambda)$.
\begin{lemma}\label{[A] span}
    The set $\Hom_{\HWeb}(\mu, \lambda)$ has spanning set $\{[A]| A\in \HWeb_{\lambda, \mu}\}$.
\end{lemma}
\begin{proof}
    By Lemma~\ref{reducedpreserved} and the fact that $\HWeb$ can be generated from just merges and splits, $\Hom_{\HWeb}(\mu, \lambda)$ is spanned by the set of $\lambda\times \mu$ reduced chicken foot diagrams. However, by Lemma~\ref{onetype} this set of reduced chicken foot diagrams is itself spanned by $\{[A]| A\in \HWeb_{\lambda, \mu}\}$. The proof follows.
\end{proof}
\subsection{Proof of Equivalence between $\HWeb$ and $\HSchur$}

We can now prove our main theorem, restated below.
\begin{theorem}
\label{catiso}
 There is a an equivalence of categories $\Phi: \HWeb \rightarrow \HSchur$ which is the identity on objects and acts on generating morphisms as
\[
\begin{tikzpicture}[baseline = -.5mm, anchorbase]
	\draw[-,line width=1pt] (0.28,-.3) to (0.08,0.04);
	\draw[-,line width=1pt] (-0.12,-.3) to (0.08,0.04);
	\draw[-,line width=2pt] (0.08,.4) to (0.08,0);
        \node at (-0.22,-.4) {$\scriptstyle a$};
        \node at (0.35,-.4) {$\scriptstyle b$};
\end{tikzpicture} 
\mapsto \xi_{\left(\begin{smallmatrix}a&       b&0&0\\0&0&b&a\end{smallmatrix}\right)},
\begin{tikzpicture}[baseline = -.5mm, anchorbase]
 \draw[-,line width=1pt,dashed,lightblue] (0.08,-.75) to (0.08,0.7);
	\draw[-,line width=1pt] (0.28,-.3) to (0.08,0.04);
	\draw[-,line width=1pt] (-0.12,-.3) to (0.08,0.04);
    \draw[-,line width=1pt] (0.08,-.3) to (0.08,0.04);
    \draw[-,line width=2pt] (0.08,.4) to (0.08,0);
        \node at (-0.22,-.4) {$\scriptstyle a$};
        \node at (0.07,-.4) {$\scriptstyle 2b$};
        \node at (0.35,-.4) {$\scriptstyle a$};
\end{tikzpicture} 
\mapsto \xi_{\left(\begin{smallmatrix}a&2b&a\end{smallmatrix}\right)},
\begin{tikzpicture}[baseline = -.5mm, anchorbase]
	\draw[-,line width=2pt] (0.08,-.3) to (0.08,0.04);
	\draw[-,line width=1pt] (0.28,.4) to (0.08,0);
	\draw[-,line width=1pt] (-0.12,.4) to (0.08,0);
        \node at (-0.22,.5) {$\scriptstyle a$};
        \node at (0.36,.5) {$\scriptstyle b$};
\end{tikzpicture}
\mapsto\xi_{\left(\begin{smallmatrix}a&0\\b&0\\0&b\\0&a\end{smallmatrix}\right)}, 
\]
\[
\begin{tikzpicture}[baseline = -.5mm, anchorbase]
\draw[-,line width=1pt,dashed,lightblue] (0.08,-.7) to (0.08,0.75);
	\draw[-,line width=2pt] (0.08,-.3) to (0.08,0.04);
	\draw[-,line width=1pt] (0.28,.4) to (0.08,0);
	\draw[-,line width=1pt] (-0.12,.4) to (0.08,0);
 \draw[-,line width=1pt] (0.08,.4) to   (0.08,0);
        \node at (-0.22,.5) {$\scriptstyle a$};
        \node at (0.36,.5) {$\scriptstyle a$};
        \node at (0.07,.5) {$\scriptstyle 2b$};
\end{tikzpicture}
\mapsto \xi_{\left(\begin{smallmatrix}a\\ 2b\\   a\end{smallmatrix}\right)}, 
\begin{tikzpicture}[baseline=-.5mm]
	\draw[-,thick] (-0.3,-.3) to (.3,.4);
	\draw[-,thick] (0.3,-.3) to (-.3,.4);
        \node at (0.3,-.4) {$\scriptstyle b$};
        \node at (-0.3,-.4) {$\scriptstyle a$};
\end{tikzpicture}
\mapsto \xi_{\left(\begin{smallmatrix}0 & b&0&0\\a & 0&0&0\\0 & 0&0&a\\0 &0&b&0\end{smallmatrix}\right)}, 
\begin{tikzpicture}[baseline=-.5mm, anchorbase]
  \draw[-,line width=1pt,dashed,lightblue] (0,-.5) to (0,0.6);
	\draw[-,thick] (-0.3,-.3) to (.3,.4);
	\draw[-,thick] (0.3,-.3) to (-.3,.4);
    \draw[-,thick] (0,-.3) to (0,.4);
        \node at (0.3,-.4) {$\scriptstyle a$};
        \node at (0,-.4) {$\scriptstyle 2b$};
        \node at (-0.3,-.4) {$\scriptstyle a$};
\end{tikzpicture}
\mapsto\xi_{\left(\begin{smallmatrix}0 & 0&a\\ 0&2b&0 \\  a &0 &0\end{smallmatrix}\right)}
.\]
\end{theorem}
\begin{proof}
    First of all, by our work in Subsection~\ref{verify} we know that $\Phi$ is well-defined. 

Now consider $A\in \HMat_{\lambda, \mu}$ and observe that $[A] = [A^{+}]\circ[A^{\circ}]\circ [A^{-}]$ for  $A^{+}, A^{\circ}, A^{-}$ defined after Lemma~\ref{splitA} since $[A^{+}]$, $[A^{\circ}]$, and $[A^{-}]$ are the merge, crossing, and split portions of the reduced chicken foot diagram $[A]$. It is also straightforward to see that~$\Phi([A^{+}]) = \xi_{A^{+}}$, $\Phi([A^{\circ}]) = \xi_{A^{\circ}}$, and $\Phi([A^{-}]) = \xi_{A^{-}}$. Therefore by Lemma~\ref{splitxi} we get
\[\Phi([A]) =\Phi([A^{+}])\cdot \Phi([A^{\circ}])\cdot \Phi( [A^{-}]) = \xi_{A^{+}}\cdot \xi_{A^{\circ}}\cdot \xi_{A^{-}} = \xi_{A} .\]
The morphisms $\xi_A$ for $A\in \HMat_{\lambda, \mu}$ are a basis of $\Hom_{\HSchur}(\mu, \lambda)$, while Lemma~\ref{[A] span} tells us that the analogous diagrams $[A]$ span $\Hom_{\HWeb}(\mu, \lambda)$. This means that $\Phi$ is faithful, so combined with the work after Lemma~\ref{splitxi} which shows $\Phi$ to be full, our equivalence of categories follows.
\end{proof}
\section{Appendix. Relations}
In this section we will prove the relations of $\HWeb$, Equations~\ref{switch} to~\ref{hsymmetric}. First of all, relations off the vertical axis, namely Equations~\ref{switch} to~\ref{braid}, follow from the work on the Schur algebra in~\cite[Appendix A]{Brundan}. This is because our generating morphisms and relations off the vertical axis are identical to those of $\Web$. Therefore, we will only prove Equations~\ref{hthickcrossing} to~\ref{hsymmetric} which all lie on the vertical axis. For visual clarity, we will only show this axis at beginning and end of a proof.
\begin{lemma}\label{binom}
    For $s$ a non-negative integer, we have
\[ \sum_{u  =0}^{s}(-1)^u\binom{s}{u}= \begin{cases} 
      1 & s=0 \\
      0 & s>0. 
   \end{cases}.\]
\end{lemma}
\begin{proof}
    The case of $s=0$ is trivial and the case of $s>0$ follows from the Binomial Theorem on the expression $(1-1)^s$.
\end{proof}
\noindent
We will use Lemma~\ref{binom} throughout our proofs in order to reduce summations down to the case where one or more variables equal $0$.
\bigskip
\\*
{\em Proof of~\ref{hthickcrossing}.}
Observe that with Equations~\ref{hsplitchoice} and~\ref{htrivial} we have the following:
\[\sum_{t = 0}^{\min(a, c)}(-1)^t
.\]
With Lemma~\ref{binom} on $t+u$, this summation reduces to
\[ \sum_{v = 0}^{\min(a, b,c,d)}%
 .\]
We now have
\[\sum_{u = 0}^{\min(a, b)}(-1)^u\sum_{t = 0}^{a-u}(-1)^{t}%
 . \]
Using  Equations~\ref{hsplitchoice},~\ref{htrivial}, and~\ref{sliders} we can combine the $u$ and $v$ length strands to get
\[\sum_{u = 0}^{\min(a, b)}\sum_{v = 0}^{\min(a-u, b-u)}(-1)^u\binom{u+v}{u}%
 . \]
With Lemma~\ref{binom} on $u+v$, the only term that remains is with $u+v=0$, proving Equation~\ref{hthickcrossing} as below:
\[   \sum_{u = 0}^{\min(a, b)}\sum_{t = 0}^{a-u}(-1)^{t+u}%
 .\]
\begin{lemma}\label{pullmerge}
    We have the following relation for merges and splits on the axis,
    \[ %
.
\]
Applying Equations~\ref{hsplitchoice},~\ref{htrivial}, and~\ref{swallows} this becomes the following:
\[ \sum_{u = 0}^{\min(a, b)}\sum_{t = 0}^{a-u}(-1)^{t+u}%
.
 \]
We can now apply Lemma~\ref{pullmerge} to continue reducing our diagrams to
 \[\sum_{u = 0}^{\min(a, b)}\sum_{t = 0}^{a-u}(-1)^{t+u}2^{a-u-t}\binom{a+b-u-t}{b}\binom{a}{t}%
.\]
We can rearrange the above coefficient to get the below double summation.
\[\sum_{t = 0}^{a}2^{a-t}(-1)^t\binom{a}{t}\cdot \sum_{u = 0}^{\min(a-t, b)}(-1)^{u}\binom{a-t+(b-u)}{b}\cdot \binom{b}{u}.\]
By~\cite[Lemma A.1]{Brundan}, with $m = a-t$, $r = b-u$ and $s = u$, the inner summation equals $1$ so this simplifies to
\[\sum_{t = 0}^{a}2^{a-t}(-1)^t\binom{a}{t} = (2-1)^a = 1.\]
This proves Equation~\ref{hswallows}.
\bigskip
\\*
{\em Proof of~\ref{hmidsliders}.} First of all, by Equations~\ref{hsplitchoice} and~\ref{hthickcrossing} we have
\[%
.\]
With Equation~\ref{hmergesplit} this transforms into
\[\sum_{u = 0}^{\min(a, b)}\sum_{t = 0}^{a-u}\sum_{x,y}(-1)^{t+u}  %
. \]
We have the following diagram to show the locations of $x,y,o,g$,
\[%
.\]
Now, with Equations~\ref{hsplitchoice},~\ref{htrivial},~\ref{swallows}, and~\ref{sliders} we can combine like-colored strings to get binomial coefficents $\binom{o+u}{u}, \binom{g+t}{t}$. With Lemma~\ref{binom} on $o+u$ and $g+t$, all terms except for those with $o+u = 0$ and $g+t = 0$ vanish. This leaves the below singular term:
\[%
.\]
Below is an expanded diagram to illustrate the locations of $o', g', l', y'$,
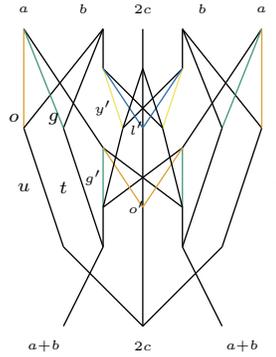
\[ %
.\]
Observe that we can reparameterize $u,t$ into $l, y$ where $o+l = u, g+y = t$ as shown below:
\[ %
.\]
With Equations~\ref{hsplitchoice},~\ref{htrivial},~\ref{swallows},~\ref{sliders}, and~\ref{hmidsliders} we can combine like-colored strands to get binomial coefficients $\binom{o+o'}{o},\binom{g+g'}{g}, \binom{l+l'}{l}, \binom{y+y'}{y}$. Then, with \[(-1)^{t+u} = (-1)^o(-1)^l(-1)^g(-1)^y\] we can apply Lemma~\ref{binom} to each of $o+o', g+g', l+l', y+y'$ so that the only terms that remain are those with $o+o'=g+g'=l+l'=y+y'=0$. This is precisely
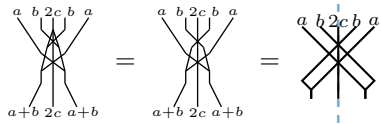
\[ %
 \]
which proves Equation~\ref{hsidesliders}.
\bigskip
\\*
\noindent
{\em Proof of~\ref{hbraid}.} 
With Equations~\ref{sliders},~\ref{hthickcrossing}, and~\ref{hmidsliders} we have the following:
\[%
.
\]
\begin{lemma}\label{hsimplesymmetric}
    We have the following relation for crossings on the axis,
    \[%
.\]
Below is an expanded diagram to demonstrate the locations of $o,g$,
\[%
.\]
With Equations~\ref{hsplitchoice},~\ref{htrivial},~\ref{swallows},~\ref{sliders}, ~\ref{hmidsliders}, and ~\ref{hsidesliders} we can combine like-colored strings to get binomial coefficients $\binom{o+u}{u}, \binom{g+t}{t}$. Then, with Lemma~\ref{binom} on $o+u, g+t$ the only terms that remain are those with $o+u=g+t=0$, namely
\[%
.\]

\end{document}